\documentclass[11pt]{amsart}

\usepackage{amsmath,amssymb,mathtools,xcolor, dsfont, footnote, mathrsfs}
\usepackage{enumerate}

\usepackage[left=2.5cm, right=2.5cm, bottom=2.8cm]{geometry}

\usepackage[backref=page, colorlinks = true, linkcolor = black, citecolor = blue]{hyperref}
\usepackage[nameinlink,noabbrev]{cleveref}

\usepackage{comment}

\usepackage{mathalpha} 
\DeclareFontFamily{U}{txcal}{\skewchar \font =45}
\DeclareFontShape{U}{txcal}{m}{n}{
	<-> txr-cal}{}
\DeclareFontShape{U}{txcal}{b}{n}{
	<-> txb-cal}{}
\DeclareMathAlphabet{\mtxcal}{U}{txcal}{m}{n}
\SetMathAlphabet{\mtxcal}{bold}{U}{txcal}{b}{n}
\DeclareMathAlphabet{\mtxbcal} {U}{txcal}{b}{n}

\DeclareFontFamily{OT1}{mathc}{}
\DeclareFontShape{OT1}{mathc}{m}{it}{<-> mathc10}{}
\DeclareMathAlphabet{\mathcal}{OT1}{mathc}{m}{it}

\newtheorem{theorem}{Theorem}[section]

\newtheorem{corollary}[theorem]{Corollary}
\newtheorem{lemma}[theorem]{Lemma}

\newtheorem{proposition}[theorem]{Proposition}
\newtheorem*{lemma*}{Lemma}

\theoremstyle{definition}
\newtheorem{notation}[theorem]{Notation}
\newtheorem{example}[theorem]{Example}
\newtheorem{definition}[theorem]{Definition}
\newtheorem{question}[theorem]{Question}
\newtheorem*{definition*}{Definition}

\theoremstyle{remark}
\newtheorem{remark}[theorem]{Remark}

\newtheorem{fact}[theorem]{Fact}

\numberwithin{equation}{section}

\newcommand{\R}{\mathbb{R}}

\newcommand{\Hcal}{\mathcal{H}}

\newcommand{\Ccal}{\mathcal{C}}

\newcommand{\Scal}{\mathcal{S}}

\newcommand{\Ncal}{\mathcal{N}}

\newcommand{\bn}{\mathbf{n}}

\newcommand{\bC}{\mathbf{C}}

\newcommand{\Ric}{\operatorname{Ric}}

\newcommand{\obC}{\overline{\bC}}

\newcommand{\mres}{\mathbin{\vrule height 1.6ex depth 0pt width
0.13ex\vrule height 0.13ex depth 0pt width 1.3ex}}

\newcommand{\hbC}{\widehat{\bC}}

\newcommand{\1}{\mathds{1}}

\newcommand{\aX}{\mtxcal{X}}

\newcommand{\aB}{\mathcal{B}}
\newcommand{\vz}{\mathcal{v}}

\newcommand{\VarTan}{\operatorname{VarTan}}

\newcommand{\spt}{\operatorname{spt}}
\newcommand{\dist}{\operatorname{dist}}

\newcommand{\inj}{\operatorname{inj}}

\newcommand{\reg}{\operatorname{reg}}
\newcommand{\sing}{\operatorname{sing}}

\newcommand{\Riem}{\operatorname{Rm}}

\DeclareMathOperator{\divergence}{div}
\NewDocumentCommand{\sff}{}{\mathrm{I\!I}}

\title[Regularity and compactness theory of stable free-boundary minimal hypersurfaces]{Regularity and compactness theory of stable free-boundary minimal hypersurfaces 
}

\author[Costante Bellettini]{Costante Bellettini}
\address{Department of Mathematics, University College London, London WC1E 6BT, United Kingdom}
\email{c.bellettini@ucl.ac.uk}

\author[Martin Li]{Martin Man-chun Li}
\address{Department of Mathematics, The Chinese University of Hong Kong, Shatin, N.T., Hong Kong}
\email{martin.li@cuhk.edu.hk}

\author[Davide Parise]{Davide Parise}
\address{Department of Mathematics, Zeeman Building, University of Warwick, Gibbet Hill Road, Coventry CV4 7AL, UK}
\email{Davide.Parise@warwick.ac.uk}

\author[Lorenzo Sarnataro]{Lorenzo Sarnataro}
\address{Department of Mathematics, University of Toronto, 40 St
George Street, Toronto, ON
M5S 2E4, Canada}
\email{lorenzo.sarnataro@utoronto.ca}

\begin{document}
\begin{abstract}
We study stable free-boundary minimal hypersurfaces in a complete $(n+1)$-dimensio\-nal Riemannian manifold-with-boundary $(M, \partial M)$, and establish the free-boundary version of the well-known (interior) Schoen--Simon regularity and compactness theory from \cite{SchoenSimon81}. Our approach follows the intrinsic PDE scheme introduced by the first-named author in \cite{Bellettini}.

In the key step towards the compactness and regularity, we identify and address the two possible scenarios in which local weak flatness (of a stable free-boundary minimal hypersurface) is verified near a point on $\partial M$: first, weak flatness with respect to $\partial M$; second, weak flatness with respect to a reference hypersurface-with-boundary that meets $\partial M$ almost orthogonally.
In both flatness regimes we prove sheeting theorems, i.e.~a graphical decomposition (possibly in the multi-valued sense) and estimates for the stable minimal hypersurface.

The sheeting theorems are obtained, in their most basic analytic form, as $\varepsilon$-regularity theorems (where smallness of $\varepsilon$ reflects regimes with weak flatness). The starting point is a careful identification of geometrically adapted tilt functions. We choose them so that they quantify flatness and keep track of the ambient geometry, and we show that they solve appropriate non-linear elliptic PDEs on the free-boundary minimal hypersurface. These PDEs are amenable to an intrinsic analysis (of the tilt functions) via iteration \`a la De Giorgi, using stability to obtain weak Caccioppoli-type inequalities that make the iteration viable, and leading to the $\varepsilon$-regularity. 

The two $\varepsilon$-regularity results (and consequent sheeting theorems) are valid for immersed hypersurfaces that are allowed to have a small singular set (a priori of locally finite $(n-2)$-Hausdorff measure). When specialised to embeddings, they yield (in combination with soft classical tools) the compactness and regularity theorem: a sequence of stable free-boundary embedded minimal hypersurfaces with uniformly bounded area (possibly with a singular set of locally finite $(n-2)$-measure) admits a subsequence converging to a free-boundary minimal hypersurface that is smoothly embedded away from a set of Hausdorff dimension at most $n-7$ (its singular set). The convergence is smooth and graphical (possibly with multiplicity) away from this (singular) set.

As an application of the compactness and regularity, we extend the free-boundary Almgren--Pitts existence theory of \cite{Li2021} to arbitrary dimension. In every compact Riemannian manifold-with-boundary of dimension $n+1$, there exists a free-boundary minimal hypersurface, smoothly embedded except possibly for a singular set of Hausdorff dimension at most $n-7$.
\end{abstract}
\maketitle

\tableofcontents

\section{Introduction} \label{sec: introduction} 

Let $(M^{n+1},\partial M,\mathrm g)$ be a smooth Riemannian manifold-with-boundary. A smooth hypersurface-with-boundary $(\Sigma,\partial\Sigma)\subset (M,\partial M)$ is a \emph{free-boundary minimal hypersurface} if its mean curvature vanishes and it meets $\partial M$ orthogonally along $\partial\Sigma$. This characterises stationarity with respect to area for compactly supported deformations that preserve the constraint $\partial \Sigma \subset \partial M$. 
The purpose of this paper is to develop a regularity and compactness theory for stable free-boundary minimal hypersurfaces with a small singular set, and to apply it to the existence of optimally regular\footnote{i.e. smoothly embedded away from a codimension 7 singular set.} free-boundary minimal hypersurfaces in compact Riemannian manifolds-with-boundary, in arbitrary dimension.

The analogous interior problem ($\partial M=\emptyset$) found its resolution in the Schoen--Simon theory, a cornerstone of the field. In \cite{SchoenSimon81}, they proved that a stable minimal hypersurface (without boundary), smoothly and properly embedded outside a closed set of locally finite $\mathcal H^{n-2}$-measure, extends to one that is smooth and embedded away from a singular set of Hausdorff dimension at most $n-7$, and established compactness of this class in the varifold topology under uniform mass bounds. The main efforts in \cite{SchoenSimon81} are devoted to the proof of a sheeting theorem, which, combined with a slicing argument, establishes the aforementioned compactness theorem. Finally, they used these results to obtain the existence of an optimally regular closed minimal hypersurface in every closed Riemannian manifold, extending the low-dimensional Almgren--Pitts theory \cite{Pitts81}, which was based on the Schoen--Simon--Yau curvature estimates in \cite{SSY75}. 

The Schoen--Simon theory also underpins Wickramasekera's influential work \cite{Wick14}, which (among other merits) has been fundamental in the development of an Allen--Cahn min-max approach to the existence of closed minimal hypersurfaces in compact Riemannian manifolds \cite{Hutchinson2000, Tonegawa2012, Guaraco2018}. By analogy, a (future) free-boundary theory in the spirit of \cite{Wick14} (i.e.~for free-boundary stationary stable integral varifolds without classical singularities) would most likely have to rely on the sheeting results that we provide here.  Such developments would further lead to the completion of an Allen--Cahn min-max approach to the existence of free-boundary minimal hypersurfaces, cf.~\cite{LiPariseSarnataro2026}.

The low-dimensional Almgren--Pitts theory was recently extended to the free-boundary setting by the second-named author in collaboration with X. Zhou in \cite{Li2021}.  The missing ingredient for the higher-dimensional theory was a comprehensive free-boundary analogue of Schoen--Simon's compactness for stable hypersurfaces with a singular set of locally finite $\mathcal H^{n-2}$-measure. A partial adaptation of the techniques of \cite{SchoenSimon81} to the free-boundary setting was carried out in \cite{GruterJostIHP86} for surfaces ($n=2$) in Euclidean domains with strictly uniformly convex boundary. Subsequently, \cite{DeLellisRamic} noted that the same arguments carry over in arbitrary dimension and for Riemannian domains with strictly uniformly convex boundaries. A noticeable effect of the convexity assumption in \cite{GruterJostIHP86} and \cite{DeLellisRamic} is that it prevents interior touching of $\Sigma$ with $\partial M$ (by interior touching we mean that $(\Sigma\setminus \partial \Sigma) \cap \partial M \neq \emptyset$). Our results confirm a posteriori that (in the convex case) touching is also ruled out under a limit procedure. Touching behaviour can on the other hand arise when $\partial M$ is not convex, as easy examples show, and may happen in the limit even if we start from a sequence that presents no touching. Handling this possibility (especially in the limit) is crucial for the development of a compactness theory, as well as an existence theory. Indeed, \cite{Li2021} deal with this situation in low dimensions by means of curvature estimates \cite{Guang2020} (in the spirit of Schoen--Simon--Yau \cite{SSY75}), while the compactness theory developed here permits its handling in all dimensions.

An intrinsic PDE approach to the Schoen--Simon theory was recently introduced by the first-named author in~\cite{Bellettini}, where the key $\varepsilon$-regularity result for the classical tilt function (in Euclidean space, with respect to a hyperplane) is obtained by direct (non-linear elliptic) PDE analysis, carried out on the hypersurface. Our proof follows this intrinsic paradigm: rather than linearising the minimal surface equation in an ambient coordinate system (as in \cite{SchoenSimon81, GruterJostIHP86}), we derive differential inequalities directly on the immersed hypersurface for geometrically adapted tilt functions. 
Our tilt functions retain information about the Riemannian geometry of $M$ and of $\partial M$, which turns out to be crucial for the success of the method, and particularly useful at the boundary (see \Cref{sec: intro geometric tilt} below). The $\varepsilon$-regularity result is the key ingredient to obtain a sheeting theorem, from which the compactness theory will follow via soft classical geometric measure theory arguments.

Another advantage of the intrinsic PDE method lies in the fact that (unlike \cite{SchoenSimon81}) it allows us to prove $\varepsilon$-regularity and sheeting theorems for stable \emph{immersed} (rather than embedded) free-boundary minimal hypersurfaces with a sufficiently small (non-immersed) singular set. Both sheeting theorems are new in the case of stable immersions, with the boundary-parallel one new even in the embedded setting. 

\subsection{Main result}
Our main result is the following free-boundary counterpart of the Schoen--Simon compactness theory.

\begin{theorem}[Compactness theory] \label{thm: compactness}
	Let $n\geq 2$ and let $(M^{n + 1}, \partial M, \mathrm{g})$ be a smooth complete Riemannian manifold-with-boundary. Consider a sequence of smooth (up to the boundary) embedded two-sided hypersurfaces-with-boundary $\{(\Sigma_k, \partial \Sigma_k)\}_{k}$. Assume the following, for each $k$. 
	\begin{enumerate}[(i)]
		\item \emph{(a priori control on the singular set)} $\mathcal{S}_k=\overline{\Sigma}_k \setminus \Sigma_k$ has locally finite $\mathcal{H}^{n - 2}$-measure, where the closure is taken in $M$, and the embedding is proper in $M\setminus \mathcal{S}_k$. 
		\item \emph{(free boundary condition)} $\partial \Sigma_k  \subset \partial M$.
		\item \emph{(stationarity)} $\Sigma_k$ is minimal in its interior $\Sigma_k \setminus \partial \Sigma_k$, in the sense that its mean curvature vanishes; for every $p\in \partial \Sigma_k$ the tangent half-plane to $\Sigma_k$ at $p$ is orthogonal to $T_p \partial M$. 
		\item \emph{(stability)} for every $\varphi \in C^{0,1}_c(\Sigma_k)$ the stability inequality holds on $\Sigma_k$, namely
        \begin{equation}\label{eq: stability intro}
		  \int_{\Sigma_k} |\nabla_{\Sigma_k} \varphi|^2 - \int_{\Sigma_k}\big(|A_{\Sigma_k}|^2 + \Ric_M(\nu_{\Sigma_k},\nu_{\Sigma_k})\big) \varphi^2  - \int_{\partial \Sigma_k} \sff_{\partial M}(\nu_{\Sigma_k}, \nu_{\Sigma_k}) \varphi^2\geq 0,
		\end{equation}
		where $\nu_{\Sigma_k}$ is a continuous determination of a unit normal to $\Sigma_k$.
		\item \emph{(mass bounds)} For every compact set $K \subset M$ we have $\sup_k \mathcal{H}^n(\Sigma_k \cap K) < + \infty$. 
	\end{enumerate}
	Then $\Sigma_k$ converges locally subsequentially in $M$ to a smooth embedded hypersurface-with-boundary $\Sigma$ such that $\mathcal{S}=\overline{\Sigma} \setminus \Sigma$ has $\dim_{\mathcal{H}}(\mathcal{S}) \leq n - 7$; in addition, if $n=7$ the set $\mathcal{S}$ is discrete. The convergence is locally smooth and graphical (possibly with multiplicity) on $M\setminus \mathcal{S}$.
Moreover, the hypersurface $\Sigma$ is a free-boundary minimal hypersurface (i.e. it satisfies (ii)-(iii) with $\Sigma$ in place of $\Sigma_k$) and, when taken with multiplicity, $\Sigma$ is the image of a two-sided stable free-boundary minimal immersion. 
\end{theorem}

\begin{remark}[Stability properties of the limit hypersurface $\Sigma$]
More precisely, $\Sigma$ in \Cref{thm: compactness} (while not being necessarily two-sided) satisfies:
	\begin{itemize}
		\item If $\Gamma$ is a two-sided connected component of $\Sigma$, with unit normal $\nu_\Gamma$, then it satisfies the stability inequality \eqref{eq: stability intro} with $\Gamma$ in place of $\Sigma_k$.
		\item If $\Gamma$ is a one-sided connected component of $\Sigma$, then the inclusion $\Gamma\xhookrightarrow{} M$ lifts to a two-to-one immersion $\tilde{\iota}:\widetilde{\Gamma} \to M$ of the orientable double cover $\widetilde{\Gamma}$ in the normal bundle $N\Gamma$, with image $\Gamma$. Then $\widetilde{\Gamma}$ endowed with the pullback metric $\tilde{\iota}^*\mathrm{g}$ has a globally-defined unit normal $\tilde{\nu}_{\widetilde{\Gamma}}$ and the following stability inequality holds
		  \begin{equation*}
		  	\int_{\widetilde{\Gamma}} |\nabla_{\widetilde{\Gamma}} \phi|^2 - \int_{\widetilde{\Gamma}}\big(|A_{\widetilde{\Gamma}}|^2 + \Ric_{M}(\tilde{\nu}_{\widetilde{\Gamma}},\tilde{\nu}_{\widetilde{\Gamma}})\big) \phi^2  - \int_{\partial \widetilde{\Gamma}} \sff_{\partial M}(\tilde{\nu}_{\widetilde{\Gamma}}, \tilde{\nu}_{\widetilde{\Gamma}}) \phi^2\geq 0,
		  \end{equation*}
		  for all $\phi\in C^{0,1}_c(\widetilde{\Gamma})$.
	\end{itemize}
    See \Cref{sec: limit stability} for more details.
\end{remark}

\begin{remark}\label{rmk: compactness theorem varifold}
Convergence aspects of \Cref{thm: compactness} are naturally formulated in the language of \emph{varifolds}. Indeed, under the assumptions of \Cref{thm: compactness}, each $\Sigma_k$ defines a free-boundary stationary integral varifold $|\Sigma_k|\coloneqq \mathbf{v}(\Sigma_k, \1_{\Sigma_k})$. (We show this in \Cref{sec: capacity and extension of stationarity}, as a consequence of the size assumption on $\mathcal{S}_k$, see also \Cref{sec: smallness intro}.) We then prove that, up to a subsequence, the varifolds $|\Sigma_k|$ converge to the varifold $V=\theta|\Sigma|$, where $\Sigma$ is as in \Cref{thm: compactness}, and $\theta$ is a locally constant integer-valued locally summable positive function on $\Sigma$. In particular, the density function $\theta$ is constant on each connected component of $\Sigma$. Moreover, if $\Gamma$ is a one-sided connected component of $\Sigma$, then $\theta$ is necessarily an even positive integer on $\Gamma$.
Finally, $V$ is free-boundary stationary and stable, i.e. for every vector field\footnote{$\mathfrak{X}^{\mathrm{tan}}(M)=\{X\in C^1_c(M, TM): X(x)\in T_x(\partial M) \text{ for all }x\in \partial M\}$.} $X\in\mathfrak{X}^{\mathrm{tan}}(M)$, we have $\delta V(X)=0$ and $\delta^2V(X)\geq 0$. See \Cref{sec: compactness} for details.
\end{remark}
  
\begin{remark}\label{rmk: smoothness boundary}
We emphasise that, throughout this paper, every reference to \emph{smoothness} or \emph{smooth convergence} of hypersurfaces-with-boundary should be interpreted as \emph{smoothness up to the boundary} and \emph{smooth convergence up to the boundary} respectively.
\end{remark}

\begin{remark}\label{rmk: zero mean curvature limit}
Notice that the limit hypersurface $\Sigma$ in \Cref{thm: compactness} may still touch the boundary in its interior, namely it is possible that $(\Sigma\setminus \partial \Sigma)\cap \partial M\neq \emptyset$. However, since the convergence to $\Sigma$ is smooth and graphical, the mean curvature of $\Sigma$ vanishes on $(\Sigma\setminus \partial \Sigma)\cap \partial M$. This constrains the possible behaviours. For example, by the maximum principle, if $\partial M$ is strictly mean-convex (namely $H_{\partial M}>0$), then necessarily $(\Sigma\setminus \partial \Sigma)\cap \partial M= \emptyset$. Similarly, if $U\subset \partial M$ is a relatively open set and $U\subset(\Sigma\setminus \partial \Sigma)\cap \partial M$, then, again by the maximum principle, $H_{\partial M}\equiv 0$ on $U$.

It is also possible that $(\overline{\Sigma}\setminus (\overline{\partial \Sigma} \cup \Sigma))\cap \partial M \neq \emptyset$, namely that $\overline \Sigma$ touches the boundary at an interior singular point (see  \Cref{example: interior branch point}).
\end{remark}

\begin{remark}
Notice that, as in the interior case, the dimension bound $\dim_{\mathcal{H}}(\mathcal{S}) \leq n - 7$ is sharp, as one can see by considering e.g. one half of the Simons cone, namely the area-minimising (and therefore stable) singular free-boundary minimal hypersurface $\Sigma^+=\bigl\{\sum_{j=1}^4 x_j^2=\sum_{j=5}^8 x_j^2;\  x_8\geq 0\bigr\}\subset \R^8_+$, which can be obtained as a smooth (away from its singular set $\mathcal{S}=\{0\}$) limit of a sequence of stable smooth free-boundary minimal hypersurfaces $\Sigma^+_k$. Indeed, it is well-known that the complement $\R^8 \setminus \Sigma$ of the full Simons' cone consists of two connected components $E_+, E_-$, and each of $E_\pm$ is foliated by rescalings $\{\lambda \Gamma_\pm\}_{\lambda>0}$ of a smooth area-minimising hypersurface $\Gamma_\pm$ such that $\lambda \Gamma_\pm \to \Sigma$ as $\lambda\to 0^+$ both in the sense of varifolds, and locally smoothly away from the origin. Since each $\Gamma_\pm$ is smooth and invariant under reflection in $\{x_8=0\}$, then $\Sigma^+_k\coloneqq k^{-1}(\Gamma_+\cap \{x_8\geq 0\})$ is an area-minimising (hence stable) smooth free-boundary minimal hypersurface in $\R^8_+$ and $\Sigma^+_k\to\Sigma^+$ (which has a codimension 7 singular set) as $k\to \infty$.
\end{remark}

\begin{remark}\label{rmk: curvature estimates}
    In dimensions $3\leq n+1\leq 7$, where the singular set is empty, the compactness theory can equivalently be expressed as a curvature estimate. More precisely, our \Cref{thm: compactness} implies \cite[Theorem 1.1]{Guang2020} under the additional assumption that the stable free boundary minimal hypersurface $\Sigma$ is embedded. We note that the original theorem \cite[Theorem 1.1]{Guang2020} can now be strengthened, due to the first-named author's extension of the Schoen--Simon--Yau estimates to ambient dimension $n+1=7$ and to the recent progress on the stable Bernstein problem. In particular, as a consequence of \cite[Theorem 2]{Bellettini}, the curvature estimate of \cite[Theorem 1.1]{Guang2020} holds for \emph{immersed} stable free boundary minimal hypersurfaces in the full dimensional range $3\leq n+1\leq 7$, and, as a consequence of \cite{stableBernstein4,stableBernstein5,stableBernstein6,stableBernstein7}, the constant $C_1$ in \cite[Theorem 1.1]{Guang2020} can be taken to be independent of the area bound $C_0$.
\end{remark}

As announced above, replacing the low-dimensional compactness input in~\cite{Li2021} by \Cref{thm: compactness} gives the following existence theorem, which is new for $n\geq 7$.

\begin{theorem} \label{thm: existence}
    For any given compact Riemannian manifold-with-boundary $(M^{n + 1}, \partial M, \mathrm{g})$ with $n\geq 2$, there exists a smooth embedded hypersurface-with-boundary $\Sigma$ that satisfies the following properties.
\begin{itemize}
    \item $\partial \Sigma  \subset \partial M$.
    \item $\Sigma$ is minimal in its interior $\Sigma \setminus \partial \Sigma$ (in the sense that its mean curvature vanishes) and for every $p\in \partial \Sigma$ the tangent half-plane to $\Sigma$ at $p$ is orthogonal to $T_p \partial M$.
    \item $\overline{\Sigma}\setminus \Sigma$ has Hausdorff dimension at most $n-7$.
\end{itemize}
\end{theorem}

\begin{remark}
Similarly to \Cref{rmk: compactness theorem varifold}, in the language of varifolds, the varifold $V=|\Sigma|$ induced by the free-boundary minimal hypersurface $\Sigma$ in \Cref{thm: existence} is free-boundary stationary, namely its first variation $\delta V$ vanishes along every vector field $X\in \mathfrak{X}^{\mathrm{tan}}(M)$. Moreover, the singular set $\sing V$ of $V$ has Hausdorff dimension at most $n-7$.       
\end{remark}

\begin{remark}\label{rmk: non-empty boundary} 
Note that in \Cref{thm: existence} the set $\spt\Vert V\Vert\cap\partial M$ might be empty, i.e. $V$ might be supported on an optimally regular closed minimal hypersurface contained in the interior of $M$ (see \Cref{subsec: example empty boundary}). However, under additional geometric assumptions on $(M^{n+1},\partial M, \mathrm g)$, for example if $\Ric_{\mathrm g}\geq0$ on $M$ and $\partial M$ is strictly mean-convex, it is possible to guarantee that $V$ is supported on an optimally regular free-boundary minimal hypersurface with non-empty boundary. See \Cref{subsec: Ricci mean convex} for a proof.
\end{remark}

\subsection{Ideas of the proof} 
Let us now explain the overall structure of the argument for \Cref{thm: compactness}, and the various intermediate results of independent interest we obtain en route to the compactness theorem.
As we mentioned at the very beginning of this introductory section, the approach is based on the method developed in \cite{Bellettini}, which relies on intrinsic PDE analysis on immersed minimal hypersurfaces.
The proof of \Cref{thm: compactness} is organised around two sheeting theorems (\Cref{thm: boundary sheeting} and \Cref{thm: orthogonal sheeting}), which characterise stable singular free-boundary hypersurfaces that are $L^2$-close to certain almost-flat models as unions of minimal graphs over these models. Although the compactness theorem (\Cref{thm: compactness}) only deals with stable \emph{embedded} free-boundary minimal hypersurfaces (see \Cref{sec: intro sheeting} for more on this aspect), the sheeting results can be proven with barely any additional effort in the larger class of stable \emph{immersed} free-boundary minimal hypersurfaces (see \Cref{def: sing min hyp} for a precise definition of the class of hypersurfaces we consider).

Our first step in the free-boundary setting is to identify the two almost-flat models. In the first, we use as reference hypersurface a geodesic ball in the boundary $\partial M$ (leading to the \emph{boundary-parallel} sheeting result in \Cref{thm: boundary sheeting}); in the second, the reference hypersurface is a smooth embedded half-disc meeting $\partial M$ orthogonally (leading to the \emph{boundary-orthogonal} sheeting result in \Cref{thm: orthogonal sheeting}).

The analytic architecture is the same in the two cases:
the sheeting theorem is obtained as a consequence of an $\varepsilon$-regularity theorem (\Cref{thm: boundary eps reg} and \Cref{thm: orthogonal eps reg} respectively) for an appropriately chosen scale-invariant tilt excess, which is supposed to measure closeness to the model.
Once the correct geometrically adapted tilt function is identified for each case, the proof of the $\varepsilon$-regularity formally follows the same steps as in \cite{Bellettini}: first, a differential inequality is derived for the tilt function; secondly, by using a truncation of the tilt function as a test function in the stability inequality, an intrinsic weak Caccioppoli inequality on $\Sigma$ is obtained; finally, by using an appropriate Sobolev inequality to relate the $L^2$ energies of successive truncations on nested cylinders, an iteration \`a la De Giorgi can be set up to yield an $L^2$-to-$L^\infty$ estimate for the tilt function at a smaller scale. 
This $L^2$-to-$L^\infty$ estimate under the assumption of initial smallness is the essence of our $\varepsilon$-regularity.
Once the tilt is uniformly small, projection onto the model is locally invertible; so properness and the smallness of the projected singular set promote local graphs to a finite multigraph structure.

In order to highlight the common abstract structure of the arguments, we will present the common ingredients (such as the Sobolev inequalities in \Cref{subsec: Trace-Sobolev} or the iteration lemmas in \Cref{sec: abstract DG}) as stand-alone results in \Cref{sec: preliminaries}.

\subsubsection{Geometric choice of the tilt functions}\label{sec: intro geometric tilt}
Two considerations determine the comparison hypersurfaces and tilt functions used in this argument. The first concerns the metric in which the differential identities are derived. In their boundary curvature estimates, Gr\"uter--Jost~\cite[Section~3]{GruterJostIHP86} flatten the boundary by a local diffeomorphism and treat the transformed area integrand as a perturbation of Euclidean area, following the framework of Schoen--Simon~\cite{SchoenSimon81}. In doing so, the exact variational information is inevitably lost: stationarity and stability conditions become approximate conditions (for the Euclidean area), and the PDE for the tilt will feature differentiation of the Euclidean mean-curvature, a term that admits no natural estimate and is not amenable to the PDE methods we use. Instead, the tilt functions that we use here are naturally induced by the Riemannian metric, and this permits us to preserve all variational information.

The second consideration is compatibility with the boundary geometry, which affects the choices of both the \emph{reference} for flatness, and the \emph{tilt function} to measure flatness (with respect to said reference). This issue persists even when $M$ is the
closure of a smooth domain in $\mathbb{R}^{n+1}$ equipped with the
Euclidean metric, and the $n$-area is already Euclidean. As the proof of the compactness and regularity theorem proceeds via tangent cone analysis, it might be very tempting to aim for sheeting theorems where flatness is captured by the standard Euclidean tilt function with respect to either the hyperplane $T_x \partial M\subset \mathbb{R}^{n+1}$, or a half-hyperplane orthogonal to $T_x \partial M$. In both cases, however, we would run into trouble with the PDE analysis for the tilt, because Euclidean tilt functions introduce terms arising from the free boundary that, in general, are not controlled by the estimates used in our iteration. We consequently adapt both the comparison hypersurface
and the reference normal field to the curved boundary. In the boundary-parallel case, we compare with a local portion of $\partial M$ itself and define the tilt $g$ using its distance function. In the boundary-orthogonal case, we use a comparison half-disc meeting $\partial M$ orthogonally and define the tilt $\tau$ using a smooth unit extension $Z$ of its normal that is tangent to $\partial M$ along the ambient boundary.
These choices yield the boundary
identities needed to eliminate or control the conormal terms in the Caccioppoli inequalities.

\subsubsection{Boundary-parallel sheeting theorem}
Let $d_{\partial M}$ be the distance to the boundary in a collar neighbourhood, and define
\[
 g:=|\nabla_\Sigma d_{\partial M}|.
\]
This will be our tilt function when the comparison hypersurface is chosen to be a geodesic ball in $\partial M$: it measures the tangential component of the boundary-normal direction, and hence the deviation of $T\Sigma$ from the tangent spaces to the parallel hypersurfaces of the collar. A crucial feature is that
\[
 g=1\qquad\text{on }\partial\Sigma, 
\]
while it satisfies a PDE in $\Sigma$ thanks to Bochner's formula. Thus an $L^2$-to-$L^\infty$ estimate which ensures $g<1$ in a smaller domain automatically shows that there are \emph{no} free-boundary points there. The subsequent sheeting is therefore by interior minimal graphs over $\partial M$. In particular, boundary-parallel $L^2$ closeness rules out a degeneration in which boundary components collapse into a boundary-parallel limit through flatter and flatter half-catenoidal necks.

\begin{theorem}[Boundary-parallel $\varepsilon$-regularity for the tilt] \label{thm: boundary eps reg}
Let $M$ be as above, $p\in \partial M$, and
let $P$ be a geodesic ball in $\partial M$ centred at $p$. Assume $\rho_0>0$ is such that $P$ has a $\rho_0$-collar neighbourhood $\Ncal^+_{\rho_0}(P)$ in $M$, and 
\begin{equation}\label{eq: geometric control on neighbourhood}
|\nabla_\mathrm{g}^2d_{\partial M}|^2+|\nabla^3_{\mathrm{g}}d_{\partial M}|+|\operatorname{Rm}_{\mathrm{g}}| \leq \Lambda^2\quad \text{ in $\Ncal^+_{\rho_0}(P)$,}
\end{equation}
where $d_{\partial M}$ denotes the distance function from $\partial M$.
There exist $\varepsilon_0>0, R_0>0$ depending only on the dimension $n$ and the geometric constant $\Lambda$ such that the following holds.\\
Let $\delta \in(0, \rho_0/2)$, $R\in(0, R_0)$, $\sigma\in(0,1)$, and let $P_R = B_R^{\partial M}(p)\subset P$. Let $\Sigma$ be a stable singular properly immersed free-boundary minimal hypersurface in $\Ncal^+_{\rho_0}(P_R)$ such that $\overline{\Sigma} \subset \Ncal^+_\delta(P_R)$ and $\partial \Sigma \subset P_R$. If  
\begin{equation*}
\mathrm{E}_{\Sigma, P}(R)^2\coloneqq R^{-n}\int_{\Sigma} |\nabla_\Sigma d_{\partial M}|^2 <(1-\sigma)^n\varepsilon^2 \quad  \text{for some $\varepsilon<\varepsilon_0$},     
\end{equation*}
then
\begin{equation*}
    |\nabla_\Sigma d_{\partial M}|<C_0(1-\sigma)^{-\frac{n}{n+2}}\,\mathrm{E}_{\Sigma, P}(R)^\frac{2}{n+2}, \quad \text{on $\Sigma \cap \Ncal^+_{\rho_0}(P_{\sigma R})$}, 
\end{equation*}
where the constant $C_0 = C_0(n, \Lambda)$ only depends on the dimension $n$ and the constant $\Lambda$.
\end{theorem}

\begin{theorem}[Boundary-parallel sheeting] \label{thm: boundary sheeting}
Under the assumptions of \Cref{thm: boundary eps reg}, we further conclude that $\partial \Sigma \cap P_{R/2} = \emptyset$, and that there exist an integer $Q \geq 1$ and ordered Lipschitz functions
\[
        0\leq u_1\leq u_2\leq \cdots \leq u_Q, \qquad u_j \colon P_{R/2} \to [0, \delta),
\]
such that
\[
        \overline{\Sigma} \cap \Ncal^+_{\rho_0}(P_{R/2}) = \bigcup_{j=1}^Q F_{\partial M}\bigl(\operatorname{graph}(u_j)\bigr),
\]
where $F_{\partial M}\bigl(\operatorname{graph}(u_j)\bigr)\coloneqq \{F_{\partial M}(y,u_j(y)):y \in P_{R/2}\}$ and $F_{\partial M} \colon P \times [0, \rho_0) \rightarrow M$ is the normal exponential map of $\partial M$ (as defined in \Cref{sec: preliminaries}) restricted to $P$. Furthermore,
\[
        \operatorname{Lip}_P(u_j)\leq C(n, \Lambda)\mathrm{E}_{\Sigma, P}(R)^\frac{2}{n+2}, 
\]
for every $j=1,\ldots, Q$, where the constant $C(n, \Lambda)$ only depends on the dimension $n$ and the geometric constant $\Lambda$ in \eqref{eq: geometric control on neighbourhood}. Moreover,
\begin{enumerate}[(i)]
    \item If we further assume $\mathcal{H}^{n - 2}(\mathcal{S}) = 0$, then there exist smooth functions $v_1, v_2, \ldots, v_Q \colon P_{R/2} \rightarrow [0, \delta)$ such that 
    \[
     \overline{\Sigma} \cap \Ncal^+_{\rho_0}(P_{R/2}) = \bigcup_{j=1}^Q F_{\partial M}\bigl(\operatorname{graph}(v_j)\bigr), 
    \]
    and with 
    \[
    \Vert \nabla^2 v_j \Vert_{C^0(P_{R/2})}+R^\alpha[\nabla^2v_j]_{\alpha, P_{R/2}} \leq C(n, \alpha, \Lambda)\left(R^{-1}\mathrm{E}_{\Sigma, P}(R)^\frac{2}{n+2}+1\right)
    \]
    for all $j = 1, 2, \ldots, Q$ and $\alpha\in(0,1)$. In particular, $R\vert A_\Sigma \vert \leq C(n, \Lambda)\left(\mathrm{E}_{\Sigma, P}(R)^{\frac{2}{n+2}}+ R\right)$ on $\Sigma \cap \Ncal^+_{\rho_0}(P_{R/2})$. 
    \item If, on the other hand, we assume that $\Sigma$ is embedded, then $\overline{\Sigma} \cap \Ncal^+_{\rho_0}(P_{R/2}) = \Sigma \cap \Ncal^+_{\rho_0}(P_{R/2})$ and $0 \leq u_1 < u_2 < \ldots < u_Q$ are smooth with estimates
    \[
    \Vert \nabla^2 u_j \Vert_{C^0(P_{R/2})}+R^\alpha[\nabla^2u_j]_{\alpha, P_{R/2}} \leq C(n, \alpha, \Lambda)\left(R^{-1}\mathrm{E}_{\Sigma, P}(R)^\frac{2}{n+2}+1\right)
    \]
    for all $j = 1, 2, \ldots, Q$ and $\alpha \in(0,1)$. Lastly, $R\vert A_\Sigma \vert \leq C(n, \Lambda)\left(\mathrm{E}_{\Sigma, P}(R)^{\frac{2}{n+2}}+ R\right)$ on $\Sigma \cap \Ncal^+_{\rho_0}(P_{R/2})$. 
\end{enumerate}
\end{theorem}

\subsubsection{Boundary-orthogonal sheeting theorem}
The second model is a smooth two-sided $n$-dimensional half-disc $P$ meeting $\partial M$ orthogonally. We need to define a tilt function that quantifies deviation from this model. Directly extending the unit normal to $P$ as the gradient of $d_P$ is not suitable for the boundary calculation: the normal derivative of the associated tilt would contain the uncontrolled component $A(\eta,\eta)$, where $\eta$ is the free-boundary conormal. We instead construct a unit vector field $Z$ on a neighbourhood of $P$ in $M$ such that
\[
 Z|_P=\nu_P,
 \qquad Z\in T(\partial M)\quad\text{along }\partial M.
\]
Its flow preserves $\partial M$, and we measure tilt by
\[
 \tau:=|Z^T|,
\]
where $Z^T$ is the projection of $Z$ onto $T\Sigma$. Tangentiality of $Z$ converts the boundary derivative of $\tau$ into terms involving only the second fundamental form of $\partial M$ and derivatives of $Z$ (see \Cref{lemma: partial tau boundary}), so the boundary contribution can be handled by the trace inequality. 

\begin{remark}
Our boundary-orthogonal sheeting theorem includes, as a special case, the result in \cite[Section 3]{GruterJostIHP86}. Our construction of $Z$ replaces, in some sense, the flattening of $\partial M$ performed in \cite{GruterJostIHP86}. While $M$ is therein locally described in coordinates as a half ball $B^{n+1}_+(0)$, with $\partial M$ identified with $B^n(0) \times \{0\}$, we use (flow lines of) $Z$ and $\nabla^\mathrm{g} d_{\partial M}$ to reach a neighbourhood of $P$ in $M$ (adapted neighbourhood, using our terminology in \Cref{sec: proof orthogonal sheeting}), noting that $Z$ preserves $\partial M$ and $\nabla^\mathrm{g} d_{\partial M}$ is orthogonal to it. Crucially, we keep the ambient metric and the variational information, while \cite{GruterJostIHP86} work with the Euclidean metric in $B^{n+1}_+(0)$ (following \cite{SchoenSimon81}). That choice would not work for our method, since the Euclidean mean curvature of the transformed hypersurface becomes nonzero and its gradient inevitably appears in the estimates, producing error terms that are not controlled by the available hypotheses. On the other hand, the tilt function $\tau$, being built from $Z$, reflects the Riemannian geometry of the ambient boundary; moreover, with the ambient metric and variational hypotheses left unaltered, $\tau$ satisfies a PDE that allows us to successfully carry out the approach of \cite{Bellettini} in our setting. We note that $Z$ and $\nabla^\mathrm{g} d_{\partial M}$ play a similar role, in flavour, to the construction of a so-called orthogonal foliation in \cite{Li-Zhou-max}, used in the context of a maximum principle for free-boundary minimal submanifolds.
\end{remark}

More precisely, let $P$ be a smooth two-sided embedded $n$-dimensional half-disc, let $\partial_0P=P\cap\partial M$, and assume that $P$ meets $\partial M$ orthogonally along $\partial_0P$. Such a $P$ can be constructed locally, for instance by taking the half-disk in Fermi coordinates centred at $p \in \partial M$. See \Cref{sec: setup FB} for further details. If $p\in\partial_0P$, write $P_R=B_R^P(p)$. The adapted field $Z$ has a flow map $F_Z$, and we set
\[
 N_\delta^Z(P_R):=F_Z\bigl(P_R\times(-\delta,\delta)\bigr).
\]

\begin{theorem}[Boundary-orthogonal $\varepsilon$-regularity for the tilt] \label{thm: orthogonal eps reg}
Let $M$ be as above and let $P$ be a smooth two-sided, embedded, $n$-dimensional half-disc in $M$ meeting $\partial M$ orthogonally along $\partial_0P$. Let the vector field $Z$ be as above, and let $p\in \partial_0P=P\cap \partial M$. Assume $\rho_1>0$ is such that the flow map $F_Z$ is a diffeomorphism from $P_{\rho_1}\times(-\rho_1, \rho_1)$ onto the relatively open neighbourhood  $\Ncal^Z_{\rho_1}(P)$ of $P$ in $M$. Assume that
\begin{equation}\label{eq: geometric control on neighbourhood orthogonal}
	\begin{aligned}
		&\sup_{ P_{\rho_1}}|A_P|^2+\sup_{P_{\rho_1}}|\nabla_P A_P|+\sup_{ P_{\rho_1}}|\nabla^2_PA_P|^{2/3}\leq \Lambda^2,\\  
		&\sup|\nabla^2_\mathrm{g}d_{\partial M}|^2+\sup|\nabla^3_\mathrm{g}d_{\partial M}|+\sup|\nabla^4_\mathrm{g}d_{\partial M}|^{2/3}\leq \Lambda^2,\\
		&\sup|\Riem_\mathrm{g}|+\sup|\nabla^{\mathrm{g}}\Riem_\mathrm{g}|^{2/3}+\sup|\nabla^2_{\mathrm{g}}\Riem_\mathrm{g}|^{1/2}\leq \Lambda^2,
	\end{aligned}
\end{equation}
in $\Ncal^Z_{\rho_1}(P)$, where $d_{\partial M}$ denotes the distance to $\partial M$ in $M$, and $A_P$ denotes the second fundamental form of $P$ in $M$. 
Then there exist $\varepsilon_1>0$,  $ \delta_1 \in (0, \rho_1)$ and $R_1>0$ depending only on the dimension $n$ and the geometric constant $\Lambda$ such that the following holds.\\
Let $\delta \in(0, \delta_1)$, $R\in(0, R_1)$, $\sigma\in(0,1)$. Assume $\Sigma$ is a stable singular properly immersed free-boundary minimal hypersurface in $M$ such that $\overline{\Sigma} \subset \Ncal^Z_\delta(P_R)$ and $\partial \Sigma \subset (\partial M \cap \Ncal^Z_{\rho_1}(P))$. If we let $Z^T$ denote the orthogonal projection of $Z$ onto $T\Sigma$, we have that if
\begin{equation*}
	\mathrm{E}^Z_{\Sigma, P}(R)^2\coloneqq R^{-n}\int_{\Sigma} |Z^T|^2 < (1-\sigma)^n\varepsilon^2 \quad \text{for some $\varepsilon<\varepsilon_1$,}
\end{equation*}
then
\begin{equation}\label{eq: nonlinear Z-tilt estimate}
	|Z^T|\leq C_1(n, \Lambda)(1-\sigma)^{-\frac{n}{n+2}}\,\mathrm{E}^Z_{\Sigma, P}(R)^{\frac{2}{n+2}} \quad \text{on $\Sigma \cap \Ncal^Z_{\rho_1}(P_{\sigma R})$},  
\end{equation}
where the constant $C_1 = C_1(n, \Lambda)$ only depends on the dimension $n$ and the constant $\Lambda$. 

\end{theorem}

\begin{theorem}[Boundary-orthogonal sheeting] \label{thm: orthogonal sheeting}
    Under the assumptions of \Cref{thm: orthogonal eps reg}, there exist an integer $Q \geq 1$ and ordered Lipschitz functions 
    \begin{equation*}
        u^Z_1 \leq u^Z_2 \leq \ldots \leq u^Z_Q, \qquad u^Z_j \colon P_{R/2} \rightarrow (-\delta, \delta), 
    \end{equation*}
    such that 
    \begin{equation*}
         \overline{\Sigma} \cap \Ncal^Z_{\rho_1}(P_{R/2}) = \bigcup_{j=1}^Q F_Z\bigl(\operatorname{graph}(u^Z_j)\bigr),
    \end{equation*}
    where $F_Z(\operatorname{graph}(u^Z_j)) = \{F_Z(y, u^Z_j(y)) : y \in P_{R/2} \}$ and $F_Z \colon P_{R/2} \times (-\delta, \delta) \rightarrow M$ is the flow map of the vector field $Z$. Furthermore, 
    \begin{equation}\label{eq: nonlinear Lip estimate}
    \operatorname{Lip}_P(u^Z_j)\leq C(n, \Lambda)\left(\delta+\mathrm{E}^Z_{\Sigma, P}(R)^\frac{2}{n+2}\right), 
    \end{equation}
    for every $j=1,\ldots, Q$, where the constant $C(n, \Lambda)$ only depends on the dimension $n$ and the geometric constant $\Lambda$ in \eqref{eq: geometric control on neighbourhood orthogonal}. Moreover, 
    \begin{enumerate}[(i)]
   	\item If we further assume $\mathcal{H}^{n - 2}(\mathcal{S}) = 0$, then there exist smooth functions $v^Z_1, v^Z_2, \ldots, v^Z_Q \colon P_{R/2} \rightarrow (-\delta, \delta)$ such that 
   \[
   \overline{\Sigma} \cap \Ncal^Z_{\rho_1}(P_{R/2}) = \bigcup_{j=1}^Q F_{Z}\bigl(\operatorname{graph}(v^Z_j)\bigr), 
   \]
   and with 
   \[
   \Vert \nabla^2 v^Z_j \Vert_{C^0(P_{R/2})}+R^\alpha[\nabla^2v^Z_j]_{\alpha, P_{R/2}} \leq C(n, \alpha, \Lambda)\left(R^{-1}\left(\mathrm{E}^Z_{\Sigma, P}(R)^\frac{2}{n+2}+ \delta\right)+1\right)
   \]
   for all $j = 1, 2, \ldots, Q$ and $\alpha\in(0,1)$. In particular, 
   \[
    R\vert A_\Sigma \vert \leq C(n, \Lambda)\left( \delta+ R + \mathrm{E}^Z_{\Sigma, P}(R)^{\frac{2}{n+2}}\right)
   \]
   on $\Sigma \cap \Ncal^Z_{\rho_1}(P_{R/2})$. 
   \item If, on the other hand, we assume that $\Sigma$ is embedded, then $\overline{\Sigma} \cap \Ncal_{\rho_1}(P_{R/2}) = \Sigma \cap \Ncal_{\rho_1}(P_{R/2})$, $u^Z_1, \ldots u^Z_Q$ are smooth and
   \[
    u^Z_1<u^Z_2<\dots<u^Z_Q \quad \text{on $P_{R/2}$}.
   \]
   Moreover, the $u^Z_j$ satisfy estimates
   \[
   \Vert \nabla^2 u^Z_j \Vert_{C^0(P_{R/2})}+R^\alpha[\nabla^2u^Z_j]_{\alpha, P_{R/2}} \leq C(n, \alpha, \Lambda)\left(R^{-1}\left(\mathrm{E}^Z_{\Sigma, P}(R)^\frac{2}{n+2}+ \delta\right)+1\right)
   \]
   for all $j = 1, 2, \ldots, Q$ and $\alpha \in(0,1)$. Lastly, $R\vert A_\Sigma \vert \leq C(n, \Lambda)\left(\delta+ R+\mathrm{E}^Z_{\Sigma, P}(R)^{\frac{2}{n+2}}\right)$ on $\Sigma \cap \Ncal^Z_{\rho_1}(P_{R/2})$. 
   \end{enumerate}
\end{theorem}

\subsubsection{Examples}
We now construct two examples showing that, in the immersed setting, when $\mathcal{H}^{n-2}(\Scal)>0$ genuine singular points of branch type may be present both in $\overline{\Sigma} \setminus \overline{\partial \Sigma}$ (and may possibly lie on $\partial M$) and in $\overline{\partial \Sigma}$. By abuse of terminology, we will refer to these two situations as interior branch points of $\Sigma$, and boundary branch points of $\Sigma$, respectively. Neither stability nor the existence of a flat tangent cone implies regularity at a point of contact with $\partial M$. \Cref{example: interior branch point} exhibits an interior branch point lying on $\partial M$, at which the surface is tangent to $\partial M$ and has a multiplicity-two planar tangent cone. The product construction in \Cref{example: boundary branch point} produces a genuine free-boundary stable minimal immersion whose codimension-two branch set reaches the boundary and whose tangent cone there is a multiplicity-two half-hyperplane meeting $\partial M$ orthogonally. Thus, both types of flat cone appearing in our boundary sheeting theory can occur at branch points and therefore cannot, for immersed hypersurfaces, yield regularity. On the other hand, the sheeting theorems imply the uniqueness of the flat tangent cone, as recorded in \Cref{cor: uniqueness of tangent hyperplanes}. 
\begin{example}[Interior branch point touching the boundary] \label{example: interior branch point}
    We will now exhibit a branched minimal immersion with an interior branch point touching the boundary $\partial M$. Consider $x_{3} = u_\pm(x)$, where $x = (x_1, x_2)$, satisfying $\vert u(x) \vert \leq C \vert x \vert^\alpha$. Let $b_R(x) = - R + \sqrt{R^2 - \vert x \vert^2} = - \frac{\vert x \vert^2}{2R} + O(\vert x \vert^4)$ be the upper half-sphere centred at $-Re_3$. If $\alpha > 2$, then $$ u_\pm(x) \geq - c \vert x \vert^\alpha > - \frac{\vert x \vert^2}{2R} + O(\vert x \vert^4) = b_R(x),$$ for $\vert x \vert$ sufficiently small. In particular, $u = \pm c \Re(z^{5/2})$ satisfies this property, but it is harmonic, not minimal. We will use the Weierstrass--Enneper parametrization to circumvent this. Consider $\Phi(\zeta) = (\zeta(1 - \zeta^6), \mathrm{i}\zeta(1 + \zeta^6), 2 \zeta^4),$ and note that $\Phi \cdot \Phi \equiv 0$. In particular,
    \[
    X(\zeta) = \Re \int \Phi(\zeta) \, d\zeta = \Re \left(\frac{\zeta^2}{2} - \frac{\zeta^8}{8}, \mathrm{i} \frac{\zeta^2}{2} + \mathrm{i} \frac{\zeta^8}{8}, \frac{2\zeta^5}{5} \right)  \in \mathbb{R}^3. 
    \]
    The Weierstrass--Enneper theorem (cf. \cite[p.~33]{Colding2011}) shows that $X$ is a branched minimal immersion, and that $\zeta = 0$ is an isolated branch point. Furthermore, writing $X = (X^\prime, X_3)$ with $X^\prime \in \mathbb{R}^2$, we have 
    \[
    X^\prime(\zeta) = \frac{1}{2}(\Re(\zeta^2), -\Im(\zeta^2)) + O(\vert \zeta \vert^8) \qquad \text{and} \qquad X_3(\zeta) = \frac{2}{5}\Re(\zeta^5),
    \]
    so that $\vert X^\prime(\zeta) \vert \simeq \vert \zeta \vert^2$, while $\vert X_3(\zeta) \vert \leq C \vert \zeta \vert^5 \leq C \vert X^\prime(\zeta) \vert^{5/2},$ so that $X_3 \simeq \pm c \Re((x_1 + \mathrm{i} x_2)^{5/2})$ as required. Now, for small enough $\zeta$, there holds 
    \[
    \vert X(\zeta) + Re_3 \vert^2 - R^2 = \vert X^\prime(\zeta) \vert^2 + X_3(\zeta)^2 + 2R X_3(\zeta) \geq c \vert \zeta \vert^4 - C_R \vert \zeta \vert^5 > 0, 
    \]
    with equality holding at $\zeta = 0$. This is a smooth stable immersion $f:S \coloneqq D\setminus \{0\} \to \mathbb{R}^3$, where $D$ is a sufficiently small disc centred at the origin in $\mathbb{R}^2=\mathbb{C}$. Stability follows from Barta's theorem (cf. \cite[Lemma 1.36]{Colding2011}). More precisely, if $\nu(\zeta)$ denotes the upward pointing unit normal, then $\langle \nu(\zeta), e_3 \rangle$ gives a positive Jacobi field. Take then as manifold-with-boundary $\mathbb{R}^3 \setminus B_R(-Re_3)$. We refer to \cite{SimonWicExa} for more examples of stable branched minimal immersions. 
\end{example}

\begin{example}[Boundary branch point]\label{example: boundary branch point}
Let $f\colon S \rightarrow \mathbb{R}^3$ be as in \Cref{example: interior branch point}, so that in particular $f$ is a two-sided stable branched minimal immersion, with $\partial\Sigma=\emptyset$ and discrete branch set $\mathcal B$. Set $\Sigma=f(S)$, $N=\mathbb{R}^3$, and $F:S\times \mathbb{R}_+ \to \mathbb{R}^4_+$, as well as
\[
\widehat\Sigma\coloneqq\Sigma \times\mathbb R_+,\qquad \widehat M\coloneqq N\times\mathbb R_+ = \mathbb{R}_+^4,\qquad F(x,y,s)\coloneqq(f(x),s),
\]
where $\mathbb R_+=[0,\infty)$ and $\widehat g=g+ds^2$. Then $\partial\widehat\Sigma=\Sigma\times\{0\}$ and $ F(\partial\widehat\Sigma)\subset\partial\widehat M.$ For $X\in T_x\Sigma$, and $a\in\mathbb R$, $dF(X,a)=(df(X),a)$ so that $\operatorname{rank}dF=\operatorname{rank}df + 1.$ Hence
\[
\widehat\Sigma_{\mathrm{reg}} =(\Sigma\setminus\mathcal B)\times\mathbb{R}_+, \qquad \mathcal B_F=\mathcal B\times\mathbb R_+.
\]
Thus $\dim\mathcal B_F = 1$, so $\mathcal B_F$ has codimension $2$, and its boundary part is $\mathcal B\times\{0\}$. On the parameter manifold $S\times \mathbb{R}_+$, $F^*\widehat g=f^*g+ds^2.$ If $\nu$ is a unit normal to $f$, then $\widehat\nu=(\nu,0)$ is a unit normal to $F$, and the product connection gives $A_F\bigl((X,a),(Y,b)\bigr)=A_f(X,Y).$ All mixed and Euclidean-factor terms vanish; hence $\operatorname{tr}A_F=\operatorname{tr}A_f=0,$ so that $F$ is minimal. Along $\partial\widehat\Sigma_{\mathrm{reg}}$, the outward conormal of $F$ and the outward normal to $\partial\widehat M$ are both $-\partial_s$. Thus $F$ is free-boundary, see also \Cref{L:reflection stationary}. Since $\partial\widehat M$ is totally geodesic, $|A_F|^2=|A_f|^2,$ and $\operatorname{Ric}_{\widehat M}(\widehat\nu,\widehat\nu) =\operatorname{Ric}_{N}(\nu,\nu).$ Therefore, for $\varphi\in C_c^\infty(\widehat\Sigma_{\mathrm{reg}})$, writing $\varphi_{s}(x) \coloneqq \varphi(x,s)$,
\[
Q_F(\varphi) = \int_{\mathbb R_+} Q_f(\varphi_{s}) \, ds +\int_{\widehat\Sigma_{\mathrm{reg}}} |\partial_s\varphi|^2 \,d\mu_F \geq0.
\]
By splitting a variation into odd and even components with respect to reflection in the totally geodesic boundary, we see that $F$ is free-boundary stable. Stability extends across $\mathcal B_F$ by logarithmic cutoff. Lastly, if the tangent varifold of $f$ at $p$ is $Q|P|$, then that of $F$ at $F(p,0)$ is $Q\vert P\times\mathbb R_+\vert,$ a multiplicity-$Q$ half-hyperplane meeting $\partial\widehat M$ orthogonally. 
\end{example}

\begin{corollary}[Uniqueness of flat tangent cones] \label{cor: uniqueness of tangent hyperplanes}
Let $(\Sigma, \partial \Sigma)$ be as in \Cref{def: stable immersion}, and let $x \in \overline{\Sigma}$ be such that a tangent cone (to the stationary varifold $|\Sigma|$ at $x$) is either a half-hyperplane orthogonal to $T_x\partial M$, or the hyperplane $T_x \partial M$ (in both cases, a constant integer multiplicity may be present). Then the tangent cone is unique. \end{corollary}

\subsubsection{From the sheeting theorems to the compactness theory}\label{sec: intro sheeting}
We finally indicate how the local estimates yield \Cref{thm: compactness}. The local mass bounds, together with the corresponding local first-variation bounds, give subsequential varifold compactness. In the present class, the required support convergence follows directly from a uniform lower density estimate obtained by combining the usual interior monotonicity formula with the boundary-centred monotonicity formula of~\cite{Guang2020}. In particular, no reference to the (computationally heavy) monotonicity of reflected density is needed.

Let $V$ be a limit varifold and consider a tangent cone at a boundary point. After doubling in the tangent half-space, a cone with an $n$-dimensional spine is flat and falls into one of the two models above: either a multiple of $T_x\partial M$, or a multiple of an orthogonal half-hyperplane. Varifold convergence gives height closeness to the relevant model. The height-to-tilt inequalities of \Cref{sec: height-tilt} turn this into small $L^2$ tilt, and \Cref{thm: boundary sheeting,thm: orthogonal sheeting} then give smooth graphical convergence at a fixed scale. This proves regularity at every flat tangent cone and, at the same time, uniqueness of that cone.

Up to here, no distinction between embeddings and immersions is needed. The next step, on the other hand, uses embeddedness in an essential way. The Schoen--Simon slicing argument, adapted to the boundary, excludes stationary classical cones with an $(n-1)$-dimensional spine. For immersed hypersurfaces, sheets may branch or pass through one another, as \Cref{example: interior branch point,example: boundary branch point} illustrate, and this exclusion mechanism is no longer available in the same form 
(see \Cref{q: compactness immersed} for recent progress).

Once the $(n-1)$-dimensional stratum is excluded, the standard dimension-reduction argument for stable cones gives $\dim_{\Hcal}\sing V\leq n-7$, by adapting the usual argument to the case of free-boundary stationary varifolds. At this stage as well we do not need to introduce a reflected density or reflected monotonicity formula: the ordinary interior monotonicity formula and the boundary monotonicity formula of~\cite{Guang2020} suffice. Reflection is used only at the tangent-cone level, where it becomes straightforward, as the manifold-with-boundary has become a half-space and the boundary has become a Euclidean hyperplane.

\subsubsection{On the smallness assumption and the stationarity condition}
\label{sec: smallness intro}
Our assumptions on the singular set are also novel, and less restrictive than in other previous works on stable immersed free-boundary minimal hypersurfaces, such as \cite{WicM2}. Indeed, we only require the set of non-immersed singular points to have locally finite $\Hcal^{n-2}$ measure, and only impose stationarity away from this set. In order to be able to use the consequences of global stationarity, such as monotonicity of area ratios, we then show that stationarity extends across sets of locally finite $\Hcal^{n-2}$ measure by constructing a function that blows up on the exceptional set and is superharmonic along every $n$-dimensional tangent plane (this part is inspired by a construction of Harvey--Lawson in \cite{HarveyLawson2014}); a logarithmic Caccioppoli estimate then yields cutoff functions of vanishing tangential energy. See \Cref{sec: capacity and extension of stationarity} for details.

\subsection{Open questions and future directions} We finish this introduction by mentioning some questions and open problems that arose during the preparation of this work. 

\begin{question}\label{q: compactness immersed}
Is there a compactness theory for stable singular properly immersed free-boundary minimal hypersurfaces with a singular set of locally finite $\Hcal^{n-2}$-measure? This seems likely given the present work and the recent substantial advances in the interior theory in $\mathbb{R}^{n+1}$: in \cite{GWangimmersed}, Wang--Zhang  produce a tilt function that allows them to extend the methods and conclusions of \cite{Bellettini}, leading to a sharp $(n-3)$-dimensional bound for the non-branch singular set, and establish a branched sheeting theorem near cones that are unions of hyperplanes with a spine of dimension $n-1$ and $n-2$. 
At the free boundary one must additionally understand boundary branch points and their interaction with the touching set.
\end{question}    


\begin{question} \label{question: orthogonal sheeting refinement}
Can the boundary-orthogonal $\varepsilon$-regularity and sheeting theorems (\Cref{thm: orthogonal eps reg,thm: orthogonal sheeting}) be proven under the weaker notion of \emph{stability away from the touching set} used in~\cite[Definition~2.2]{GuangWangZhouFiniteness}? The present proof uses the analytic stability inequality for all compactly supported scalar test functions, including across the touching set, and it is not immediate how to recover the same Caccioppoli estimate from the weaker hypothesis.

Notice that, on the other hand, the tilt function used in the boundary-parallel case naturally vanishes on the touching set, and it is an admissible test function for the stability inequality under the weaker notion of stability away from the touching set.
\end{question}



\begin{question}
    Can the sheeting and compactness theorems above be generalised to the capillary setting? This appears to be promising, in light of the present work and of very recent or simultaneous progress on related aspects. A theory for stable minimal hypersurfaces-with-boundary with a given capillary angle in a Euclidean half-space is developed in \cite{GWangcapillary}. Moreover, in the forthcoming work \cite{BellettiniJiang26}, a regularity and compactness theory is established, for stable hypersurfaces whose mean curvature is prescribed by an ambient function on a Riemannian manifold. A combination of the three works is likely to lead to a compactness and regularity theory for capillary hypersurfaces in the presence of a wetting energy and a potential energy.    
\end{question}

\begin{question}
Let $V$ be an $n$-dimensional rectifiable varifold which is stationary away from a relatively closed set $S$ satisfying $\Hcal^{n-1}(S)=0$. Must $V$ be stationary across $S$? \Cref{prop: locally finite then global stationarity} gives a positive answer when $S$ has locally finite $\Hcal^{n-2}$-measure. Its proof is critical at that exponent and the construction used there no longer has the required properties under the weaker assumption $\Hcal^{n-1}(S)=0$.
\end{question}

\begin{question}
What structural hypothesis gives an optimal regularity and compactness theory for free-boundary stationary integral varifolds whose regular part is stable? The desired statement would be a free-boundary counterpart of Wickramasekera's interior theory~\cite{Wick14}, allowing stable regular parts without an a priori $\Hcal^{n-2}$ bound on the singular set and excluding the appropriate boundary classical singularities. This question is the subject of ongoing work.

\end{question}

\subsection*{Acknowledgments} M.L. is substantially supported by research grants from the Research Grants Council of the Hong Kong Special Administrative Region, China [Project No.: CUHK 14306624, CUHK 14308325 and RFS2526-4S02]. He would like to express his gratitude to Prof. Rick Schoen for his constant support and encouragement. D.P. was partially supported by the Cecilia Tanner Research Impulse Fund DRI349DP from Imperial College London. L.S. would like to thank Robert Haslhofer and Yevgeny Liokumovich for their support, and Paul Minter for several useful conversations. Groundwork discussions for this work were carried out at SLMath (Berkeley, California) during Fall 2024; the authors would like to thank the National Science Foundation that supported their residence under Grant No. DMS-1928930.

\section{Preliminaries and Notation} \label{sec: preliminaries} 

\subsection{Inequalities and constants} We shall use the convention that $C(\alpha, \beta, \gamma, \dots)$ denotes a positive constant (which may change from line to line) and depends only on the parameters $\alpha, \beta, \gamma, \dots$ unless we specify otherwise.

\subsection{Geometric setup} \label{subsec: geometric set-up}
We refer to \cite[pp.~24--29]{LeeSmooth} for the definition of a (smooth) $(n+1)$-dimensional manifold-with-boundary $M$. We stress that $\partial M$ is itself a (smooth) $n$-dimensional manifold, and that $M$ contains it. The interior of $M$, i.e.~$M\setminus \partial M$, is a (smooth) $(n+1)$-dimensional manifold\footnote{The notions of boundary and interior in this setting are not to be understood with respect to the topology of an ambient space in which we realise $M$ as a subset.}.

A Riemannian manifold-with-boundary $(M, \mathrm{g})$ is obtained from $M$ by endowing it with a smooth Riemannian metric $\mathrm{g}$, with smoothness required up to the boundary\footnote{In local charts, this means that the metric is smooth in the interior and all of its derivatives extend continuously to the boundary.}. 
Let $\nabla$ denote the Levi-Civita connection on $M$ induced by $\mathrm{g}$. 
We will sometimes stress the dependence on the metric by writing $\nabla^\mathrm{g}$.

We denote by $\bn$ the inward unit normal of $\partial M$, and we use the following notation for the (scalar-valued) second fundamental form of $\partial M$ in $M$: 
\[
\sff_{\partial M}(Y,Z) \coloneqq \mathrm{g}(\nabla_Y Z, \bn) = -\mathrm g(\nabla_Y \bn,Z) ,
\]
for vector fields $Y$ and $Z$ tangent to $\partial M$.

\begin{remark}\label{rmk: extension Omega}
We note that, essentially by the collar neighbourhood construction (\cite[Theorems 9.25 and 9.26]{LeeSmooth}), in the case $\partial M \neq \emptyset$, the manifold $M$ can be realised as a proper (i.e. strict) closed subset $\iota(M) \subsetneq \Omega$ of a smooth $(n+1)$-manifold $\Omega$, with $\iota \colon M \rightarrow \Omega$ a smooth proper embedding.
In other words, $M$ can be seen as a closed subset $M\subset \Omega$ with smooth boundary $\partial M$. Additionally, $\Omega$ can be taken to be smoothly homotopic to the interior of $M$. Furthermore, $\Omega$ admits a (smooth) Riemannian metric that restricts on $M$ to the given one\footnote{This follows essentially from a lemma of Borel, showing that a function is smooth up to the boundary, in the sense defined earlier, if and only if it admits a smooth extension across the boundary, see \cite[p.27]{LeeSmooth}.}. 

Roughly speaking, the collar neighbourhood (which is a subset of $M$ diffeomorphic to $\partial M \times [0, 1)$, with $\partial M \times \{0\}$ identified with $\partial M$) can be doubled, reflected and smoothly attached outside of $M$ along $\partial M$, thus creating $\Omega$.
\end{remark}

By \Cref{rmk: extension Omega} and \cite[Theorem 5.25 and Corollary 6.39(d)]{LeeRiem}, whenever $P\subset \partial M$ is a relatively compact smooth submanifold of $\partial M$, then there exists some $\rho_P>0$ such that the normal exponential map $F_{\partial M}$ of $\partial M$ is a diffeomorphism of $P\times [0,\rho_P)$ onto its image $\Ncal^+_{\rho_P}(P)\coloneqq F_{\partial M}(P\times [0, \rho_P))$ of $P$ in $M$.

\subsubsection{Free-boundary minimal hypersurfaces}
We will be interested in the following class of codimension one submanifolds-with-boundary, inside a smooth complete, Riemannian manifold-with-boundary $(M^{n+1}, \mathrm{g})$. 

\begin{definition}[Free-boundary minimal immersions and embeddings]\label{def: fb smooth immersed min hyp}
Given a smooth $n$-dimensional manifold-with-boundary $N$ and a relatively open set $U\subset M$, a smooth  immersion $\iota: (N, \iota^*\mathrm{g}) \to U\subset (M,\mathrm{g})$ is a \emph{free-boundary minimal immersion} in $U$ if 
\begin{enumerate}
	\item $\iota(\partial N) \subset \partial M \cap U$,
	\item its mean curvature vanishes on the interior $N\setminus \partial N$ of $N$,
	\item for all $z\in\partial N$, the inward conormal of $\partial N$ in $N$ at $z$ is mapped to $\bn(\iota(z))$.
\end{enumerate}

We shall denote the image $\iota(N)$ of $\iota$ by $\Sigma$, and let $\partial \Sigma \coloneqq \iota(\partial N)$. Then $(\Sigma, \partial \Sigma)$ is said to be a smooth \emph{immersed free-boundary minimal hypersurface} in $(M, \partial M)$. If, furthermore, $\iota$ is a homeomorphism of $N$ onto its image, we will say that $(\Sigma, \partial \Sigma)$ is a smooth \emph{embedded free-boundary hypersurface} in $(M, \partial M)$. 
\end{definition}

In what follows we shall often simplify the notation and write ``$\Sigma$ is a smooth immersed free-boundary minimal hypersurface'', without explicitly referring to the manifold-with-boundary $N$ or the map $\iota:N\to M$. Similarly, we shall often drop the more cumbersome notation $(M, \partial M, \mathrm g)$, $(\Sigma, \partial \Sigma)$ for (Riemannian) manifolds-with-boundary and immersed/embedded hypersurfaces with boundary, in favour of the simpler $(M, \mathrm g)$, $\Sigma$.

\begin{remark}\label{rmk: touching allowed}
Note that the intersection of the interior of $\Sigma$ and $\partial M$ is not necessarily empty as easy examples show. We stress that the orthogonality between $\Sigma$ and $\partial M$ is only imposed at points in $\partial \Sigma$. 
\end{remark}

\begin{remark}\label{rmk: FB stationarity}
Notice that the conditions in \Cref{def: fb smooth immersed min hyp} imply that the immersion $\iota:N\to M$ is stationary for area with respect to compactly supported variations $(\iota_t \colon N \rightarrow M)_{t\in(-\delta, \delta)}$ which preserve the inclusion $\iota_t(\partial N)\subset \partial M$ for all $t$, as is easily seen from the first variation formula for area\footnote{Notice that the conditions in \Cref{def: fb smooth immersed min hyp} are stronger than stationarity with respect to compactly supported variations which map $\partial N$ into the boundary $\partial M$. Indeed stationarity only implies the vanishing of the mean curvature of $\Sigma$ away from the \emph{touching set} $(\Sigma\setminus\partial\Sigma)\cap\partial M$, which is not assumed to be empty, as observed in \Cref{rmk: touching allowed}.}.
\end{remark}

\begin{notation}\label{not: immersion}
From now on, whenever for $A\subset M$ we write an integral $\int_{\Sigma\cap A}f$, we shall always implicitly assume it is shorthand for the integral $\int_{\iota^{-1}(A)}f\circ \iota\, d\Hcal^n_{\iota^*\mathrm g}$ over $\iota^{-1}(A)$ with respect to the induced metric $\iota^*\mathrm{g}$. Similarly quantities such as the second fundamental form $A_\Sigma$ or a unit normal $\nu_\Sigma$ are to be understood implicitly as quantities on $N$.
\end{notation}

\subsubsection{Singular free-boundary minimal hypersurfaces} We will be concerned with the class of \emph{stable} immersed free-boundary minimal hypersurfaces (see \Cref{def: stable immersion} below). It is well-known that allowing a small singular set is unavoidable for a comprehensive theory in this class (e.g.~for regularity, compactness, existence). We shall do so by introducing the notion of \emph{singular properly immersed free-boundary minimal hypersurfaces}, where we impose some a priori size constraint on the singular set. 

\begin{definition}[Singular properly immersed free-boundary minimal hypersurfaces]\label{def: sing min hyp}
Let $\iota:N \to M$ be a free-boundary minimal immersion as in \Cref{def: fb smooth immersed min hyp} with image $\Sigma = \iota(N)$, and let $\Scal = \overline{\Sigma} \setminus \Sigma$, where the closure is taken in $M$. We shall say that $\Sigma$ (or, more precisely, $\overline{\Sigma}$) is a \emph{singular properly immersed free-boundary minimal hypersurface} in $M$ if:
\begin{enumerate}
    \item The immersion $\iota$ is proper in $M \setminus \mathcal{S}$, i.e.~for every compact set $K\subset M\setminus \mathcal{S}$ the inverse image $\iota^{-1}(K)$ is compact in $N$.
    \item The measure $\mathcal{H}^n\mres \Sigma$ is locally finite, in the sense that for any compact $K\subset M$ we have $\mathcal{H}^n(\Sigma \cap K)<\infty$. 
    \item The set $\mathcal{S}\coloneqq \overline{\Sigma}\setminus \Sigma$ has locally finite $\mathcal{H}^{n-2}$-measure (in $M$), i.e.~for any compact $K\subset M$ we have $\mathcal{H}^{n-2}(\mathcal{S} \cap K)<\infty$. 
\end{enumerate}
We shall refer to $\Scal$ as the \emph{singular set} of $\Sigma$. 
\end{definition}

Notice that by (2), a singular properly immersed free-boundary minimal hypersurface $\Sigma$ defines an integral $n$-varifold $V=|\Sigma|$ in $M$ with weight measure $\Vert V \Vert = \mathcal{H}^n\mres\Sigma$ and $\spt \Vert V \Vert = \overline{\Sigma}$. Moreover, by definition of a smooth free boundary minimal immersion, $V$ is free-boundary stationary in $M \setminus \mathcal{S}$ in the sense of \Cref{subsec: fb varifolds}.

\begin{remark}\label{rmk: area bounds} 
Under the hypotheses of \Cref{def: sing min hyp} the varifold $V=|\Sigma|$ is free-boundary stationary in $M$ in the sense of \Cref{sec: FB stationary varifolds}. More precisely, under those assumptions one can extend stationarity from $M \setminus \mathcal{S}$ to $M$ (thus giving e.g. area-ratio estimates on balls centred on $\mathcal{S}$ by the monotonicity formula), see the discussion in \Cref{sec: capacity and extension of stationarity} (\Cref{prop: locally finite then global stationarity} in particular). 
\end{remark}

\begin{remark}\label{rmk: true singular set}
As stated, in \Cref{def: sing min hyp} we may have points in $\mathcal{S}$ that are not genuinely singular, e.g.~we take $\Sigma$ to be a punctured equatorial disc in $\mathbb{B}^3$ and $\mathcal{S}$ to be the missing point. Strictly speaking, one should define the singular set $\sing(\Sigma)$ as the set of all $x\in M$ such that for every $r>0$ the intersection of $\overline{\Sigma}$ with the geodesic ball $B_r(x)$ in $\Omega$ is not the image of a smooth free-boundary immersion in $M$ which is \emph{proper} in $B_r(x)$, that is, $\sing(\Sigma)$ is the set of points across which $\Sigma$ cannot be smoothly extended as a proper free-boundary immersion\footnote{The set $\sing(\Sigma)$ corresponds to the \emph{(non-immersed) singular set} of the varifold $|\Sigma|$ induced by $\Sigma$, in the definition of \cite{WicM2}.}. However, clearly $\sing (\Sigma) \subset\Scal$ and, after extending $\Sigma$ across its removable singularities, one may take $\Scal = \sing (\Sigma)$. 
\end{remark}

\begin{remark}
We will speak of a smooth properly immersed free-boundary minimal hypersurface when $\mathcal{S} = \emptyset$. In this case the notion reduces to \Cref{def: fb smooth immersed min hyp} with the additional assumption of properness in $M$. 
\end{remark}

\begin{remark}
In our main compactness theorem, cf. \Cref{thm: compactness}, we specialise to (possibly) singular embeddings, in which case the immersion in \Cref{def: sing min hyp} is assumed to be an embedding. We also recall that a proper injective  immersion with closed image is necessarily an embedding.
\end{remark}

\begin{remark}
When we refer to a singular properly immersed free-boundary minimal hypersurface as being \emph{two-sided} we shall always mean that there exists a continuous global determination of a unit normal (more precisely, that the pullback normal bundle on $N$ is trivial, and a global unit normal is a continuous section of the pullback unit normal bundle on $N$).
\end{remark}

In what follows, we shall often simply write ``$\Sigma$ is a singular free-boundary minimal hypersurface in $M$''. This is always to be interpreted as shorthand for ``$\Sigma$ is a singular properly immersed free-boundary minimal hypersurface in $M$'' in the sense of \Cref{def: sing min hyp}, unless we specify otherwise.

\subsubsection{Stability} 
We shall take the following analytic inequality as our notion of stability for singular properly immersed free-boundary minimal hypersurfaces. 
We shall restrict ourselves to two-sided hypersurfaces, as we make clear in the following definition.

\begin{definition}[Stable free-boundary minimal hypersurface] \label{def: stable immersion}
Let $\Sigma=\iota(N)$ be a singular properly immersed free-boundary minimal hypersurface in the sense of \Cref{def: sing min hyp}. Further assume $\Sigma$ to be two-sided, with unit normal $\nu_\Sigma$ and second fundamental form $A_\Sigma$\footnote{More precisely, $\nu_\Sigma$ and $A_\Sigma$ are tensor fields on $N$ along $\iota$, see \cref{not: immersion}.}. We say that $\Sigma$ is \emph{stable} if the following inequality holds
\begin{equation}\label{eq: stability-sing}
    \int_N |\nabla_N \varphi|^2-(|A_\Sigma|^2+\Ric_M(\nu_\Sigma, \nu_\Sigma))\varphi^2 - \int_{\partial N}\sff_{\partial M}(\nu_\Sigma, \nu_\Sigma)\varphi^2\geq 0
\end{equation}
for all $\varphi\in C^{0,1}_c(N)$, where $N$ is endowed with $\iota^* \mathrm{g}$. 
\end{definition} 
\begin{remark}\label{rmk: intro stability}
It turns out (see \Cref{def: sing min hyp} and the discussion in \Cref{sec: capacity and extension of stationarity}) that the assumptions in \Cref{def: sing min hyp} imply that the singular set $\mathcal{S}$ has \emph{vanishing relative 2-capacity in $\Sigma$}, denoted $\operatorname{cap}_2(\Scal)=0$. Thus, \eqref{eq: stability-sing} extends to any test function of the form $\varphi = \phi \circ \iota$ for $\phi \in C^1_c(M)$. The proof is by approximation using the cutoff functions provided by \Cref{cor: cap} (see \Cref{lemma: trace inequality} for a similar argument). 
\end{remark}

From now on, following the convention laid out in \cref{not: immersion}, when we want to plug a test function of the form $\varphi = \phi \circ \iota$ for $\phi \in C^1_c(M)$ into \eqref{eq: stability-sing}, we shall always write integrals directly on $\Sigma$; that is, we shall write $\int_\Sigma|\nabla_\Sigma\phi|^2$ in place of $\int_N|\nabla_N(\phi\circ\iota)|^2$. 

\subsubsection{Trace and Sobolev inequality}\label{subsec: Trace-Sobolev}
Fix a relatively compact open set $U\subset M$ such that the distance function $d_{\partial M}$ to $\partial M$ is smooth on $U$.
Let $\Sigma$ be a two-sided singular free-boundary minimal hypersurface in $M$ such that $\overline{\Sigma}\subset U$. Then,  appropriate versions of the classical trace inequality and Sobolev inequality hold for $\Sigma$.

In this subsection, functions on $\Sigma$ are understood on the parameter manifold $N$, as in \Cref{not: immersion}.

\begin{lemma}\label{lemma: trace inequality}
Let $\Sigma$ be a two-sided singular free-boundary minimal hypersurface in $M$ such that $\overline{\Sigma}\subset U$. Then if $f$ is a non-negative compactly supported Lipschitz function in $U$, we have 
\begin{equation}\label{eq: trace inequality smooth 1}
	\int_{\partial \Sigma}f \leq \int_{\Sigma}|\nabla_\Sigma f|+\left(\sqrt n\sup_{U}|\nabla^2_\mathrm{g}d_{\partial M}|\right)\int_\Sigma f.
\end{equation}
For simplicity, write $C_\mathrm{tr}=C_\mathrm{tr}(n, U)\coloneqq \sqrt n\sup_{U}|\nabla^2_\mathrm{g} d_{\partial M}|$.
\end{lemma}

\begin{proof}
Let us first give a one-line proof in the case $f$ is supported away from $\Scal$ (see also \cite[Lemma 2.1]{Edelen2016}), and then extend it to the more general case of $f$ compactly supported in $U$ by a (more or less standard) excision argument based on \Cref{cor: cap}.

Suppose $f$ is supported away from $\mathcal S$. Then by the divergence theorem on the immersed hypersurface $\Sigma$ we have
\[
	\int_{\partial \Sigma} f = \int_{\partial \Sigma} \mathrm{g}(f \nabla^\mathrm{g}d_{\partial M}, \eta)= -\int_\Sigma \operatorname{div}_\Sigma(f \nabla^\mathrm{g} d_{\partial M}) =-\int_\Sigma \mathrm{g}(\nabla_\Sigma f, \nabla^\mathrm{g} d_{\partial M})+f\operatorname{tr}_\Sigma(\nabla^2_\mathrm{g}d_{\partial M}),
\]
where in the last equality we used that the mean curvature of $\Sigma$ vanishes identically. Equation \eqref{eq: trace inequality smooth 1} follows directly.

For a general $f$, let $K\Subset U$ be such that $\spt f\subset K$. By \Cref{cor: cap}, there exist cutoffs $\chi_j\in C_c^{0,1}(U)$ such that $0\leq \chi_j\leq 1$, $\spt \chi_j\subset U\setminus \Scal$, $\chi_j\to 1$ locally on $\Sigma$, and $\int_{\Sigma}|\nabla_\Sigma\chi_j|^2\to 0$.
Apply \eqref{eq: trace inequality smooth 1} to $f_j\coloneqq \chi_j f\circ \iota$ to get
\[
\begin{aligned}
	\int_{\partial \Sigma} f_j &\leq \int_\Sigma|\nabla_\Sigma(\chi_j f)|+C_\mathrm{tr}\int_\Sigma f_j \leq \int_\Sigma (\chi_j|\nabla_\Sigma f|+f|\nabla_\Sigma\chi_j|)+C_{\mathrm{tr}}\int_\Sigma f_j.
\end{aligned}
\]
Notice that by (ii) of \Cref{def: sing min hyp}, $\Sigma$\footnote{or, more precisely, the measure $\Hcal^n \mres\Sigma$.} has locally finite mass, so by the properties of the cutoffs $\chi_j$ we have
\[
\int_\Sigma f|\nabla_\Sigma\chi_j|\leq \|f\|_{L^\infty}\Hcal^n(\Sigma\cap K)^{1/2}\left(\int_{\Sigma\cap K}|\nabla \chi_j|^2\right)^{1/2}\to 0.
\]
The other terms on the right hand side converge by dominated convergence, and 
\[
\int_{\partial \Sigma}f \leq \liminf_j\int_{\partial \Sigma}f_j
\]
by Fatou's lemma. This concludes the proof.
\end{proof}

\begin{remark}
Notice that by taking $f\equiv 1$ in \eqref{eq: trace inequality smooth 1}, we obtain an upper bound for $\mathcal{H}^{n-1}(\partial \Sigma)$ in terms of the geometry of $U$ and $\mathcal{H}^n(\Sigma)$.
\end{remark}

The trace inequality turns out to be very useful in dealing with \emph{stable} singular free-boundary minimal hypersurfaces, because it allows us to absorb the boundary term in the stability inequality into an interior term. 

\begin{lemma}\label{lemma: stability absorbed 1}
Let $\Sigma$ be a \emph{stable} singular free-boundary minimal hypersurface in $M$ such that $\overline{\Sigma}\subset U$. For every $\alpha>0$ there is a finite constant $C(n,\alpha, U)$ depending on $n, \alpha$, $\sup_{U}|\Ric_\mathrm{g}|, \sup_{U}|\nabla^2_\mathrm{g} d_{\partial M}|$  such that
	\[
	\int_{\Sigma}|A|^2\zeta^2\leq (1+\alpha)\int_\Sigma|\nabla_\Sigma\zeta|^2+C(n,\alpha, U)\int_\Sigma \zeta^2
	\]
	for all bounded compactly supported Lipschitz functions $\zeta$ on $U$. More precisely, $C(n,\alpha,U)$ can be taken to be $\sup_U|\Ric_{\mathrm{g}}|+(4\alpha^{-1}+\sqrt n)\sup_{U}|\nabla^2_\mathrm{g}d_{\partial M}|^2$.
\end{lemma}

\begin{proof}
Let us first consider the case where $\zeta$ is supported away from $\Scal$. Then \eqref{eq: stability-sing} holds. 
Therefore, by \Cref{lemma: trace inequality} and Young's inequality
\[
	\left|\int_{\partial \Sigma} \sff_{\partial M}(\nu, \nu) \zeta^2\right|\leq \alpha \int_\Sigma |\nabla_\Sigma\zeta|^2+ \|\sff_{\partial M}\|_{L^\infty}\left(\alpha^{-1}\|\sff_{\partial M}\|_{L^\infty}+C_\mathrm{tr}\right)\int_\Sigma \zeta^2.
\]
This proves the result for compactly supported bounded Lipschitz functions $\zeta$ with $\spt \zeta\subset U\setminus\mathcal{S}$.
	
For the general case, it is enough to apply the estimate we have just obtained to $\zeta_j\coloneqq \chi_j \zeta\in C_c^{0,1}(U\setminus\mathcal{S})$, where $\chi_j$ is as in the proof of \Cref{lemma: trace inequality}.
\end{proof}

Finally, we are going to need to use the following Sobolev inequality, which follows from Edelen's Michael--Simon inequality for smooth free-boundary hypersurfaces in \cite[Theorem 2.2]{Edelen2016}.
	
\begin{lemma}\label{lemma: Sobolev inequality FB 1}
Let $\Sigma$ be a two-sided singular free-boundary minimal hypersurface in $M$ such that $\overline{\Sigma}\subset U$.
There are a dimensional constant $C_{\mathrm{Sob}}=C_{\mathrm{Sob}}(n)>0$ and a geometric constant $K_{\mathrm{Sob}}=K_{\mathrm{Sob}}(n,U)$ depending on $n$ and $\sup_{U}|\nabla^2_\mathrm{g}d_{\partial M}|>0$ such that for all compactly supported Lipschitz functions $\zeta$ on $U$, we have
\begin{equation}\label{eq: Sobolev inequality 3d 1}
	\left(\int_\Sigma |\zeta|^{\frac{2n}{n-2}}\right)^{\frac{n-2}{n}}\leq C_{\mathrm{Sob}}\left(\int_\Sigma|\nabla_\Sigma \zeta|^2+K_{\mathrm{Sob}}^2\int_\Sigma\zeta^2\right)\quad \text{for $n\geq 3$},
\end{equation}
and
\begin{equation}\label{eq: Sobolev inequality 2d 1}
	\left(\int_\Sigma |\zeta|^{4}\right)^{\frac{1}{2}}\leq C_{\mathrm{Sob}}
	\left(\int_\Sigma \zeta^2\right)^{\frac{1}{2}}\left(\int_\Sigma |\nabla_\Sigma \zeta|^2+K_{\mathrm{Sob}}^2 \int_\Sigma\zeta^2\right)^{\frac{1}{2}}\quad \text{for $n= 2$}.
\end{equation}
\end{lemma}
	
\begin{proof}
For $\zeta$ supported away from $\mathcal{S}$, this follows from the free-boundary Michael-Simon inequality in \cite[Theorem 2.2]{Edelen2016} and the trace inequality \eqref{eq: trace inequality smooth 1}
applied to $f=|\zeta|^{\frac{2(n-1)}{n-2}}$ for $n\geq 3$ and $f=\zeta^2$ for $n=2$, followed by Young's inequality. 

The general case then follows from the same approximation argument as in \Cref{lemma: stability absorbed 1} using \Cref{cor: cap}.
\end{proof}

\subsection{Iteration \`a la De Giorgi}\label{sec: abstract DG}
Now let $\Sigma$ be a free-boundary minimal hypersurface in $M$. In what follows we shall be interested in the situation where $\Sigma$ is weakly close, in a sense to be made precise later, to a smooth embedded two-sided hypersurface $P\subset M$ of one of the following two types:
\begin{itemize}
	\item $P$ is a geodesic ball $B^{\partial M}_R(p)$ in $\partial M$ centred at a point $p\in \partial M$ and $0<R<\inj_p(\partial M)$ (see \Cref{sec: boundary sheeting});
	\item $P$ is an $n$-dimensional free-boundary half-disc  in $M$ intersecting $\partial M$ orthogonally (see \Cref{sec: fb sheeting}). 
\end{itemize}
As explained in the Introduction, the central idea in the proofs of both \Cref{thm: boundary eps reg} and \Cref{thm: orthogonal eps reg} is an $L^2$-to-$L^\infty$ estimate for the tilt relative to $P$, which is obtained via a non-linear adaptation of De Giorgi's iteration technique following \cite{Bellettini}.

The overall strategy in both cases is very similar: first, we introduce a \textit{tilt function} to measure the deviation of the normal of $\Sigma$ from the normal of $P$, and show that this function satisfies a differential inequality involving the norm $|A|^2$ of the second fundamental form of $\Sigma$; then by using truncations of this tilt as appropriate test functions in the stability inequality, we obtain an intrinsic Caccioppoli-type inequality on $\Sigma$; finally, thanks to the differential inequality satisfied by the tilt and the Sobolev inequality in \Cref{subsec: Trace-Sobolev}, we set up a De Giorgi-style iterative scheme which results in a nonlinear recurrence equation for the $L^2$-norm of successive truncations of the tilt function on smaller and smaller cylinders. The superlinear nature of the recurrence (see e.g.~\cite[Lemma B.1]{Bellettini}) implies that, provided the $L^2$-norm of the tilt is sufficiently small at the original scale, then the $L^2$-norm of successive truncations on shrinking cylinders must go to 0, thereby providing an $L^\infty$ estimate for the tilt function at a smaller scale.

We now formulate this procedure abstractly, in a framework that applies to both tilt functions we shall use in \Cref{sec: boundary sheeting} and \Cref{sec: fb sheeting}.

\subsubsection{Setup}
Let $\Sigma$ be a two-sided singular free-boundary minimal hypersurface in $M$ and assume $\overline{\Sigma}\subset U$, where $U$ is a relatively compact open subset of $M$. Let $\sigma\in(0,1)$, $R>0$ and let $\{U_r\}_{0<r\leq R}$ be a nested family (increasing in $r$) of relatively open subsets of $\Sigma$. Assume that for all $\sigma R\leq r<r' \leq R$, there exists a compactly supported Lipschitz function $\phi_{r,r'}$ on $\Sigma$ such that 
\[
	0\leq \phi_{r,r'}\leq 1, \quad \phi_{r,r'}\equiv 1 \text{ on } U_r, \quad \spt \phi_{r,r'}\subset U_{r'}, \quad |\nabla_\Sigma \phi_{r, r'}|\leq \frac{C_c}{r'-r}
\]
for some constant $C_c$ independent of $r,r'$.

Let $u:\Sigma\to [0,1]$ be a locally Lipschitz function on $\Sigma$ and assume that there are  constants $k_*\in(0,1]$, $C_{\mathrm{Cac}}, K_{\mathrm{Cac}}>0$ such that for every $k\in[0, k_*]$ and every compactly supported Lipschitz function $\phi\in C^{0,1}_c(\Sigma)$, the following Caccioppoli-type inequality is satisfied
\begin{equation}\label{eq: abstract Caccioppoli}
	\int_{\{u>k\}} \left(1-\frac{k}{u}\right)|\nabla_\Sigma u|^2\phi^2 \leq C_{\mathrm{Cac}}\left(\int_\Sigma (u-k)^2_+|\nabla_\Sigma\phi|^2+K_{\mathrm{Cac}}^2\int_{\{u>k\}}\phi^2\right),
\end{equation}
where as usual all integrals should more precisely be understood on $N$ (see \ref{not: immersion}).
Finally, let $K_*\coloneqq \max\{K_\mathrm{Sob}, K_{\mathrm{Cac}}\}$.

\begin{lemma}[De Giorgi iteration]\label{lemma: abstract DG}
Let $\Sigma$, $u$ be as above and let $\sigma\in(0,1)$ be fixed. There exists $\varepsilon_{\mathrm{DG}}=\varepsilon_{\mathrm{DG}}(n, k_*, C_{\mathrm{Sob}}, C_\mathrm{Cac}, C_c)>0$ with the following property.
Let $\mathcal{d}\in(0,k_*]$.
If $K_*R\leq 1$ and
\begin{equation}\label{eq: DG initial estimate}
	R^{-n}\int_{U_R}u^2\leq \varepsilon_{\mathrm{DG}}^2(1-\sigma)^n \mathcal{d}^{n+2},
\end{equation}
then 
\begin{equation}\label{eq: DG conclusion}
	\sup_{U_{\sigma R}}u\leq \mathcal{d}.
\end{equation}
\end{lemma}

\begin{proof}
Define
\[
	k_\ell=\mathcal{d}\left(1-\frac{1}{2^\ell}\right), \quad r_\ell= \sigma R+2^{-\ell}(1-\sigma)R
\]
for $\ell\geq 0$. Hence $k_0=0$ and $k_\ell\uparrow \mathcal{d}\leq k_*$,  $r_0=R$ and $r_\ell\downarrow \sigma R$.
	
Since $k_\ell\leq k_*$, we can apply \eqref{eq: abstract Caccioppoli}. Let $U_\ell\coloneqq U_{r_\ell}$, $u_\ell\coloneqq (u-k_\ell)_+$, and $D_\ell\coloneqq U_\ell \cap \{u>k_\ell\}$. Then let 
\[
	G_\ell\coloneqq \int_{\Sigma\cap U_\ell} u_\ell^2.
\]
Notice that, by definition (since both the cylinders and the truncations decrease), 
\[
G_{\ell+1}\leq G_\ell.
\]
Now let $\phi_\ell\coloneqq \phi_{r_{\ell+1}, r_\ell}$, so that 
\[
0\leq \phi_\ell\leq 1, \quad \phi_\ell\equiv 1 \text{ on } U_{\ell+1}, \quad \spt \phi_\ell\subset U_{\ell}, \quad |\nabla_\Sigma \phi_\ell|\leq \frac{C_c}{(1-\sigma)R}2^{\ell+1},
\]
where $C_c$ is the fixed constant from the setup.

On $D_{\ell+1}$, by definition, we have $u>k_{\ell+1}$, so that
\[
	1-\frac{k_\ell}{u}\geq \frac{k_{\ell+1}-k_\ell}{k_{\ell+1}}\geq 2^{-\ell-1} \quad \text{on $D_{\ell+1}$}.
\]
Therefore,
\[
2^{-\ell-1}\int_\Sigma\phi_\ell^2|\nabla_\Sigma u_{\ell+1}|^2 = 2^{-\ell-1}\int_{U_{\ell}\cap\{u>k_{\ell+1}\}}\phi_\ell^2|\nabla_\Sigma u|^2\leq \int_{D_\ell}\left(1-\frac{k_\ell}{u}\right)|\nabla_\Sigma u|^2\phi_\ell^2,
\]
and, by \eqref{eq: abstract Caccioppoli}, since $|\nabla_\Sigma\phi_\ell|^2\leq C^2_c4^\ell ((1-\sigma)R)^{-2}$ and $\spt \phi_\ell\subset U_\ell$, we obtain 
\[
\begin{aligned}
\int_\Sigma \phi_\ell^2|\nabla_\Sigma u_{\ell+1}|^2 &\leq 2^{\ell+1}C_{\mathrm{Cac}}\left(C^2_c\,4^{\ell+1}((1-\sigma)R)^{-2}G_\ell+K^2_{\mathrm{Cac}}\,\Hcal^n(D_\ell)\right)\\
&\leq C_08^\ell\left(((1-\sigma)R)^{-2}G_\ell+K^2_{\mathrm{Cac}}\,\Hcal^n(D_\ell)\right),
\end{aligned}
\]
where, throughout this proof, $C_0$ denotes a constant which is allowed to change from line to line depending on $n, k_*, C_{\mathrm{Cac}}, C_\mathrm{Sob}, C_c$ only.

Since $|\nabla_\Sigma(\phi_\ell u_{\ell+1})|^2\leq 2\phi^2_\ell|\nabla_\Sigma u_{\ell+1}|^2+2u^2_{\ell+1}|\nabla_\Sigma \phi_\ell|^2$, then
\begin{equation}\label{eq: abstract energy estimate}
	\int_\Sigma |\nabla(\phi_\ell u_{\ell+1})|^2\leq C_08^\ell \left(((1-\sigma)R)^{-2}G_\ell+K^2_{\mathrm{Cac}}\,\Hcal^n(D_\ell)\right).
\end{equation}

Notice  that on $D_\ell$ we have $u_{\ell-1}=u-k_{\ell-1}>k_\ell-k_{\ell-1}\geq \mathcal{d}2^{-\ell}$, so that by Markov's inequality
\begin{equation}\label{eq: abstract Markov estimate}
\mathcal{H}^n(D_\ell)\leq \left(\frac{2^\ell}{ \mathcal{d}}\right)^2\int_{D_\ell}u_{\ell-1}^2\leq \left(\frac{2^\ell}{ \mathcal{d}}\right)^2\int_{U_{\ell-1}}u_{\ell-1}^2\leq 4^\ell \mathcal{d}^{-2} G_{\ell-1},
\end{equation}
as $D_\ell\subset U_\ell\subset U_{\ell-1}$.
Since $u_{\ell+1}\leq u_\ell$, $G_\ell\leq G_{\ell-1}$ and $\mathcal{d}\leq 1$, by \eqref{eq: abstract energy estimate} and \eqref{eq: abstract Markov estimate}, we obtain 
\begin{equation}\label{eq: energy+Markov}
\begin{aligned}	
	\int_\Sigma |\nabla_\Sigma(u_{\ell+1}\phi_\ell)|^2+K_{\mathrm{Sob}}^2\int_\Sigma u^2_{\ell+1}\phi_\ell^2 &\leq C_08^\ell\left(((1-\sigma)R)^{-2}G_\ell+4^\ell\mathcal{d}^{-2} K_*^2G_{\ell-1}\right)+K_*^2G_\ell\\
	&\leq C_032^\ell ((1-\sigma)R)^{-2}\mathcal{d}^{-2}G_{\ell-1},
\end{aligned}	
\end{equation}
where in the last inequality we used the condition $K_*R\leq 1$.

\noindent{\bf Case $n\geq 3$}: Assume $n\geq 3$, and apply the Sobolev inequality in \Cref{lemma: Sobolev inequality FB 1} to obtain 
\[
\begin{aligned}
	\left(\int_\Sigma |u_{\ell+1}\phi_\ell|^{\frac{2n}{n-2}}\right)^{\frac{n-2}{n}} &\leq C_{\mathrm{Sob}}\left(\int_\Sigma|\nabla_\Sigma(u_{\ell+1}\phi_\ell)|^2+K_{\mathrm{Sob}}^2\int_\Sigma u_{\ell+1}^2\phi_\ell^2\right)\\
	&\leq C_{\mathrm{Sob}}\left(\int_\Sigma|\nabla_\Sigma(u_{\ell+1}\phi_\ell)|^2+K_{\mathrm{Sob}}^2G_\ell\right).
\end{aligned}
\]
Since $u_{\ell+1}\leq u_\ell$ and $\spt\phi_\ell\subset U_\ell$ and $\phi_\ell\equiv 1$ on $U_{\ell+1}$, then, by \eqref{eq: energy+Markov}
\begin{equation}\label{eq: abstract Markov 1}
	\left(\int_{U_{\ell+1}} u_{\ell+1}^{\frac{2n}{n-2}}\right)^{\frac{n-2}{n}} \leq C_032^\ell((1-\sigma)R)^{-2}\mathcal{d}^{-2} G_{\ell-1}.
\end{equation}
By H\"older's inequality and \eqref{eq: abstract Markov estimate}
\begin{equation}\label{eq: abstract iteration 2step}
	\begin{aligned}
		G_{\ell+1} = \int_{D_{\ell+1}}u_{\ell+1}^2 &\leq \left(\int_{U_{\ell+1}} u_{\ell+1}^{\frac{2n}{n-2}}\right)^{\frac{n-2}{n}}\mathcal{H}^n(D_{\ell+1})^{\frac{2}{n}}\\
		&\leq C_032^\ell ((1-\sigma)R)^{-2}\mathcal{d}^{-2} G_{\ell-1}( 4^{\ell+1}\mathcal{d}^{-2}G_\ell)^\frac{2}{n}\\
		&\leq C_0 128^\ell ((1-\sigma)R)^{-2}\mathcal{d}^{-2-\frac{4}{n}} G_{\ell-1}^{1+\frac{2}{n}}.
	\end{aligned}
\end{equation}

\bigskip

\noindent {\bf Case $n=2$}:
For $n=2$, the Sobolev inequality in \Cref{lemma: Sobolev inequality FB 1} gives 
\[
\begin{aligned}
	\left(\int_\Sigma u_{\ell+1}^4\phi_\ell^{4}\right)^{1/2} &\leq C_{\mathrm{Sob}}\left(\int_\Sigma u_{\ell+1}^2\phi_\ell^2\right)^{1/2}\left(\int_\Sigma |\nabla_\Sigma (u_{\ell+1}\phi_\ell)|^2+ K_{\mathrm{Sob}}^2\int_\Sigma u_{\ell+1}^2\phi_\ell^2\right)^{1/2}
\end{aligned}		
\]
so that
\[
\begin{aligned}
	\left(\int_{U_{\ell+1}}u^4_{\ell+1}\right)^{1/2} &\leq C_{\mathrm{Sob}}G_\ell^{1/2}\left(C_032^\ell((1-\sigma)R)^{-2}\mathcal{d}^{-2} G_{\ell-1}\right)^{1/2}\\
	&\leq C_032^\ell((1-\sigma)R)^{-1}\mathcal{d}^{-1}G_{\ell-1}.
\end{aligned}
\]
By H\"older's inequality,
\begin{equation}
\begin{aligned}\label{eq: abstract iteration 2step 2d}
	G_{\ell+1} &\leq \left(\int_{U_{\ell+1}}u_{\ell+1}^4\right)^{1/2}\mathcal{H}^n(D_{\ell+1})^{1/2}\\
	&\leq \left(C_032^\ell((1-\sigma)R)^{-1}\mathcal{d}^{-1}G_{\ell-1}\right)\left(4^{\ell+1}\mathcal{d}^{-2}G_\ell\right)^{1/2}\\
	&\leq C_0 128^\ell((1-\sigma)R)^{-1}\mathcal{d}^{-2}G_{\ell-1}^{3/2}.
\end{aligned}
\end{equation}

Notice that \eqref{eq: abstract iteration 2step} and \eqref{eq: abstract iteration 2step 2d} are of the same form. More concretely, if we set 
\[
p_n\coloneqq \begin{cases}
	\frac{2}{n} &\text{if $n\geq 3$,}\\
	\frac{1}{2}	&\text{if $n=2$,}
\end{cases}
\]
and $S_j\coloneqq ((1-\sigma)R)^{-n}\mathcal{d}^{-n-2}G_{2j}$, then by letting $\ell=2j+1$ in  \eqref{eq: abstract iteration 2step} and \eqref{eq: abstract iteration 2step 2d} respectively, we obtain
\[
	S_{j+1}\leq C_0 128^{2j+1} S^{1+p_n}_{j}
\]
for $j\geq 0$. The standard nonlinear iteration lemma (see \cite[Lemma B.1]{Bellettini}) implies that there exists a constant $\varepsilon_{\mathrm{DG}}=\varepsilon_{\mathrm{DG}}(n, k_*, C_\mathrm{Sob}, C_{\mathrm{Cac}}, C_c)>0$ such that $S_j\to 0$ as $j\to \infty$ provided $S_0\leq \varepsilon_{\mathrm{DG}}^2$.
Since $\{G_\ell\}_\ell$ is a non-increasing sequence, this means that
\[
G_\ell\to 0\, \text{as $\ell\to\infty$, provided that } G_0\leq \varepsilon_{\mathrm{DG}}^2 ((1-\sigma)R)^n \mathcal{d}^{n+2},
\]
as required.
\end{proof}

We shall always apply \Cref{lemma: abstract DG} in the following form.
\begin{corollary}\label{cor: abstract DG cor}
Under the assumptions of \Cref{lemma: abstract DG}, let 
\[
	\mathrm{E}(R)\coloneqq \left(R^{-n}\int_{U_R}u^2\right)^{\frac{1}{2}}.
\]
Then there exist $\varepsilon_*=\varepsilon_*(n, k_*,  C_\mathrm{Sob}, C_{\mathrm{Cac}}, C_c)>0$ and $C_*=C_*(n, k_*, C_\mathrm{Sob}, C_{\mathrm{Cac}}, C_c)>0$ such that if 
\[
	K_*R\leq 1\quad \text{and } \mathrm{E}(R)\leq \varepsilon_*(1-\sigma)^{n/2},
\]
then $\displaystyle{\sup_{U_{\sigma R}}u\leq C_*(1-\sigma)^{-\frac{n}{n+2}}\,\mathrm{E}(R)^{\frac{2}{n+2}}}$.
\end{corollary}
\begin{proof}
	This follows from \Cref{lemma: abstract DG}. Indeed, let $\varepsilon_*=\varepsilon_{\mathrm{DG}}k_*^{\frac{n+2}{2}}$ and $C_*\coloneqq  \varepsilon_{\mathrm{DG}}^{-\frac{2}{n+2}}$. Then, if $\mathrm{E}(R)\leq \varepsilon_*(1-\sigma)^{n/2}$, take $\mathcal{d}\coloneqq \varepsilon_{\mathrm{DG}}^{-\frac{2}{n+2}}(1-\sigma)^{-\frac{n}{n+2}}\,\mathrm{E}(R)^{\frac{2}{n+2}}$ and apply \Cref{lemma: abstract DG}.
\end{proof}

\section{Boundary-parallel sheeting} \label{sec: boundary sheeting}

\subsection{Geometric setup}\label{sec: geometric setup boundary}
Throughout this section, let us fix $p\in\partial M$, and let $\rho_0>0$ be such that 
\begin{itemize}
	\item the intrinsic geodesic balls $P_R = B^{\partial M}_R(p)$ are smooth and relatively compact in $\partial M$ for $0\leq R\leq 2\rho_0$;
	\item the normal exponential map $F_{\partial M}$ of $\partial M$ is a diffeomorphism of $P_{2\rho_0}\times[0, 2\rho_0)$ onto a relatively open set.
\end{itemize}
For $0\leq \delta\leq 2\rho_0$, $0<R\leq 2\rho_0$, define $\Ncal^+_\delta(P_R)\coloneqq F_{\partial M}\bigl(P_R\times [0,\delta)\bigr)$.
In particular, the distance function $d_{\partial M}$ to $\partial M$ is smooth on $\Ncal^+_{2\rho_0}(P_{2\rho_0})$.

Let $\Lambda\geq 0$ be such that
\begin{equation}\label{eq: geometric bounds boundary}
	|\nabla^2_\mathrm{g}d_{\partial M}|^2+|\nabla^3_\mathrm{g} d_{\partial M}|+|\Riem_\mathrm{g}|\leq \Lambda^2\quad \text{ in $\Ncal^+_{2\rho_0}(P_{2\rho_0})$}.
\end{equation}
Now let us assume as in \Cref{thm: boundary eps reg} that $\Sigma$ is a two-sided free-boundary minimal hypersurface in $\Ncal^+_{\rho_0}(P_R)$ such that $\overline{\Sigma}\subset\Ncal^+_{\delta}(P_R)$ for some $\delta\in(0, \rho_0), R \in(0, R_0)$, where $R_0\in (0, \rho_0)$ is to be specified later depending on $n$ and $\Lambda$ only.

\begin{definition}[Tilt function]\label{def: tilt}
	Let $X\coloneqq \nabla_\Sigma d_{\partial M}$, $g=|X|_\mathrm{g}$, and $\theta\coloneqq \mathrm{g}(\nu, \nabla^\mathrm{g}d_{\partial M})$. We shall refer to $g$ as the \textit{tilt function} on $\Sigma$ with respect to $\partial M$.
\end{definition}

We start by computing the partial differential inequality satisfied by the tilt on $\Sigma$. 

\begin{lemma}[Differential inequalities for the tilt] \label{lemma: diff ineq tilt}
Let $(M,\mathrm{g})$ and $P_R$ be as above. Additionally, let $\Sigma$ be a smooth two-sided free-boundary minimal hypersurface in $M$ as in \Cref{sec: preliminaries} and assume that $\Sigma\subset \Ncal^+_\delta(P_R)$. Let $\nu$ be a choice of unit normal to $\Sigma$, $A$ its second fundamental form, and let $X, g, \theta$ be as in \Cref{def: tilt}. Then $g^2+\theta^2=1$. Set
$B\coloneqq \nabla^2_\mathrm{g} d_{\partial M}|_{T\Sigma}$. Then the following hold.

First, on $\{g>0\}$, if $v\coloneqq X/g$, then
\[
\nabla_\Sigma g=\theta A(\cdot,v)+B(\cdot,v)
\]
and, for every $\gamma>0$,
\[
|\nabla_\Sigma g|^2 \leq (1+\gamma)(1-g^2)|A(\cdot,v)|^2 + C(\gamma) |B|^2.
\]
Second, assuming that \eqref{eq: geometric bounds boundary} holds, then for every $\gamma>0$ on $\{g>0\}$
\[
\begin{aligned}
	g\Delta_\Sigma g
	&\geq (1-\gamma)|A|^2(1-g^2) - \left(\frac{n-1}{n}+\gamma\right)|A|^2 - \gamma|A|^2 g^2 - C(n,\gamma)\Lambda^2.
\end{aligned}
\]
\end{lemma}

\begin{proof}
Since $d_{\partial M}$ is the signed $\mathrm{g}$-distance to $P$ in $\Ncal^+_{\rho_0}(P)$, we have
$|\nabla^{\mathrm{g}} d_{\partial M} |_{\mathrm{g}}=1$ in the tubular neighbourhood. Thus $\nabla^{\mathrm{g}} d_{\partial M}=X+\theta\nu$ along $\Sigma$, and so $g^2+\theta^2=1$.

The identity $\nabla_\Sigma^2d_{\partial M}=B+\theta A$ follows from the relation between
ambient and intrinsic Hessians. Indeed, for $Y,Z\in T\Sigma$,
\[
        \nabla_\Sigma^2 d_{\partial M}(Y,Z) = \nabla^2_{\mathrm{g}} d_{\partial M}(Y,Z) + \mathrm{g}(\nabla^{\mathrm{g}}d_{\partial M},\nu)A(Y,Z) = B(Y,Z)+\theta A(Y,Z).
\]
By taking $Z=X$ in this identity, we immediately obtain that
\[
g\, Yg = \frac{1}{2}Y(g^2)=\nabla_\Sigma^2d_{\partial M}(Y, X)=B(Y,X)+\theta A(Y,X).
\]
On $\{g>0\}$, writing $X=gv$, this becomes $Yg=\theta A(Y,v)+B(Y,v)$ as required.

Moreover, taking norms and using $\theta^2=1-g^2$, we obtain
\[
        |\nabla_\Sigma^2 d_{\partial M}|^2 = |A|^2(1-g^2)+2\theta\langle A,B\rangle+|B|^2.
\]
The classical Bochner formula on $\Sigma$ gives
\[
        \Delta_\Sigma\frac{g^2}{2} = |\nabla_\Sigma^2 d_{\partial M}|^2 + \langle\nabla_\Sigma\Delta_\Sigma d_{\partial M},X\rangle + \operatorname{Ric}_\Sigma(X,X).
\]
Since $\Delta_\Sigma(g^2/2)=g\Delta_\Sigma g+|\nabla_\Sigma g|^2$, we get
\[
        g\Delta_\Sigma g = |\nabla_\Sigma^2 d_{\partial M}|^2 + \langle\nabla_\Sigma\Delta_\Sigma d_{\partial M},X\rangle + \operatorname{Ric}_\Sigma(X,X) - |\nabla_\Sigma g|^2.
\]
The contracted Gauss equation gives, since $H=0$,
\[
        \operatorname{Ric}_\Sigma(X,X) = \operatorname{Ric}_{\mathrm{g}}(X,X) - \Riem_{\mathrm{g}}(X,\nu,X,\nu) - |A(\cdot, X)|^2.
\]
Combining this with the expansion of $|\nabla_\Sigma^2 d_{\partial M}|^2$ yields the exact Bochner identity.

Now let $\gamma>0$. From the identity $\nabla_\Sigma g=\theta A(\cdot, v)+B(\cdot, v)$ and Young's inequality, we obtain 
\[
|\nabla_\Sigma g|^2 \leq (1+\gamma)(1-g^2)|A(\cdot, v)|^2 + C(\gamma) |B|^2.
\]
By definition, $|A(\cdot, X)|^2=g^2|A(\cdot, v)|^2$, so that 
\[
|\nabla_\Sigma g|^2 + |A(\cdot, X)|^2 \leq (1+\gamma(1-g^2))|A(\cdot, v)|^2 + C(\gamma) |B|^2.
\]
Since $H=0$, $A$ is trace-free. Hence, for every unit vector $v$, we have $|A(\cdot, v)|^2\leq \frac{n-1}{n}|A|^2$.  Therefore,
\[
|\nabla_\Sigma g|^2 + |A(\cdot, X)|^2 \leq \left(\frac{n-1}{n}+\gamma\right)|A|^2 + C(\gamma) |B|^2.
\]
Again using Young's inequality, we also bound
\[
2|\theta\langle A,B\rangle| \leq \gamma |A|^2\theta^2+C(\gamma) |B|^2 = \gamma |A|^2(1-g^2)+C(\gamma) |B|^2,
\]
so that, after substituting into the exact Bochner identity, we obtain
\[
\begin{aligned}
	g\Delta_\Sigma g
	&\geq (1-\gamma)|A|^2(1-g^2) - \left(\frac{n-1}{n}+\gamma\right)|A|^2 - C(\gamma) |B|^2 \\
	&\quad 
	+\langle\nabla_\Sigma\Delta_\Sigma d_{\partial M},X\rangle
	+\operatorname{Ric}_{\mathrm{g}}(X,X)-\Riem_{\mathrm{g}}(X,\nu,X,\nu).
\end{aligned}
\] 
Finally, since $\Sigma$ is minimal, $\Delta_\Sigma d_{\partial M} = \operatorname{tr}_\Sigma \nabla^2_{\mathrm{g}} d_{\partial M} = \Delta_{\mathrm{g}} d_{\partial M} - \nabla^2_{\mathrm{g}} d_{\partial M}(\nu,\nu)$.
Differentiating this expression tangentially gives terms involving
$\nabla^2_{\mathrm{g}}d_{\partial M}$, $\nabla^3_{\mathrm{g}}d_{\partial M}$, and one derivative of $\nu$, hence one factor of $|A|$, namely
\[
	\left| \langle\nabla_\Sigma\Delta_\Sigma d_{\partial M},X\rangle \right| \leq (C(n) |\nabla^3_\mathrm{g}d_{\partial M}|+2 |A||\nabla^2_\mathrm{g} d_{\partial M}|)g.
\]
By Young's inequality,
$2 |A||\nabla^2_\mathrm{g}d_{\partial M}|g\leq \gamma |A|^2g^2+C(\gamma)|\nabla^2_\mathrm{g}d_{\partial M}|^2$. The curvature terms can be estimated by
\[
\left| \operatorname{Ric}_{\mathrm{g}}(X,X)-\Riem_{\mathrm{g}}(X,\nu,X,\nu) \right|
\leq C(n)|\Riem_\mathrm{g}| g^2.
\]
Assuming \eqref{eq: geometric bounds boundary} holds, this gives the final estimate, since $0\leq g\leq 1$ by definition.
\end{proof}

Next we compute the normal derivative of the tilt function along $\partial M$.

\begin{lemma} \label{lemma: boundary absorption}
Let $\Sigma, M$, $d_{\partial M}$, and $g$ be as in \Cref{lemma: diff ineq tilt}. Let $\eta$ be the inward unit conormal to $\partial\Sigma$ in $\Sigma$. Then 
\[
	g=1\quad \text{and} \quad\partial_\eta g=0\quad \text{on $\partial\Sigma$}.
\]
\end{lemma}

\begin{proof}
By definition of a free-boundary minimal hypersurface, $\eta$ coincides with the normal $\bn$ to $\partial M$ along $\partial \Sigma$. Since $P_R$ is a relatively open subset of  $\partial M$ and $\partial \Sigma \subset P_R$, then $\eta=\nabla^\mathrm{g}d_{\partial M}|_{\partial \Sigma}$.

Hence, $\nabla^\mathrm{g}d_{\partial M}$ is tangential to $\Sigma$ along $\partial \Sigma$ and 
so, by definition, $\nabla^\mathrm{g}d_{\partial M}|_{\partial\Sigma}=\nabla_\Sigma d_{\partial M}|_{\partial\Sigma}$. In particular, then, for every $x\in \partial\Sigma$
\[
	g(x)=|\nabla_\Sigma d_{\partial M}(x)|=|\nabla^\mathrm{g}d_{\partial M}(x)|=1,
\]
i.e. $g|_{\partial \Sigma}\equiv 1$. Since $g^2+\theta^2=1$ everywhere on $\Sigma$, then $\theta|_{\partial\Sigma}\equiv 0$. Thus, on $\partial \Sigma$
\[
	\partial_\eta g=g\partial_\eta g=\frac{1}{2}\partial_\eta g^2=-\frac{1}{2}\partial_\eta \theta^2=-\theta \partial_\eta \theta=0
\]
as $\theta$ vanishes on $\partial \Sigma$.
\end{proof}

\subsection{Caccioppoli-type inequality}
We shall now prove the intrinsic (weak) Caccioppoli inequality that will be used later to implement the iteration scheme à la De Giorgi. 

\begin{lemma}[Caccioppoli inequality]\label{lemma: Caccioppoli inequality 1}
Let $\Sigma$ be a stable \emph{singular} free-boundary minimal hypersurface in $M$ and assume that  $\overline{\Sigma}\subset \Ncal^+_\delta(P_R)$. Assume that \eqref{eq: geometric bounds boundary} holds. Then there exists a positive dimensional constant $C(n)>0$ such that for every $k\in[0,1/(4n)]$ and every $\varphi\in C_c^{0,1}(\mathcal{N}^+_{\rho_0}(P_R))$, we have
\begin{equation}\label{eq: Caccioppoli smooth A 1}
	\begin{aligned}
		\frac{1}{2n}
		\int_{\{g>k\}}
		|A|^2
		\left(1-\frac{k}{g}\right)\varphi^2
		&\leq
		C(n)\left(\int_\Sigma (g-k)_+^2|\nabla_\Sigma\varphi|^2 + \Lambda^2\int_{\{g>k\}}\varphi^2\right).
	\end{aligned}
\end{equation}
Consequently,
\begin{equation}\label{eq: Caccioppoli smooth nabla 1}
	\begin{aligned}
		\int_{\{g>k\}}
		|\nabla_\Sigma g|^2
		\left(1-\frac{k}{g}\right)\varphi^2
		&\leq
		C(n)\left(\int_\Sigma (g-k)_+^2|\nabla_\Sigma\varphi|^2
		+ \Lambda^2\int_{\{g>k\}}\varphi^2\right).
	\end{aligned}
\end{equation}
\end{lemma}

\begin{proof}
We shall first prove \eqref{eq: Caccioppoli smooth A 1} and \eqref{eq: Caccioppoli smooth nabla 1} for test functions $\varphi\in C^{0,1}_c(\Sigma)$, and then extend to compactly supported ambient Lipschitz functions by an excision argument, similar to the one in \Cref{lemma: trace inequality}.

Let $\Omega_k\coloneqq \{g>k\}$. By \Cref{lemma: stability absorbed 1} for every compactly supported Lipschitz function $\zeta$ on $\Sigma$ and a fixed $\alpha>0$ to be determined later, we have
\[
	\int_{\Sigma}|A|^2\zeta^2\leq (1+\alpha)\int_\Sigma|\nabla_\Sigma\zeta|^2+C(n,\alpha)\Lambda^2\int_\Sigma \zeta^2
\]
where $C(n,\alpha)$ only depends on $n, \alpha$.
Let now $\zeta=(g-k)_+\varphi$ to get
\[
	\int_{\Sigma}|A|^2(g-k)_+^2\varphi^2\leq (1+\alpha)\int_\Sigma|\nabla_\Sigma((g-k)_+\varphi)|^2+C(n,\alpha)\Lambda^2\int_\Sigma (g-k)_+^2\varphi^2.
\]
After integration by parts, we have
\[
	\int_\Sigma |\nabla_\Sigma((g-k)_+\varphi)|^2 = -\int_{\Omega_k}(g-k)\Delta_\Sigma g\,\varphi^2 + \int_\Sigma (g-k)_+^2|\nabla_\Sigma\varphi|^2, 
\]
and the boundary term vanishes by \Cref{lemma: boundary absorption}. Hence,
\[
\begin{aligned}
	\int_{\Omega_k}|A|^2(g-k)^2\varphi^2 &\leq -(1+\alpha)\int_{\Omega_k}(g-k)\Delta_\Sigma g\,\varphi^2+(1+\alpha)\int_\Sigma(g-k)^2_+|\nabla_\Sigma\varphi|^2\\
	&\quad +C(n,\alpha)\Lambda^2\int_\Sigma(g-k)^2_+\varphi^2.
\end{aligned}
\]
Notice that on $\Omega_k$, one has
\[
	(g-k)\Delta_\Sigma g=\left(1-\frac{k}{g}\right)g\Delta_\Sigma g,
\]
so we can use the differential inequality for the tilt in \Cref{lemma: diff ineq tilt} to obtain for $\gamma>0$ to be determined later
\[
	\begin{aligned}
		-(g-k)\Delta_\Sigma g \leq -\left(1-\frac{k}{g}\right)\left[\left((1-\gamma)(1-g^2)-\frac{n-1}{n}-\gamma-\gamma g^2\right)|A|^2-C(n,\gamma)\Lambda^2\right].
	\end{aligned}
\]
Thus, after moving all the terms containing $|A|^2$ to the left hand side and using the obvious inequality $0\leq (g-k)_+\leq 1$, we get 
\[
	\int_{\Omega_k}\mathcal{Q}_{\alpha,\gamma}|A|^2\varphi^2 \leq  (1+\alpha)\int_\Sigma(g-k)^2_+|\nabla_\Sigma\varphi|^2+(C(n,\alpha)+C(n,\gamma))\Lambda^2\int_{\Omega_k}\varphi^2,
\]
where 
\[
\mathcal{Q}_{\alpha,\gamma}\coloneqq (g-k)^2+(1+\alpha)\left(1-\frac{k}{g}\right)\left((1-\gamma)(1-g^2)-\frac{n-1}{n}-\gamma-\gamma g^2\right).
\]
By direct computation,
\[
	\mathcal{Q}_{\alpha,\gamma} = \left(1-\frac{k}{g}\right)\left(\frac{1+\alpha}{n}-kg-\alpha g^2-2(1+\alpha)\gamma\right).
\]
We can choose $\alpha=\alpha_n\coloneqq 1/(8n)$ and $\gamma=\gamma_n\coloneqq 1/(16n(1+\alpha_n))$ to obtain
\[
	\frac{1+\alpha}{n}-kg-\alpha g^2-2(1+\alpha)\gamma\geq \frac{1}{2n}
\]
since $k\in[0, 1/(4n)]$ and $0\leq g\leq 1$ by definition. Inequality \eqref{eq: Caccioppoli smooth A 1} then follows.

As for \eqref{eq: Caccioppoli smooth nabla 1}, recall by \Cref{lemma: diff ineq tilt} on the open set $\{g>0\}$ we have $\nabla_\Sigma g=\theta A(v,\cdot)+B(v,\cdot)$ for the unit vector $v=g^{-1}X$. Therefore, on $\Sigma\subset \Ncal^+_{\rho_0}(P_R)$ 
\[
	|\nabla_{\Sigma}g|^2\leq 2\theta^2|A|^2+2|B|^2\leq 2|A|^2+2|\nabla^2_\mathrm{g}d_{\partial M}|^2\leq 2|A|^2+2\Lambda^2,
\]
and \eqref{eq: Caccioppoli smooth nabla 1} then follows from \eqref{eq: Caccioppoli smooth A 1}.

Finally, apply \Cref{cor: cap} on a compact neighbourhood of $\spt \varphi$: there exist Lipschitz cutoff functions $\chi_j$, each supported away from $\Scal$, such that $0\leq \chi_j\leq 1$, $\chi_j\to 1$ locally on $\Sigma$, and $\int_{\Sigma}|\nabla_\Sigma\chi_j|^2\to 0$.
We can then apply the smooth inequality \eqref{eq: Caccioppoli smooth A 1} with test function $\chi_j\varphi$ to get
\[
	\begin{aligned}
		\frac{1}{2n} \int_{\{g>k\}} |A|^2 \left(1-\frac{k}{g}\right)\chi_j^2\varphi^2 \leq C(n) \int_\Sigma (g-k)_+^2|\nabla_\Sigma(\chi_j\varphi)|^2 + C(n)\Lambda^2\int_{\{g>k\}}\chi_j^2\varphi^2.
	\end{aligned}
\]
Since
\[
	|\nabla_\Sigma(\chi_j\varphi)|^2 \leq 2\chi_j^2|\nabla_\Sigma\varphi|^2 + 2\varphi^2|\nabla_\Sigma\chi_j|^2,
\]
and  $0\leq (g-k)_+\leq 1$, the contribution of the second term is bounded by $\|\varphi\|_{L^\infty}^2 \int_{\Sigma}|\nabla_\Sigma\chi_j|^2$,
which tends to zero. All other terms pass to the limit by Fatou's lemma, the dominated convergence theorem, and the local area bound, cf. \Cref{rmk: area bounds}. Thus, letting $j\to\infty$, we obtain the asserted estimate with $|A|^2$ on the left-hand side. Then \eqref{eq: Caccioppoli smooth nabla 1} follows exactly as above. 
\end{proof}

\subsection{Proof of \Cref{thm: boundary eps reg}}\label{sec: DG iteration boundary}
We are now able to prove \Cref{thm: boundary eps reg} by applying the iteration lemma in \Cref{sec: abstract DG}.

Let us assume that \eqref{eq: geometric bounds boundary} holds. Define $K_0(n, \Lambda)\coloneqq \max\{K_{\mathrm{Sob}}(n, \Lambda), K_{\mathrm{Cac}}(n, \Lambda)\}$, and set
\begin{equation}\label{eq: definition R0 boundary}
	\delta_0\coloneqq \frac{1}{2}\rho_0,\quad R_0\coloneqq \frac{1}{2}\min\bigl\{\rho_0, K_0^{-1}\bigr\}.
\end{equation}
Let $\Sigma$ be a stable singular free-boundary minimal hypersurface in $M$ such that $\overline{\Sigma}\subset\Ncal^+_{\delta}(P_R)$ for some $\delta\in(0, \delta_0), R\in (0, R_0)$.

Then, for $0<r\leq R$ define the \emph{Fermi cylinder} $C_r$ by
\begin{equation}\label{eq: definition Fermi cylinder}
	C_r\coloneqq F_{\partial M}\bigl(P_r\times[0, 2\delta)\bigr)=\Ncal^+_{2\delta}(P_r),
\end{equation}
and consider the following \textit{tilt excess} quantity: 
\begin{equation}
	\mathrm{E}_{\Sigma, P}(r)\coloneqq \left(r^{-n}\int_{\Sigma\cap C_r}g^2\right)^{1/2}
\end{equation}
where $g$ is the \emph{tilt function} in \Cref{def: tilt}.

Let us apply \Cref{lemma: abstract DG} with  $u=g$ and $k_*=1/(4n)$. Notice that \Cref{lemma: Caccioppoli inequality 1} provides the necessary Caccioppoli-type inequality \eqref{eq: abstract Caccioppoli}. 
Therefore, by \Cref{cor: abstract DG cor}, there exist $\varepsilon_0=\varepsilon_0(n, \Lambda)>0$ and $C_0=C_0(n, \Lambda)>0$ such that 
\[
	\sup_{\Sigma \cap C_{\sigma R}}g\leq C_0(1-\sigma)^{-\frac{n}{n+2}}\,\mathrm{E}_{\Sigma, P}(R)^{\frac{2}{n+2}}
\]
provided $\mathrm{E}_{\Sigma, P}(R)^2\leq \varepsilon_{0}^2(1-\sigma)^n$, which is the desired statement.

\subsection{Proof of \Cref{thm: boundary sheeting}}\label{sec: proof boundary sheeting}
We shall now prove our first sheeting theorem, \Cref{thm: boundary sheeting}. The proof is similar to \cite[Section 5]{Bellettini}, to which we shall often refer.

Assume $\overline{\Sigma}\subset\Ncal^+_\delta(P_R)=F_{\partial M}(P_R\times[0, \delta))$, and recall that $F_{\partial M}$ is a diffeomorphism $P_R\times[0, \delta)\to \Ncal^+_\delta(P_R)$. Let the inverse map $F_{\partial M}^{-1}: \Ncal^+_\delta(P_R)\to P_R\times[0, \delta)$ be given by $(\pi_{\partial M}, d_{\partial M})|_{\Ncal^+_\delta(P_R)}$, where $\pi_{\partial M}$ is the nearest point projection map onto $\partial M$.

By \Cref{thm: boundary eps reg} with $\sigma=3/4$, we have that 
\begin{equation}\label{eq: smalleness assumption}
	\sup_{\Sigma\cap C_{3R/4}} g \leq C_0(n, \Lambda)\mathrm{E}_{\Sigma, P}(R)^\frac{2}{n+2}.
\end{equation}
By decreasing $\varepsilon_0$ if necessary, we may assume $\sup_{\Sigma\cap C_{3R/4}} g \leq \frac{1}{4}$.

\textit{Step 0: $\partial\Sigma\cap P_{3R/4}=\emptyset$.}  Suppose $x\in \partial\Sigma\cap P_{3R/4}$. By definition, the inner conormal $\eta(x)$ of $\partial \Sigma$ at $x$ equals $\bn(x)=\nu_P(x)$. Hence, $\nu_P(x)\in T_x\Sigma$ and $g(x)=|\mathrm{proj}^{T_x\Sigma}_\mathrm{g}\nabla^\mathrm{g}d_{\partial M}(x)|=|\mathrm{proj}^{T_x\Sigma}_\mathrm{g}\nu_P(x)|=1$, contradicting \eqref{eq: smalleness assumption}.

\textit{Step 1: Local graphicality.} Let $\iota:N\to M\setminus\mathcal{S}$ parametrise $\Sigma$. At every point of $\Ncal^+_\delta(P_R)$, 
\begin{equation}\label{eq: kernel projection}
	\operatorname{Ker}D\pi_{\partial M}=\operatorname{span}\{\nabla^{\mathrm{g}}d_{\partial M}\}.
\end{equation}
By \eqref{eq: smalleness assumption} and \eqref{eq: kernel projection}, the differential of $\pi_{\partial M}\circ\iota$ is invertible at every $z\in N$ such that $\iota(z)\in C_{3R/4}$. Thus the argument in \cite[Proof of Theorem 6]{Bellettini} shows that there is a neighbourhood of $z$ in $N$ the image of which is a smooth Fermi graph over an open subset of $P_{3R/4}$.

\textit{Step 2: Finite covering away from the projected singular set.} Let 
\[
	\Gamma\coloneqq \pi_{\partial M}(\Scal\cap \overline{\Ncal^+_{\delta}}(P_{3R/4}))\subset \bar P_{3R/4}, \quad B\coloneqq P_{3R/4}\setminus \Gamma.
\]
 Since $\Hcal^{n-2}(\Scal\cap \overline{\Ncal^+_{\delta}}(P_{3R/4}))<\infty$ and $\pi_{\partial M}$ is a Lipschitz map in $\Ncal^+_{\rho_0}(P_R)$, then $\Hcal^{n-1}(\Gamma)=0$. Since $\Scal$ is closed, it is easy to see that $\Gamma$ is relatively closed in $P_{3R/4}$. Then $B$ is open in $P_{3R/4}$ and path connected (see \cite[Proof of Theorem 5]{Bellettini}).
Define 
\[
	N_0\coloneqq \{z\in N: \iota(z)\in C_{3R/4},\ \pi_{\partial M}(\iota(z))\in B\}
\]
and $\hat\pi\coloneqq \pi_P\circ\iota|_{N_0}:N_0\to B$. By \textit{Step 1}, $\hat \pi$ is a local diffeomorphism. Because of properness of $\iota$ as an immersion into $M\setminus \Scal$, one can check that $\hat \pi$ is also proper. Since $B$ is connected, each connected component of $N_0$ is therefore mapped onto the whole of $B$. When restricted to each connected component of $N_0$, $\hat\pi$ is a proper local diffeomorphism between connected smooth manifolds, so it is a smooth covering map by \cite[Proposition 4.46]{LeeSmooth}. In particular, for each $y\in B$, $\hat\pi^{-1}(y)$ is finite. Its cardinality is a locally constant function on $B$, so it is constant as $B$ is connected. Let this cardinality be $Q$. 

\textit{Step 3: Lipschitz estimate and extension across $\Gamma$.} By \textit{Step 2}, for each $y\in B$ the normal fibre $ F(\{y\}\times[0,\delta))$ has exactly $Q$ preimages in $N_0$.  Order the heights of their images under $\iota$ counting repetitions, and write $ 0\leq u_1(y)\leq u_2(y)\leq\cdots\leq u_Q(y),$ for $y\in B$, where the functions $u_j$ are locally Lipschitz. Then, 
\begin{equation}\label{eq:bs-graph-decomposition-on-B}
	\Sigma\cap\pi_{\partial M}^{-1}(B)
	=\cup_{j=1}^{Q}F_{\partial M}\bigl(\operatorname{graph}(u_j)\bigr).
\end{equation}

Write $F_{\partial M}^{*}\mathrm{g} = dt^{2}+h_{t}$, where $h_t$ is a smooth family of metrics on $P_R$ with $h_0=\mathrm{g}|_{P_R}$. It is easy to check that for a smooth graph $G[u]=F_{\partial M}\bigl(\{(y, u(y)):y\in P_R\}\bigr)$ on $P_R$ with $|\nabla_{G[u]}d_{\partial M}|\leq c\in(0,1)$, $|\nabla_{\partial M}u|\leq C_{\mathrm{Fermi}}(n, \Lambda)c(1-c^2)^{-1/2}$, where $C_{\mathrm{Fermi}}$ controls the uniform equivalence of the metrics $(h_t)_{t\in[0,\delta]}$ and depends on $n$ and $\Lambda$ by \eqref{eq: geometric bounds boundary}.

Since we can approximate curves in $P_{3R/4}$ by curves avoiding $\Gamma$, we conclude that 
\[
\operatorname{Lip}_{\partial M}(u_j;B)\leq 2C_{\mathrm{Fermi}}\sup_{\Sigma\cap C_{3R/4}} g\leq C(n, \Lambda)\mathrm{E}_{\Sigma, P}(R)^\frac{2}{n+2}.
\]

By density of $B$ in $P_{3R/4}$, every $u_j$ extends uniquely to a Lipschitz function (still denoted $u_j$) on the whole of $P_{3R/4}$, with the same bound on the Lipschitz constant, i.e.
\begin{equation}\label{eq: Lipschitz bound}
	\operatorname{Lip}_{\partial M}(u_j)\leq C(n,\Lambda)\mathrm{E}_{\Sigma, P}(R)^\frac{2}{n+2}.
\end{equation}
By density of the regular set in $\overline{\Sigma}$ and of $B$ in $P_{3R/4}$, we obtain 
\[
	\overline{\Sigma}\cap \Ncal^+_{\delta}(P_{3R/4}) = \cup_{j=1}^Q F_{\partial M}\bigl(\operatorname{graph}(u_j)\bigr)
\]
for the extended functions $u_j$.

We can now prove (i) and (ii).
\begin{enumerate}[(i)]
	\item Assume $\Hcal^{n-2}(\Scal)=0$. Since $\pi_{\partial M}$ is Lipschitz on $\Ncal^+_{\rho_0}(P_R)$, then necessarily $\Hcal^{n-2}(\Gamma)=0$. By \cite[Appendix A]{SimonWicFrequency}, then $B=P_{3R/4}\setminus\Gamma$ is simply connected. Therefore every connected component of the covering $\hat \pi:N_0\to B$ is one-sheeted. The components determine $Q$ globally defined, not necessarily ordered Lipschitz functions $v_1, \dots, v_Q:B\to [0, \delta)$ such that 
	\begin{equation*}
		\Sigma \cap \pi_{\partial M}^{-1}(B) = \cup_{j=1}^Q F_{\partial M}\bigl(\operatorname{graph}(v_j)\bigr).  
	\end{equation*}
	 Arguing as above, for every $j \in {1, 2, \ldots, Q}$, we can extend $v_j$ to a Lipschitz function defined on the whole of $P_{3R/4}$, preserving the Lipschitz constant. 
	 Clearly, $\Hcal^{n-1}(\Gamma)=0$. Since $\partial\Sigma\cap P_{3R/4}=\emptyset$, the Lipschitz functions $v_j$ are weak solutions of the minimal surface equation on the bounded open set $B=P_{3R/4}\setminus \Gamma$ with a uniform bound on the Lipschitz constant. Since the metrics $(h_t)_{t\in[0,\delta]}$ are uniformly equivalent, in local coordinates the area integrand $\mathcal{L}(y, z, q)=(\det h_z(y))^{1/2}(1+h_z^{ij}(y)q_iq_j)^{1/2}$ is uniformly elliptic. Therefore, the removal of singularities result in \Cref{lemma: removabillity} applies, and $\mathcal{S}\cap \Ncal^+_{\delta}(P_{3R/4})=\emptyset$.
	 
	 Finally, the gradient bound $\|\nabla v_j\|_{C^{1,\alpha}(P_{R/2})}$ follows from the combination of De Giorgi--Nash--Moser theory and Schauder theory.
	
	 \item Assume instead $\Sigma$ is embedded. Then for every $y\in B$ the points of $\Sigma$ lying in the normal fibre are distinct, i.e. $0\leq u_1<\dots<u_Q$ on $B$. Therefore, each $u_j$ is smooth on $B$. Since $\Hcal^{n-1}(\Gamma)=0$, as above by \Cref{lemma: removabillity} each $u_j$ extends to a smooth solution of the minimal surface equation on the whole of $P_{3R/4}$. Weak ordering is preserved, and, indeed, strict ordering by the strong maximum principle for the minimal surface equation. Thus $\overline{\Sigma}\cap \Ncal^+_{\delta}(P_{3R/4})$ is a union of pairwise disjoint smooth graphs. In particular $\Scal \cap \Ncal^+_{\delta}(P_{3R/4})=\emptyset$.
	 The gradient estimates then follow from standard elliptic theory as in (i).
\end{enumerate}

\section{Boundary-orthogonal sheeting} \label{sec: fb sheeting}
\subsection{Geometric setup}\label{sec: setup FB}

Let $M$ be an $(n+1)$-dimensional Riemannian manifold-with-boundary as in \Cref{sec: preliminaries}, and assume $\partial M\neq \emptyset$. 
Let $P$ be a (smooth) two-sided, embedded, $n$-dimensional manifold-with-boundary in $M$ that is diffeomorphic to a half-disc and whose boundary lies in $\partial M$. In other words, $P$ is the image of a smooth embedding $e$ from $\mathbb{B}^n_+(0)=\{x\in \mathbb{R}^{n}: |x|<1, x_{n}\geq 0\}$ into $M$ and $\partial_0 P=e(\partial_0\mathbb{B}^n_+(0)) \subset \partial M$, where $\partial_0\mathbb{B}^n_+(0)=\{x\in \mathbb{R}^{n}: |x|<1, x_{n}=0\}$. We shall further assume that $P\cap \partial M=\partial_0P$ and $P$ meets $\partial M$ orthogonally along $\partial_0 P$. Let $\nu_P$ denote a smooth choice of unit normal on $P$. Then by orthogonality to $\partial M$ along $\partial_0P$, we have $\mathrm{g}(\nu_P, \bn)=0$ on $\partial_0P$.

Let $p\in \partial_0 P$, and, for $0<R<\frac{1}{2}\inj_p(P)$, let $P_R$ denote the intrinsic half-ball $B^P_{R}(p)$ in $P$ centred at $p$.

\begin{definition}[Adapted $\rho$-neighbourhood]\label{def: adapted neighbourhood}
	Let $\rho>0$. An \emph{adapted $\rho$-neighbourhood} of $P_R$ in $M$ consists of a relatively open neighbourhood $U\subset M$ of $P_{2R}$ and a smooth function $z_P:U\to \mathbb{R}$ with the following properties:
	\begin{enumerate}[(i)]
		\item $d_{\partial M}$ is smooth in $U$;
		\item $P\cap U=\{z_P=0\}$ and $\nabla^\mathrm{g} z_P$ is a positive multiple of $\nu_P$ on $P\cap U$;
		\item $\frac{1}{2}\leq|\nabla^\mathrm{g}z_P|\leq 2$ in $U$;
		\item $\langle \nabla^\mathrm{g}z_P, \bn\rangle=0$ in $U\cap \partial M$.
        \item the flow $\Phi^Z_t$ of the vector field $Z=\frac{\nabla^\mathrm{g} z_P}{|\nabla^\mathrm{g}z_P|}$ defines a diffeomorphism $F_Z$ of $P_{2R}\times(-2\rho, 2\rho)$ onto a relatively open (compactly contained in $U$) neighbourhood of $P_{2R}$ in $M$ by
		\begin{equation}
		F_Z(y, t)\coloneqq \Phi_t^Z(y)\quad \text{ for $(y,t)\in P_{2R}\times(-2\rho, 2\rho)$.}
		\end{equation}
	\end{enumerate}
\end{definition}

\begin{remark}\label{rmk: no gradient}
	The essential datum in an adapted neighbourhood is the vector field $Z$. One could instead prescribe a smooth unit vector field on a relatively open neighbourhood $U$ of $P_{2R}$, equal to $\nu_P$ on $P_{2R}$ and tangent to $\partial M$ along $\partial M\cap U$, whose flow satisfies condition~(v) of \Cref{def: adapted neighbourhood}. With the derivative bounds \eqref{eq: Z estimates}, the flow estimate \eqref{eq: derivative flow map}--\eqref{eq: estimate transversality}, and the ambient geometric bounds used below, the $\varepsilon$-regularity and sheeting arguments apply without requiring $Z$ to be a normalized gradient.
	
	In particular, the tilt and Caccioppoli estimates involve only the first two derivatives of $Z$, whereas deriving these bounds from $Z=\nabla^{\mathrm g}z_P/|\nabla^{\mathrm g}z_P|$ requires control of three derivatives of $z_P$. The function $z_P$ therefore provides a convenient construction, rather than an essential ingredient of the argument.
\end{remark}

We now show that for every $P$ as above, there exists $\rho>0$ such that the half disc $P_\rho$ has an adapted $\rho$-neighbourhood.

\begin{lemma}\label{lemma: adapted neighbourhood}
Let $P\subset M$ be as above and let $p\in\partial_0P$. Then there exists $0<\rho<\frac{1}{2}\operatorname{inj}_p(P)$ sufficiently small such that $P_\rho$ has an adapted $\rho$-neighbourhood $U$ in the sense of \Cref{def: adapted neighbourhood}. 
\end{lemma} 

\begin{proof}
Consider the smooth embedded relatively compact submanifold $\Gamma\coloneqq \partial_0P$ of $\partial M$. By \cite[Section 3.1]{Gray}, there exists $\varrho_0>0$ such that the normal exponential map $F_{\Gamma}$ of $\Gamma$ in $\partial M$ is a diffeomorphism of $\Gamma\times(-\varrho_0, \varrho_0)$ onto an open set $O_\Gamma$ in $\partial M$. Since $O_\Gamma$ has compact closure in $\partial M$, then, as mentioned in \Cref{sec: preliminaries}, after shrinking $\varrho_0$ if needed, the normal exponential map of $\partial M$ is a diffeomorphism of $O_\Gamma\times [0, \varrho_0)$ into a normal collar neighbourhood $U_\Gamma$ of $O_\Gamma$ in $M$. In particular, the composition of the two normal exponential maps gives a diffeomorphism $\Psi$ of $\Gamma\times (-\varrho_0, \varrho_0)\times [0, \varrho_0)$ onto $U_\Gamma\subset M$ by
\[
\Psi: (q, s, r)\mapsto F_{\partial M}\bigl(\operatorname{exp}^{\partial M}_q(s\nu_P(q)), r\bigr).
\]
Let $\widehat\Pi$ be the map $(q, s, r)\mapsto(q,r)$, and let $\Pi: P\cap U_\Gamma \to \Gamma\times[0,\varrho_0)$ be defined by $\Pi = \widehat\Pi\circ\Psi^{-1}$. 
Since $P$ meets $\partial M$ orthogonally along $\Gamma$, for all $q\in \Gamma$ we have $\Pi(q)=(q,0)$, and it is easy to check that $D\Pi_q|_{T_qP}$ is an isomorphism. By the implicit function theorem, there is $\varrho >0$ and an open set $U_\varrho\subset U_\Gamma$ such that for a neighbourhood $\Gamma_p$ of $p$ in $\Gamma$ we have a smooth function $\psi:\Gamma_p\times[0, \varrho)\to(-\varrho, \varrho)$ with the property that
\begin{equation}\label{eq: P graphical}
	P\cap U_\varrho=\{\Psi(q, \psi(q, r), r): q\in \Gamma_p, r\in[0, \varrho) \}, \quad \psi(q, 0)=0, \quad \partial_r\psi(q, 0)=0. 
\end{equation}
Define the function $z_P$ on $U_\varrho$ by $z_P(\Psi(q, s, r))\coloneqq s-\psi(q, r)$. We shall slightly abuse notation and write $z_P(q,s,r)=s-\psi(q,r)$.
Notice that on $U_\varrho \cap \partial M$ (corresponding to $\{r=0\}$) we have $\partial_r z_P=-\partial_r \psi(q, 0)=0$, so $\langle\nabla z_P, \bn\rangle=0$ on  $U_\varrho \cap \partial M$. Moreover, $\partial_sz_P=1$ everywhere, so that $\mathrm{d}z_P\neq 0$. Since $z_P=0$ on $P\cap U_\varrho$ and $\mathrm{d}z_P\neq 0$, then $\nabla^\mathrm{g}z_P$ is normal to $P$ on $P\cap U_\varrho$. 
Also, by construction, $|\nabla^\mathrm{g}z_P|=1$ on $U_\varrho \cap \partial M$, so that, by shrinking $\varrho$ if necessary, we have $1/2\leq |\nabla^\mathrm{g}z_P|\leq 2$ on $U_\varrho$.

By the flowout theorem (see \cite[Theorem 9.20 and Theorem 9.34]{LeeSmooth}), there is $\rho\in(0, \varrho)$ such that (v) in \Cref{def: adapted neighbourhood} is satisfied with $U=U_\varrho$ and $R=\rho$. 
\end{proof}

\begin{remark}
	For the compactness argument in \Cref{subsec: top stratum}, it suffices to consider reference half-discs obtained by restricting $F_{\partial M}(\Gamma\times[0,\varrho))$ to a neighbourhood of a point of $\Gamma$. In the coordinates of the preceding proof, these are given by $s=0$, so that $\psi\equiv0$ and $z_P=s$.
	By choosing $T_p\Gamma$ appropriately, such reference half-discs realize every tangent half-hyperplane at $p$ orthogonal to $\partial M$. We nevertheless establish sheeting for arbitrary smooth two-sided reference half-discs meeting $\partial M$ orthogonally, since the argument applies in this greater generality.
\end{remark}

Assume $\rho_1>0$ is such that $U_1$ is an adapted $\rho_1$-neighbourhood of $P_{\rho_1}\subset P$ as in Lemma \ref{lemma: adapted neighbourhood}. Thus, for $\delta\in(0,2\rho_1]$, $R\in(0, 2\rho_1]$, we can define
\begin{equation}\label{eq: definition NcalZ}
	\Ncal^Z_\delta(P_R)\coloneqq F_Z\bigl(P_R\times(-\delta,\delta)\bigr).
\end{equation}
Let the inverse map $F^{-1}_Z$ be given by $F^{-1}_Z=(\pi_Z, \mathcal{t}_Z)$.

\begin{remark}\label{rmk: boundary preserved}
	Notice that, since $Z$ is tangential to $\partial M$ along $\partial M\cap U_1$ and does not vanish anywhere, then the flow preserves $\partial M$ and the interior of $M$, i.e. $F_Z\bigl(\partial_0P_{2\rho_1}\times(-2\rho_1, 2\rho_1)\bigr)=\Ncal^Z_{2\rho_1}(P_{2\rho_1})\cap \partial M$.
\end{remark}

Let $\Lambda\geq 0$ be such that
\begin{equation}\label{eq: geometric bounds orthogonal}
\begin{aligned}
	&\sup_{ P_{2\rho_1}}|A_P|^2+\sup_{P_{2\rho_1}}|\nabla_P A_P|+\sup_{ P_{2\rho_1}}|\nabla^2_PA_P|^{2/3}\leq \Lambda^2,\\  
	&\sup_{U_1}|\nabla^2_\mathrm{g}d_{\partial M}|^2+\sup_{U_1}|\nabla^3_\mathrm{g}d_{\partial M}|+\sup_{U_1}|\nabla^4_\mathrm{g}d_{\partial M}|^{2/3}\leq \Lambda^2,\\
	&\sup_{U_1}|\Riem_\mathrm{g}|+\sup_{U_1}|\nabla^{\mathrm{g}}\Riem_\mathrm{g}|^{2/3}+\sup_{U_1}|\nabla^2_{\mathrm{g}}\Riem_\mathrm{g}|^{1/2}\leq \Lambda^2,
\end{aligned}
\end{equation}
where $A_P$ denotes the second fundamental form of $P$ in $M$.

\begin{lemma}\label{lemma: quantitative adapted neighbourhood}
	Let $P$, $\rho_1$, $U_1$, $z_P$ and $Z$ be chosen as above, and suppose that \eqref{eq: geometric bounds orthogonal} holds. There exists a constant $\delta_1=\delta_1(n, \Lambda, \rho_1)>0$ such that for all $\delta, R\in(0, \delta_1)$ the following estimates hold on $\Ncal^Z_{2\delta}(P_{2R})$:
	\begin{equation}\label{eq: zP estimates}
		|\nabla^2_\mathrm{g}z_P|\leq C(n)\Lambda, \quad |\nabla^3_\mathrm{g}z_P|\leq C(n)\Lambda^2,
	\end{equation}
	\begin{equation}\label{eq: Z estimates}
		|\nabla_\mathrm{g}Z|\leq C(n)\Lambda, \quad |\nabla^2_\mathrm{g}Z|\leq C(n)\Lambda^2.
	\end{equation}
	Moreover, the flow map $F_Z$ restricted to $P_{2R}\times(-2\delta, 2\delta)$ is a diffeomorphism onto $\Ncal^Z_{2\delta}(P_{2R})$ and it satisfies
	\begin{equation}\label{eq: derivative flow map}
		|DF_Z|+|D(F_Z^{-1})|\leq C(n).
	\end{equation}
	Finally, if $(y,t)\in P_{2R}\times(-2\delta, 2\delta)$ and $x=F_Z(y, t)$, then
	\begin{equation}\label{eq: estimate transversality}
		\begin{aligned}
			\sup_{v\in T_yP, |v|=1}\bigl|\langle Z(x), (D_yF_Z)_{(y, t)} v\rangle\bigr| \leq C(n) \bigl(\sup|\nabla^\mathrm{g}Z|\bigr) |t|\leq C(n)\Lambda |t|.
		\end{aligned}
	\end{equation}
\end{lemma}

\begin{proof}
Recall the construction in the proof of \Cref{lemma: adapted neighbourhood}, and the definition of the map $\Psi:\partial_0P_{2\rho_1}\times(-2\rho_1, 2\rho_1)\times [0, 2\rho_1)\to U_1$ given by
\[
\Psi: (q, s, r)\mapsto F_{\partial M}\bigl(\operatorname{exp}^{\partial M}_q(s\nu_P(q)), r\bigr).
\]
Let $(\overline{\mathrm{g}}_{ij})_{1\leq i,j\leq n+1}$ denote the pullback metric tensor $\Psi^*\mathrm{g}$ and let $\theta=\theta(n, \Lambda, \rho_1)\in(0,1/2)$ be a constant to be determined in the proof; then let $\rho_\theta\coloneqq 2\theta\rho_1$. By the estimates of \cite[Appendix A]{Li2021} and the Riccati equation, there exists a constant $c_0(n)$ such that, provided we choose $\theta>0$ small enough that $\rho_\theta\Lambda\leq c_0(n)$, then 
\begin{equation}\label{eq: estimates Fermi metric}
	\begin{aligned}
	\frac{1}{2}\mathrm{g}_{\mathrm{Euc}}\leq \overline{\mathrm{g}}\leq 2\mathrm{g}_{\mathrm{Euc}},\quad |\Gamma(\overline{\mathrm{g}})|\leq C(n)\Lambda,\quad |\partial\Gamma(\overline{\mathrm{g}})|\leq C(n)\Lambda^2
\end{aligned}
\end{equation}
on $\partial_0P_{\rho_\theta}\times(-\rho_\theta, \rho_\theta)\times [0, \rho_\theta)$, where by $\Gamma(\overline{\mathrm{g}})$ we denote the Christoffel symbols of the metric $\overline{\mathrm{g}}$ and with $\partial \Gamma(\overline{\mathrm{g}})$ the collection of their partial derivatives with respect to the coordinate vector fields. Moreover, since $D\psi=0$ on $\partial_0P_{\rho_1}$, by shrinking $\theta$ if needed, we can assume $|D\psi|\leq 1$ on  $\partial_0P_{\rho_\theta}\times [0, \rho_\theta)$.
Therefore, using the graph representation in \eqref{eq: P graphical} together with the equation for the components of the second fundamental form of a graph, we obtain 
\begin{equation}\label{eq: D2psi}
	|D^2\psi|\leq C(n)\bigl(\textstyle \sup_{P_{2\rho_1}}|A_P|+|\Gamma(\overline{\mathrm{g}})|\bigr)\leq C(n)\Lambda
\end{equation}
on $\partial_0P_{\rho_\theta}\times [0, \rho_\theta)$.
Finally, differentiating the equation for the second fundamental form components in the tangential directions, we obtain 
\begin{equation}\label{eq: D3psi}
	|D^3\psi|\leq C(n)\bigl(\textstyle\sup_{P_{2\rho_1}}|\nabla_PA_P|+|\partial\Gamma(\overline{\mathrm{g}})| +\bigl(\sup_{P_{2\rho_1}}|A_P|+|\Gamma(\overline{\mathrm{g}})|+|D^2\psi|\bigr)^2\bigr)\leq C(n)\Lambda^2
\end{equation}
after shrinking $c_0(n)>0$ further, if necessary.

Recall in the proof of \Cref{lemma: adapted neighbourhood}, we defined the function $z_P$ by $z_P(\Psi(q, s, r))\coloneqq s-\psi(q, r)$. 
By differentiating the definition and using \eqref{eq: D2psi}--\eqref{eq: D3psi}, it is immediate to obtain the higher order estimates \eqref{eq: zP estimates} in $\Ncal_{\rho_\theta}(P_{\rho_\theta})$.
Finally, by definition of $Z$, for every vector $X$, 
\begin{equation}\label{eq: derivative Z}
	\nabla^\mathrm{g}_XZ=|\nabla^\mathrm{g}z_P|^{-1}\left(\mathrm{Id}-Z\otimes Z\right)\left(\nabla^2_{\mathrm{g}}z_P(X, \cdot)\right)^\sharp.
\end{equation}
Therefore, \eqref{eq: Z estimates} follows from \eqref{eq: derivative Z} and \eqref{eq: zP estimates}.

By standard ODE theory, bounds on $|Z|, |\nabla^\mathrm{g}Z|, |\nabla^2_\mathrm{g}Z|$, together with \eqref{eq: geometric bounds orthogonal}, control the flow map $F_Z$ in $C^2$. In particular, on $P_{\rho_\theta}\times (-\rho_\theta, \rho_\theta)$, after shrinking $c_0(n)$ if needed, since $|t|\leq \rho_\theta$ and $\rho_\theta\Lambda\leq c_0(n)$, we obtain \eqref{eq: derivative flow map} and \eqref{eq: estimate transversality}.
Thus, we can let $\delta_1\coloneqq \frac{1}{2}\theta \rho_1$, which satisfies all the required properties.
\end{proof}

Now let us assume as in \Cref{thm: orthogonal eps reg} that $\Sigma$ is a two-sided free-boundary minimal hypersurface in $M$ such that $\overline{\Sigma}\subset\Ncal^Z_{\delta}(P_R)$ for some $\delta\in(0, \delta_1), R \in(0, R_1)$, where $R_1\in (0, \delta_1)$ is to be specified later depending on $n$ and $\Lambda$ only.

\begin{definition}[Adapted tilt function]\label{def: adapted tilt}
 Let $\aX\coloneqq \mathrm{proj}^{T\Sigma}_\mathrm{g}(Z)$, $\tau=|\aX|_\mathrm{g}$, and $\vartheta\coloneqq \mathrm{g}(\nu, Z)$, where $Z$ is the vector field in \Cref{def: adapted neighbourhood}. 
 Then $\tau$ is a tilt function on $\Sigma$, with respect to $P$, adapted to $\partial M$. We will refer to it, for brevity, as the \emph{adapted tilt function}.
\end{definition}

We shall now compute the differential inequality satisfied by $\tau$, and notice that it is very similar to the inequality in \Cref{lemma: diff ineq tilt}. 

\begin{lemma}\label{lemma: adapted tilt differential inequality}
Let $P, \delta_1, Z$ be as in \Cref{lemma: quantitative adapted neighbourhood}. Additionally, let $\Sigma$ be a smooth two-sided free-boundary minimal hypersurface in $M$ and assume that $\Sigma\subset \Ncal^Z_\delta(P_R)$. Let $\nu$ be a choice of unit normal to $\Sigma$, $A$ its second fundamental form, and let $\tau, \aX, \vartheta$ be as in \Cref{def: adapted tilt}. Then $\tau^2+\vartheta^2=1$. Let $\aB$ be the (0,2)-tensor on $\Sigma$ corresponding to $\nabla^{\mathrm{g}}Z|_{T\Sigma}$. Then the following hold.

First, on $\{\tau>0\}$, if $\vz\coloneqq \aX/\tau$, then
\[
\nabla_\Sigma \tau=\vartheta A(\cdot,\vz)+\aB(\cdot,\vz)
\]
and, for every $\gamma>0$,
\[
|\nabla_\Sigma \tau|^2 \leq (1+\gamma)(1-\tau^2)|A(\cdot,\vz)|^2 + C(\gamma) |\aB|^2.
\]
Second, assuming that \eqref{eq: geometric bounds orthogonal} holds, then for every $\gamma>0$ on $\{\tau>0\}$
\[
\begin{aligned}
	\tau\Delta_\Sigma \tau
	&\geq (1-\gamma)|A|^2(1-\tau^2)-\left(\frac{n-1}{n}+\gamma\right)|A|^2-\gamma |A|^2\tau^2-C(n, \gamma)\Lambda^2.
\end{aligned}
\]
\end{lemma}

\begin{proof}
We shall follow the outline of the proof of \Cref{lemma: diff ineq tilt}, and highlight the analogies and differences. Notice that $|Z|_\mathrm{g}=1$ and $Z=\aX+\vartheta\nu$ along $\Sigma$, so that $\tau^2+\vartheta^2=1$.

Analogously to \Cref{lemma: diff ineq tilt}, let $\aB$ be the (0,2)-tensor on $\Sigma$ corresponding to $\nabla^{\mathrm{g}}Z|_{T\Sigma}$, so that for $V,W\in T\Sigma$,  $\aB(V, W)$ is defined by $\aB(V, W)\coloneqq \langle\nabla^\mathrm{g}_V Z, W\rangle$.

Notice that for $V\in T\Sigma$
\[
	\nabla^\mathrm{g}_V\aX = \nabla^\mathrm{g}_VZ-(V\vartheta)\nu-\vartheta\nabla^\mathrm{g}_V\nu=\nabla^\mathrm{g}_VZ-(V\vartheta)\nu+\vartheta A(V, \cdot)^\sharp.
\]
If we take a tangential projection,
\[
	\nabla^\Sigma_V\aX = (\nabla^\mathrm{g}_VZ)^T+\vartheta A(V, \cdot)^\sharp.
\]
Therefore for $V,W\in T\Sigma$,
\begin{equation}\label{eq: adapted B+theta A}
	\mathrm{g}(\nabla^\Sigma_V\aX, W)=\mathrm{g}(\nabla^\mathrm{g}_VZ, W)+\vartheta A(V, W) = \aB(V, W)+\vartheta A(V,W).
\end{equation}
By taking $W=\aX$ in this identity, we immediately obtain that
\[
\tau V\tau = \frac{1}{2}V(\tau^2)=\mathrm{g}(\nabla^\Sigma_V\aX, \aX)=\aB(V,\aX)+\vartheta A(V, \aX).
\]
On $\{\tau >0\}$, writing $\aX=\tau \vz$, this becomes $V\tau = \vartheta A(V, \vz)+\aB(V, \vz)$, or, equivalently, $\nabla_\Sigma\tau=\vartheta A(\cdot, \vz)+\aB(\cdot, \vz)$.
Taking norms, we obtain
\[
	|\nabla_\Sigma\aX|^2=\vartheta^2|A|^2+2\vartheta\langle A, \aB \rangle +|\aB|^2 = |A|^2(1-\tau^2)+2\vartheta\langle A, \aB \rangle +|\aB|^2.
\]
It follows from Young's inequality that for any fixed $\gamma>0$
\[
|\nabla_\Sigma \tau|^2 \leq (1+\gamma)\vartheta^2|A(\cdot, \vz)|^2 + C(\gamma) |\aB(\cdot, \vz)|^2.
\]
Since $\vartheta^2=1-\tau^2$ and $|\aB(\cdot, \vz)|\leq |\aB|$, we obtain
\[
|\nabla_\Sigma \tau|^2 \leq (1+\gamma)(1-\tau^2)|A(\cdot, \vz)|^2 + C(\gamma) |\aB|^2.
\]

Let $Y\in T\Sigma$ be a tangent vector field. 
Then we have the following Bochner identity (see \cite[Theorem 9.4.1]{Petersen})
\[
	\begin{aligned}
		\Delta_\Sigma\frac{|Y|^2}{2} &= |\nabla_\Sigma Y|^2+\langle Y, \nabla_\Sigma\divergence_\Sigma Y \rangle+ \langle Y, \divergence_\Sigma ((\nabla_\Sigma Y) - (\nabla_\Sigma Y)^*)\rangle+\Ric_\Sigma (Y,Y).\\
	\end{aligned}
\]
As in the proof of \Cref{lemma: diff ineq tilt}, since $H=0$, by the contracted Gauss equation
\[
	\Ric_\Sigma (Y,Y) = \Ric_{\mathrm{g}}(Y,Y)-\Riem_\mathrm{g}(Y, \nu, Y, \nu)-|A(Y, \cdot)|^2.
\]
Now, letting $Y=\aX$ in the Bochner formula, we notice that by \eqref{eq: adapted B+theta A}, 
\[
	\divergence_\Sigma \aX = \operatorname{tr}_\Sigma(\aB+\vartheta A)=\operatorname{tr}_\Sigma\aB + \vartheta H = \operatorname{tr}_\Sigma\aB
\]
since $H=0$ by minimality of $\Sigma$. Moreover, let $\mathcal K$ denote the skew-symmetric $(0,2)$-tensor on $\Sigma$ defined by $\mathcal{K}(V, W)=\langle\nabla^\Sigma_V\aX, W\rangle-\langle\nabla^\Sigma_W\aX, V\rangle$ for $V, W\in T\Sigma$, and notice that this term does not appear in the proof of \Cref{lemma: diff ineq tilt}, since in that situation $\nabla_\Sigma X$ is symmetric, because $X$ is a gradient vector field.

Since $\tau=|\aX|_{\mathrm{g}}$, we obtain 
\[
	\begin{aligned}
		\tau\Delta_\Sigma \tau+|\nabla_\Sigma \tau|^2 = \Delta_\Sigma \frac{\tau^2}{2} &= |\nabla_\Sigma \aX|^2 + \langle \aX, \nabla_\Sigma \operatorname{tr}_\Sigma \aB\rangle + \langle \aX, \divergence_\Sigma \mathcal K\rangle+\Ric_\Sigma(\aX, \aX)\\
		&= |A|^2(1-\tau^2)+2\vartheta\langle A, \aB \rangle +|\aB|^2\\
		&\quad +\langle \aX, \nabla_\Sigma \operatorname{tr}_\Sigma \aB\rangle + \langle \aX, \divergence_\Sigma \mathcal K\rangle\\
		& \qquad + \Ric_{\mathrm{g}}(\aX,\aX)-\Riem_\mathrm{g}(\aX, \nu, \aX, \nu)-|A( \cdot, \aX)|^2,
	\end{aligned}
\]
so that
\begin{equation}\label{eq: Bochner adapted}
	\begin{aligned}
		\tau\Delta_\Sigma\tau &=  |A|^2(1-\tau^2)-|A(\aX, \cdot)|^2-|\nabla_\Sigma \tau|^2 + 2\vartheta\langle A, \aB \rangle +|\aB|^2\\
		&\quad +\langle \aX, \nabla_\Sigma \operatorname{tr}_\Sigma \aB\rangle + \langle \aX, \divergence_\Sigma \mathcal K\rangle\\
		& \qquad + \Ric_{\mathrm{g}}(\aX,\aX)-\Riem_\mathrm{g}(\aX, \nu, \aX, \nu),
	\end{aligned}
\end{equation}
Because $H=0$, $A$ is trace-free. Hence, for every unit vector $v$,
$|A(v,\cdot)|^2\leq \frac{n-1}{n}|A|^2$, and 
\[
	|\nabla_\Sigma \tau|^2 \leq (1+\gamma)(1-\tau^2)\left(\frac{n-1}{n}\right)|A|^2 + C(\gamma) |\aB|^2.
\]
The same computations as in the proof of \Cref{lemma: diff ineq tilt} also show that 
\[
	|A(\cdot, \aX)|^2+|\nabla_\Sigma \tau|^2\leq \left(\frac{n-1}{n}+\gamma\right)|A|^2+C(\gamma)|\aB|^2
\]
and
\[
	2|\vartheta\langle A,\aB\rangle| \leq \gamma |A|^2\vartheta^2+C(\gamma) |\aB|^2 = \gamma |A|^2(1-\tau^2)+C(\gamma) |\aB|^2.
\]
For the remaining terms in \eqref{eq: Bochner adapted}, similarly to \Cref{lemma: diff ineq tilt}, we have 
\[
	|\nabla_\Sigma\operatorname{tr}_\Sigma\aB|+|\divergence_\Sigma\mathcal{K}|\leq C(n)(|\nabla^2_\mathrm{g}Z|+|A||\nabla^\mathrm{g}Z|).
\]
The bound in terms of $\Lambda$ then follows from \eqref{eq: geometric bounds orthogonal} since \[|\Ric_{\mathrm{g}}(\aX,\aX)-\Riem_\mathrm{g}(\aX, \nu, \aX, \nu)|\leq C(n)|\Riem_\mathrm{g}|\tau^2.\]
\end{proof}

We also need to estimate the boundary normal derivative of the adapted tilt function along $\partial M$.

\begin{lemma}\label{lemma: partial tau boundary}
	Let $\eta$ be the inward unit conormal to $\partial \Sigma$ in $\Sigma$.
	Along $\partial \Sigma\cap \{\tau>0\}$ we have $|\partial_\eta \tau|\leq |\sff_{\partial M}|+|\nabla^\mathrm{g}Z|$ pointwise.
\end{lemma}

\begin{proof}
Let us differentiate the function $\vartheta$ in the direction of the unit conormal $\eta$ to $\partial \Sigma$ along $\partial \Sigma\cap\{\tau>0\}$. Notice that $\nabla^\mathrm{g}_\eta\nu\in T\Sigma$, so that
\[
\partial_\eta \vartheta = \mathrm{g}(\nabla^\mathrm{g}_\eta\nu, Z)+\mathrm{g}(\nu, \nabla^\mathrm{g}_\eta Z) = \mathrm{g}(\nabla^\mathrm{g}_\eta\nu, \aX)+\mathrm{g}(\nu, \nabla^\mathrm{g}_\eta Z)=-A(\eta,\aX)+\mathrm{g}(\nu, \nabla^\mathrm{g}_\eta Z).
\]
By definition, $Z$ is tangent to $\partial M$ in $\partial M\cap \Ncal^Z_\delta(P_R)$. Therefore, along $\partial \Sigma \cap\{\tau>0\} \subset \partial M\cap \Ncal^Z_\delta(P_R)$, we have
\[
	\mathrm{g}(\aX, \bn)=\mathrm{g}(Z-\vartheta\nu, \bn) = \mathrm{g}(Z, \bn)-\vartheta \mathrm{g}(\nu, \bn)=0,
\]
and $\aX$ is also tangent to $\partial M$ along $\partial \Sigma$.
Hence, by symmetry of the second fundamental form and the free-boundary condition
\[
	A(\eta, \aX)=A(\aX, \eta)=-\mathrm{g}(\nabla^\mathrm{g}_{\aX}\nu, \eta)=-\mathrm{g}(\nabla^\mathrm{g}_{\aX}\nu, \bn)=-\sff_{\partial M}(\aX, \nu)
\]
and
\[
	\partial_\eta \vartheta = \sff_{\partial M}(\aX, \nu)+\mathrm{g}(\nu, \nabla^\mathrm{g}_\eta Z).
\]
Clearly, $|\sff_{\partial M}(\aX, \nu)|\leq |\sff_{\partial M}||\aX||\nu|=\tau |\sff_{\partial M}|$. Moreover, since $|Z|_\mathrm{g}\equiv 1$, 
\[
	|\mathrm{g}(\nu, \nabla^\mathrm{g}_\eta Z)|=|\mathrm{g}(\nu-\vartheta Z, \nabla^\mathrm{g}_\eta Z)|\leq |\nu-\vartheta Z||\eta||\nabla^\mathrm{g}Z|=\tau |\nabla^\mathrm{g}Z|.
\]
Finally, 
\[
	\tau |\partial_\eta \tau|=|\vartheta \partial_\eta \vartheta|\leq |\vartheta| \tau(|\sff_{\partial M}|+|\nabla^\mathrm{g}Z|)\leq \tau(|\sff_{\partial M}|+|\nabla^\mathrm{g}Z|)
\]
since $|\vartheta|\leq 1$. This concludes the proof.
\end{proof}

\subsection{Caccioppoli-type inequality}
We shall now move to the weak intrinsic Caccioppoli inequality. As in \Cref{sec: boundary sheeting}, we start with the smooth case, and then extend it to the singular set using the vanishing 2-capacity assumption on $\mathcal{S}$. 

\begin{lemma}[Caccioppoli inequality]\label{lemma: Caccioppoli FB singular}
Let $\Sigma$ be a stable \emph{singular} free-boundary minimal hypersurface in $M$ and assume that  $\overline{\Sigma}\subset \Ncal^Z_\delta(P_R)$. Assume that \eqref{eq: geometric bounds orthogonal} holds. Then there exist a positive dimensional constant $C(n)>0$ and a constant $ K_{\mathrm{Cac}}=K_{\mathrm{Cac}}(n,\Lambda)>0$ depending on $n$ and $\Lambda$ such that for every $k\in[0,1/(4n)]$ and every $\varphi\in C_c^{0,1}(\Ncal^Z_{\rho_1}(P_R))$, we have
\begin{equation}\label{eq: Caccioppoli smooth A 2}
	\begin{aligned}
		\frac{1}{2n}\int_{\{\tau>k\}}
		|A|^2
		\left(1-\frac{k}{\tau}\right)\varphi^2
		&\leq
		C(n)\left(\int_\Sigma (\tau-k)_+^2|\nabla_\Sigma\varphi|^2
		+K_\mathrm{Cac}^2\int_{\{\tau>k\}}\varphi^2\right).
	\end{aligned}
\end{equation}
Consequently,
\begin{equation}\label{eq: Caccioppoli smooth nabla 2}
	\begin{aligned}
		\frac{1}{2n}\int_{\{\tau>k\}}
		|\nabla_\Sigma \tau|^2
		\left(1-\frac{k}{\tau}\right)\varphi^2
		&\leq
		C(n)\left(\int_\Sigma (\tau-k)_+^2|\nabla_\Sigma\varphi|^2
		+K_{\mathrm{Cac}}^2\int_{\{\tau>k\}}\varphi^2\right).
	\end{aligned}
\end{equation}
\end{lemma}

\begin{proof}
We shall first prove \eqref{eq: Caccioppoli smooth A 2} and \eqref{eq: Caccioppoli smooth nabla 2} for test functions $\varphi\in C^{0,1}_c(\Sigma)$, and then extend to compactly supported ambient Lipschitz functions by an excision argument, similar to the one in \Cref{lemma: trace inequality}.

Let $\Omega_k\coloneqq\{\tau>k\}$, and apply \Cref{lemma: stability absorbed 1} to $\zeta=(\tau-k)_+\varphi$ to obtain 
\begin{equation}\label{eq: Caccioppoli stability 1}
	\int_{\Sigma}|A|^2(\tau-k)_+^2\varphi^2\leq (1+\alpha)\int_\Sigma|\nabla_\Sigma((\tau-k)_+\varphi)|^2+C(n,\alpha)\Lambda^2 \int_\Sigma (\tau-k)^2_+\varphi^2.
\end{equation}
Let $J\coloneqq  \int_\Sigma|\nabla_\Sigma((\tau-k)_+\varphi)|^2$ and integrate by parts to get 
\begin{equation}\label{eq: J by parts}
\int_\Sigma |\nabla_\Sigma((\tau-k)_+\varphi)|^2 = -\int_{\Omega_k}(\tau-k)\Delta_\Sigma \tau\,\varphi^2 + \int_\Sigma (\tau-k)_+^2|\nabla_\Sigma\varphi|^2-\int_{\partial\Sigma \cap \Omega_k}(\tau-k)_+\partial_\eta\tau \varphi^2. 
\end{equation}
By \Cref{lemma: adapted tilt differential inequality}, on $\Omega_k$ we have
\begin{equation}\label{eq: use tilt differential inequality}
	-(\tau-k)\Delta_\Sigma\tau\leq -\left(1-\frac{k}{\tau}\right)\left[\left((1-\gamma)(1-\tau^2)-\frac{n-1}{n}-\gamma-\gamma\tau^2\right)|A|^2-C(n,\gamma)\Lambda^2\right].
\end{equation}
Then \eqref{eq: Caccioppoli stability 1} becomes
\[
	\begin{aligned}
		\int_{\Omega_k}&\left(1-\frac{k}{\tau}\right)\left(\tau(\tau-k)+(1+\alpha)\left(\frac{1}{n}-\tau^2-2\gamma\right)\right)|A|^2\varphi^2\\
		&\leq (1+\alpha)\int_\Sigma (\tau-k)^2_+|\nabla_\Sigma\varphi|^2+C(n,\alpha)\Lambda^2 \int_\Sigma(\tau-k)^2_+\varphi^2+(1+\alpha)C(n,\gamma)\Lambda^2\int_{\Omega_k}\varphi^2\\
		&\quad -(1+\alpha)\int_{\partial\Sigma \cap \Omega_k}(\tau-k)_+\partial_\eta\tau \varphi^2.
	\end{aligned}
\]
The same choices $\alpha=\alpha_n, \gamma=\gamma_n$ (depending on $n$ only) as in \Cref{lemma: Caccioppoli inequality 1} make the coefficient in front of $(1-k/\tau)|A|^2\varphi^2$ bounded below by $1/(2n)+1/(8n^2)>1/(2n)$. Also, noticing that $0\leq (\tau-k)_+\leq 1$ and $(\tau-k)_+=0$ outside $\Omega_k$, we can write the previous inequality as 
\[
	\begin{aligned}
		\frac{1}{2n}\int_{\Omega_k}\left(1-\frac{k}{\tau}\right)|A|^2\varphi^2 &\leq C(n)\int_\Sigma (\tau-k)^2_+|\nabla_\Sigma\varphi|^2+C(n,\Lambda)\int_{\Omega_k}\varphi^2\\
		&\quad-C(n)\int_{\partial\Sigma \cap \Omega_k}(\tau-k)_+\partial_\eta\tau \varphi^2.
	\end{aligned}
\] 
Let us now deal with the boundary term. Let $B_k\coloneqq \int_{\partial\Sigma \cap \Omega_k}(\tau-k)_+\partial_\eta\tau \varphi^2$. By \Cref{lemma: partial tau boundary},
\[
	|B_k|\leq \left(\sup_{U_1}|\sff_{\partial M}|+\sup_{\Ncal^Z_\delta(P_R)}|\nabla^\mathrm{g}Z|\right) \int_{\partial \Sigma}(\tau-k)_+\varphi^2\leq C(n ,\Lambda)\int_{\partial \Sigma}(\tau-k)_+\varphi^2
\]
where the second inequality follows from \Cref{lemma: quantitative adapted neighbourhood} and \eqref{eq: geometric bounds orthogonal}. 
By \Cref{lemma: trace inequality}, 
\[
	\int_{\partial \Sigma}(\tau-k)_+\varphi^2\leq \int_{\Omega_k}\varphi^2|\nabla_\Sigma\tau|+2\int_\Sigma(\tau-k)_+\varphi|\nabla_\Sigma\varphi|+C_{\mathrm{tr}}(n, \Lambda)\int_\Sigma(\tau-k)_+\varphi^2.
\]
For every $\lambda>0$, by the weighted Young inequality
\[
	\begin{aligned}
		C(n, \Lambda)\left( \int_{\Omega_k}\varphi^2|\nabla_\Sigma\tau|+2\int_\Sigma(\tau-k)_+\varphi|\nabla_\Sigma\varphi|\right) &\leq \lambda\int_{\Omega_k}\varphi^2|\nabla_\Sigma\tau|^2+\int_\Sigma(\tau-k)^2_+|\nabla_\Sigma\varphi|^2\\
		&\quad +C(\lambda) C(n, \Lambda)\int_{\Omega_k}\varphi^2.
	\end{aligned}
\]
Similarly, since $(\tau-k)_+\leq \1_{\Omega_k}$, we can write $\int_\Sigma(\tau-k)_+\varphi^2\leq \int_{\Omega_k}\varphi^2$.
Thus,
\[
	\begin{aligned}
		|B_k| &\leq \lambda\int_{\Omega_k}\varphi^2|\nabla_\Sigma\tau|^2 +\int_{\Sigma}(\tau-k)^2_+|\nabla_\Sigma\varphi|^2+C(\lambda) C(n, \Lambda)\int_{\Omega_k}\varphi^2.
	\end{aligned}	 
\]

Since $\varphi\nabla_\Sigma(\tau-k)_+=\nabla_\Sigma((\tau-k)_+\varphi)-(\tau-k)_+\nabla_\Sigma \varphi$, we have the obvious inequality $\int \varphi^2|\nabla_\Sigma(\tau-k)_+|^2\leq 2J+2\int (\tau-k)^2_+|\nabla_\Sigma\varphi|^2$, so that
\begin{equation}\label{eq: absorption 1}
	|B_k|\leq 2\lambda J+ C(n,\lambda) \mathcal{R}_k
\end{equation}
where $\mathcal{R}_k\coloneqq \int (\tau-k)^2_+|\nabla_\Sigma\varphi|^2+C( n, \Lambda)\int_{\Omega_k}\varphi^2$.

Recall that by \eqref{eq: J by parts} and \eqref{eq: use tilt differential inequality}, we have
\begin{equation}\label{eq: absorption 2}
J\leq  C(n)\int_{\Omega_k}\left(1-\frac{k}{\tau}\right)|A|^2\varphi^2 +  C(n)\mathcal{R}_k+|B_k|. 
\end{equation}
Therefore, by combining \eqref{eq: absorption 1} and \eqref{eq: absorption 2}, we obtain 
\[
	|B_k|\leq 2\lambda\left(C(n)\int_{\Omega_k}\left(1-\frac{k}{\tau}\right)|A|^2\varphi^2 +  C(n)\mathcal{R}_k+|B_k|\right)+C(n,\lambda)\mathcal{R}_k,
\]
i.e.
\begin{equation}\label{eq: absorption}
	|B_k|\leq \frac{2\lambda C(n)}{1-2\lambda} \int_{\Omega_k}\left(1-\frac{k}{\tau}\right)|A|^2\varphi^2 +\frac{2\lambda C(n)+ C(n, \lambda)}{1-2\lambda}\mathcal{R}_k.
\end{equation}

By choosing $\lambda>0$ small enough (depending on $n$ only) that $\frac{2\lambda C(n)}{1-2\lambda}<\frac{1}{8n^2}$, we can reabsorb  the multiple of $\int_{\Omega_k}\left(1-\frac{k}{\tau}\right)|A|^2\varphi^2 $ coming from \eqref{eq: absorption} into the left hand side and obtain the desired inequality.

Finally, for an ambient test function $\varphi$, first apply the inequality to $\chi_j\varphi$, where $\chi_j$ are the cutoff functions of \Cref{cor: cap}, chosen to equal one in the limit on $\spt\varphi\setminus\mathcal S$, then take the limit as in the proof of \Cref{lemma: Caccioppoli inequality 1}.
\end{proof}

\subsection{Proof of \Cref{thm: orthogonal eps reg}}\label{sec: DG iteration orthogonal}
As in \Cref{sec: boundary sheeting}, we can now conclude the argument for \Cref{thm: orthogonal eps reg} by applying the iteration lemma in \Cref{sec: abstract DG}.

Let us assume that \eqref{eq: geometric bounds orthogonal} holds. Define $K_1(n, \Lambda)\coloneqq \max\{K_{\mathrm{Sob}}(n, \Lambda), K_{\mathrm{Cac}}(n, \Lambda)\}$, let $\delta_1$ be as in \Cref{lemma: quantitative adapted neighbourhood}, and set 
\begin{equation}\label{eq: definition R0 orthogonal}
	R_1\coloneqq \frac{1}{2}\min\bigl\{\delta_1, K_1^{-1}\bigr\}.
\end{equation}
Let  $\Sigma$ be a stable singular free-boundary minimal hypersurface in $M$ such that $\overline{\Sigma}\subset\Ncal^Z_{\delta}(P_R)$ for some $\delta\in(0, \delta_1), R\in(0, R_1)$.

Then, for $0<r\leq R$ define the \emph{adapted cylinder} $\Ccal^Z_r$ by
\begin{equation}\label{eq: definition adapted cylinder}
	\Ccal^Z_r\coloneqq F_Z(P_r\times(-2\delta, 2\delta))=\Ncal^Z_{2\delta}(P_r),
\end{equation}
and consider the following \textit{adapted tilt excess} quantity: 
\begin{equation}
	\mathrm{E}^Z_{\Sigma, P}(r)\coloneqq \left(r^{-n}\int_{\Sigma\cap \Ccal^Z_r}\tau^2\right)^{1/2}
\end{equation}
where $\tau$ is the \emph{adapted tilt function} in \Cref{def: adapted tilt}.

Let us apply \Cref{lemma: abstract DG} with  $u=\tau$ and $k_*=1/(4n)$. Notice that \Cref{lemma: Caccioppoli FB singular} provides the necessary Caccioppoli-type inequality \eqref{eq: abstract Caccioppoli}. 
Therefore, by \Cref{cor: abstract DG cor}, there exist $\varepsilon_1=\varepsilon_1(n, \Lambda)>0$ and $C_1=C_1(n, \Lambda)>0$ such that 
\[
\sup_{\Sigma \cap \Ccal^Z_{\sigma R}}\tau\leq C_1(1-\sigma)^{-\frac{n}{n+2}}\,\mathrm{E}^Z_{\Sigma, P}(R)^{\frac{2}{n+2}}
\]
provided $\mathrm{E}^Z_{\Sigma, P}(R)^2\leq \varepsilon_{1}^2(1-\sigma)^n$, which is the desired statement.

\subsection{Proof of \Cref{thm: orthogonal sheeting}}\label{sec: proof orthogonal sheeting}
The proof of the second sheeting theorem, \Cref{thm: orthogonal sheeting}, is rather similar to the proof of \Cref{thm: boundary sheeting}. 

Assume $\overline{\Sigma}\subset \Ncal^Z_\delta(P_R)$. By \Cref{thm: orthogonal eps reg} with $\sigma=3/4$, we have
\begin{equation}\label{eq: smalleness assumption orthogonal}
	\sup_{\Sigma\cap \Ccal^Z_{3R/4}}\tau\leq C_1(n, \Lambda)\mathrm{E}^Z_{\Sigma, P}(R)^{\frac{2}{n+2}}.
\end{equation}
By decreasing $\varepsilon_1$ if necessary, we may assume $\sup_{\Sigma\cap \Ccal^Z_{3R/4}}\tau\leq \frac{1}{4}$.

\textit{Step 0: $(\Sigma\setminus\partial\Sigma)\cap \partial M \cap \Ccal^Z_{3R/4}=\emptyset$.} Let $\iota:N\to M\setminus\mathcal{S}$ parametrise $\Sigma$. If $z\in N\setminus\partial N$ and $x=\iota(z)\in \partial M\cap \Ccal^Z_{3R/4}$, the function $d_{\partial M}\circ \iota$ has a local minimum at $z$, and $\mathrm{d}(d_{\partial M}\circ\iota)_z=0$, so that $T_x\Sigma\subset T_x\partial M$. By dimension counting, $T_x\Sigma = T_x\partial M$, which implies that $\tau(x)=1$, contradicting \eqref{eq: smalleness assumption orthogonal}. 

\textit{Step 1: Local graphicality.} Notice that  
\begin{equation}\label{eq: kernel projection 2}
	\operatorname{Ker}(D\pi_{Z})_x=\operatorname{span}\{Z(x)\}
\end{equation}
for every $x\in \Ncal^Z_\delta(P_R)$. As in the proof of \Cref{thm: boundary sheeting} in \Cref{sec: proof boundary sheeting}, if $x=\iota(z)\in \Ccal^Z_{3R/4}$ for some $z\in N$, we can show that $D(\pi_Z\circ\iota)_z$ is an isomorphism. If $z\in \partial N$, the same argument and the identity $\pi_Z(\partial M)=P\cap \partial M$ show that in this case $\pi_Z\circ\iota$ is a local diffeomorphism of manifolds-with-boundary. 
In particular, similarly to \Cref{sec: proof boundary sheeting}, for every $z\in N$ such that $\iota(z)\in \Ccal^Z_{3R/4}$ there is a (relatively open) neighbourhood (in $N$) the image of which is a smooth graph over a relatively open subset of $P_{3R/4}$ via the map $F_Z$. 

\textit{Step 2: Finite covering away from the projected singular set.}  Let
\[
	\Gamma_Z\coloneqq\pi_Z(\Scal\cap \overline{\Ncal^Z_\delta}(P_{3R/4}))\subset \bar P_{3R/4}, \quad B_Z\coloneqq P_{3R/4}\setminus \Gamma_Z.
\]
As in the proof of \Cref{thm: boundary sheeting},  $\Hcal^{n-1}(\Gamma_Z)=0$. By the same argument as in the proof of \Cref{thm: boundary sheeting}, $\Gamma_Z$ is relatively closed, and $B_Z$ is open, dense and path-connected in $P_{3R/4}$.
Define 
\[
N_Z\coloneqq \{z\in N: \iota(z)\in \Ccal^Z_{3R/4},\ \pi_Z(\iota(z))\in B_Z\},
\]
and $\hat\pi_Z\coloneqq \pi_Z\circ\iota|_{N_Z}:N_Z\to B_Z$. By \textit{Step 1}, $\hat\pi_Z$ is a local diffeomorphism of manifolds-with-boundary and, by properness of $\iota$ as an immersion into $M\setminus \Scal$, it is also proper. Again, as in the proof of \Cref{thm: boundary sheeting}, it is a finite covering map of some degree $Q$.

\textit{Step 3: Lipschitz estimate and extension across $\Gamma_Z$.} By \textit{Step 2}, after ordering the preimages in the $Z$-fibre $F_Z(\{y\}\times (-\delta, \delta))$, we obtain locally Lipschitz functions $u^Z_1, \dots, u^Z_Q$ such that $u^Z_1(y)\leq \dots\leq u^Z_Q(y)$ for all $y\in B_Z$. Hence, 
\begin{equation}\label{eq:bs-graph-decomposition-on-BZ}
	\Sigma\cap\pi_Z^{-1}(B_Z) = \cup_{j=1}^{Q}F_Z\bigl(\operatorname{graph}(u^Z_j)\bigr).
\end{equation}

For a smooth graph $G[w]=F_Z(\{(y, w(y)) : y\in P_R\})$ on $P_R$ and $v\in T_yP$, where $x=F_Z(y, w(y))$, consider the vector
\[
	V(x)=(D_yF_Z)_{(y, w(y))}(v)+dw_y(v)Z(x)\in T_xG[w].
\]
Since $|Z|_\mathrm{g}=1$, by taking the inner product with $Z$, we have 
\[
	dw_y(v)=\langle V(x), Z(x)\rangle-\langle (D_yF_Z)_{(y, w(y))}(v), Z(x) \rangle.
\]
Thus, by \Cref{lemma: quantitative adapted neighbourhood},
\[
\begin{aligned}
	|dw_y(v)| &\leq |\mathrm{proj}^{T_xG[w]}_\mathrm{g}Z(x)||V(x)|+|\langle (D_yF_Z)_{(y, w(y))}(v), Z(x) \rangle|\\
	&\leq |\mathrm{proj}^{T_xG[w]}_\mathrm{g}Z(x)|(C(n)|v|+|dw_y(v)|)+C(n) \Lambda \delta|v|.
\end{aligned}
\]
If $|\mathrm{proj}^{T_xG[w]}_\mathrm{g}Z(x)|\leq 1/4$ (as is the case for our graphs $u^Z_j$, where this is just $\tau$), then 
\[
	|dw_y(v)| \leq C(n, \Lambda)\left(|\mathrm{proj}^{T_xG[w]}_\mathrm{g}Z(x)|+\delta\right)|v|.
\]
Since we can approximate curves in $P_{3R/4}$ by curves which avoid $\Gamma_Z$, by \eqref{eq: smalleness assumption orthogonal} we conclude that
\[
	\operatorname{Lip}_P(u^Z_j;B_Z)\leq C(n, \Lambda)\left(\delta+\mathrm{E}^Z_{\Sigma,P}(R)^{\frac{2}{n+2}}\right).
\]

Similarly to the proof of \Cref{thm: boundary sheeting}, by density of $B_Z$ in $P_{3R/4}$, every $u^Z_j$ extends uniquely to a Lipschitz function (still denoted $u^Z_j$) on the whole of $P_{3R/4}$, with the same bound on the Lipschitz constant, namely
\[
	\operatorname{Lip}_P(u^Z_j)\leq C(n, \Lambda)\left(\delta+\mathrm{E}^Z_{\Sigma,P}(R)^{\frac{2}{n+2}}\right).
\]
Finally, showing that $\overline{\Sigma}\cap \Ncal^Z_\delta(P_{3R/4})=\bigcup_{j=1}^Q F_Z(\operatorname{graph}(u^Z_j))$ follows from the same argument as in the proof of \Cref{thm: boundary sheeting}.

As in \Cref{sec: boundary sheeting}, we can now prove (i) and (ii).
\begin{enumerate}[(i)]
	\item Assume $\Hcal^{n-2}(\Scal)=0$. Since $\pi_Z$ is Lipschitz on $\Ncal^Z_{2\delta}(P_{3R/4})$ by \eqref{eq: derivative flow map}, then $\Hcal^{n-2}(\Gamma_Z)=0$. By reflecting across $\partial_0 P_R$ and applying the result in \cite[Appendix A]{SimonWicFrequency}, we see that $B_Z$ must then be simply connected, so every component of the covering $\hat{\pi}_Z:N_Z\to B_Z$ is one-sheeted. As in the proof of \Cref{thm: boundary sheeting}, this determines $Q$ globally defined, not necessarily ordered Lipschitz functions $v^Z_1, \dots, v^Z_Q:B_Z\to (-\delta, \delta)$ such that
	\[
		\Sigma\cap \pi_Z^{-1}(B_Z)=\cup_{j=1}^QF_Z(\operatorname{graph}(v^Z_j)),
	\]
	each of which extends to a Lipschitz function (still denoted $v^Z_j$) on the whole of $P_{3R/4}$, preserving the Lipschitz constant.
	
	As in the proof of \Cref{thm: boundary sheeting}, $\Gamma_Z$ is relatively closed and $\Hcal^{n-1}(\Gamma_Z)=0$. Notice that the flow of $Z$ preserves $\partial M$, so the vertical variation $v^Z_j+t\varphi$ is admissible on $B_Z$ up to the boundary $\partial_0P_{3R/4}$. Therefore, the Lipschitz functions $v^Z_j$ are weak solutions of the minimal surface equation (with a uniform bound on the Lipschitz constant) for test functions compactly supported in $B_Z$. By direct computation, the area integrand $\mathcal{L}(y, z, q)$ in the local coordinates induced by the diffeomorphism $\Psi$ (in \Cref{lemma: adapted neighbourhood}) is uniformly elliptic by \Cref{lemma: quantitative adapted neighbourhood}. Therefore, \Cref{lemma: removabillity} applies and the $v^Z_j$ extend smoothly across $\Gamma_Z$ all the way to the boundary, so that $\mathcal{S}\cap \Ncal^Z_\delta(P_{R/2})=\emptyset$, as required.
	The bound on $\|\nabla v^Z_j\|_{C^{1,\alpha}(P_{3R/4})}$ follows from standard elliptic regularity theory as in the proof of \Cref{thm: boundary sheeting}.
	
	\item Assume now $\Sigma$ is embedded. Then for every $y\in B_Z$ the points of $\Sigma$ lying in the $Z$-fibre are distinct and $u^Z_1<u^Z_2<\dots<u^Z_Q$ on $B_Z$, so that each $u^Z_j$ is smooth on $B_Z$. The same removability argument as above (since $\Hcal^{n-1}(\Gamma_Z)=0$) implies that the $u^Z_j$ are globally smooth, and weak ordering is preserved. Strict ordering is also preserved by the strong maximum principle for the minimal surface equation in the interior and by the corresponding boundary strong maximum principle (Hopf lemma) using the uniformly oblique boundary condition. Thus $\overline{\Sigma}\cap \Ccal^Z_{3R/4}$ is a union of pairwise disjoint smooth graphs and $\Scal\cap \Ncal_\delta^Z(P_{3R/4})= \emptyset$. Finally, the gradient estimates follow from standard elliptic regularity theory as in (i).
\end{enumerate}

\section{Compactness} \label{sec: compactness}
In this section we give a proof of \Cref{thm: compactness}. This roughly follows \cite{SchoenSimon81} and crucially uses \Cref{thm: boundary sheeting} and \Cref{thm: orthogonal sheeting}, proven in the previous sections. 

\subsection{From height to tilt}\label{sec: height-tilt} We start by proving two propositions bounding the tilt excess by the $L^2$ and $L^1$ norms of the height. This will be used later in the proof of compactness to deduce from the varifold convergence smallness of the tilt, and consequently that we are in a regime in which we can apply the sheeting theorems. 
The Euclidean interior analogue is the standard tilt excess inequality (see e.g. \cite[Remark 4.2]{Bellettini}).

We start with the boundary parallel case and then move to the orthogonal one. 

\begin{proposition}[Tilt-excess estimate] \label{prop: boundary titl height estimate}
Let $M, P, \rho_0, R_0$ be as in the statement of \Cref{thm: boundary eps reg}, and let $\delta\in(0, \rho_0), R\in(0, R_0)$. Suppose $\Sigma$ is a singular properly immersed free-boundary minimal hypersurface in $M$ such that $\overline{\Sigma}\subset \Ncal^+_\delta(P_R)$. Then, there exists $C(n, \Lambda) > 0$ such that for $0<2r<R$ the following estimate holds 
\begin{equation} \label{eq: tilt-height excess estimate}
    \frac{1}{r^n} \int_{\Sigma \cap C_r} g^2 = \frac{1}{r^n} \int_{\Sigma \cap C_r} \vert \nabla_\Sigma d_{\partial M} \vert^2 \leq \frac{C(n, \Lambda)}{r^{n + 2}} \int_{\Sigma \cap C_{2r}} \left(d_{\partial M}^2 + r^2 d_{\partial M}\right),
\end{equation}
where $g$ is the tilt function in \Cref{def: tilt}.
\end{proposition}
\begin{proof}
Consider the vector field $X(x) = d_{\partial M}(x) \varphi^2(x) \nabla^\mathrm{g} d_{\partial M}(x)$, where $\varphi \in C_c^1(C_{2r})$ is a smooth cut-off taken to be identically $1$ on $C_r$ and satisfying $\vert \nabla \varphi \vert \leq C/r$ (for some absolute constant $C$). In particular, $X(x) \equiv 0$ on $\partial M$ so that we can test the first variation with it. 
In particular, by free-boundary stationarity
\[
	0=\int_\Sigma \divergence_\Sigma X=\int_\Sigma \varphi^2 g^2+2\varphi d_{\partial M}\langle\nabla_\Sigma \varphi, \nabla_\Sigma d_{\partial M}\rangle+\varphi^2 d_{\partial M} \operatorname{tr}_\Sigma\nabla^2_\mathrm{g}d_{\partial M},
\]
since the mean curvature of $\Sigma$ is assumed to vanish in $\Sigma \setminus \partial \Sigma$ by \Cref{def: fb smooth immersed min hyp}. Hence, 
\[
\begin{aligned}
	\int_{\Sigma \cap C_r}g^2 \leq \int_\Sigma \varphi^2 g^2\leq 2\int_\Sigma|\varphi \nabla_\Sigma d_{\partial M}| |d_{\partial M}\nabla_\Sigma\varphi|+\sqrt{n}\int_\Sigma \varphi^2d_{\partial M}|\nabla_\mathrm{g}^2d_{\partial M}|
\end{aligned}
\]
from which \eqref{eq: tilt-height excess estimate} follows by Young's inequality and \eqref{eq: geometric control on neighbourhood}.
\end{proof}

\begin{proposition}[Orthogonal tilt--height excess estimate] \label{prop: orthogonal tilt height estimate}
Let $M, P, \rho_1, \delta_1, R_1$ be as in the statement of \Cref{thm: orthogonal eps reg}. Let $z_P$ and $Z$ be as in \Cref{def: adapted neighbourhood}, and let $\delta\in(0, \delta_1), R\in(0, R_1)$. Suppose $\Sigma$ is a  singular properly immersed free-boundary minimal hypersurface in $M$ such that $\overline{\Sigma}\subset \Ncal^Z_\delta(P_R)$. Then, there exists $C(n, \Lambda) > 0$ such that for $0<2r<R$ the following estimate holds 
\begin{equation} \label{eq: orthogonal tilt height estimate}
    \frac{1}{r^n} \int_{\Sigma\cap\mathcal C^Z_r} \tau^2 \leq \frac{C(n, \Lambda)}{r^{n+2}} \int_{\Sigma\cap\mathcal C^Z_{2r}} \left(z_P^2+r^2|z_P|\right),
\end{equation}
where $\tau$ is the adapted tilt function in \Cref{def: adapted tilt}.\\
Moreover, since $ |\nabla_\Sigma z_P|^2=|\nabla^\mathrm{g} z_P|^2\tau^2,$ the same estimate holds, up to increasing the constant $C$, with $|\nabla_\Sigma z_P|^2$ in place of $\tau^2$.
\end{proposition}
\begin{proof}
Let $W\coloneqq |\nabla^\mathrm{g}z_P|^{-1}Z$ and notice that $1/2\leq |W|_{\mathrm{g}}\leq 2$ by \Cref{def: adapted neighbourhood}.
For $0< 2r\leq R$ choose a cutoff $\chi$ on $P_R$ which is compactly supported on $P_{2r}$ and such that $0\leq\chi\leq1,$ and $ \chi\equiv 1$ on $P_r,$ as well as $|\nabla_{ P}\chi| \leq C/r,$ where $C$ is an absolute constant. Then pick a compactly supported vertical cutoff $\xi\geq \1_{[-\delta, \delta]}$ supported in $(-2\delta, 2\delta)$ and let $\varphi=(\xi\circ \mathcal{t}_Z)(\chi \circ \pi_Z)$, so that $\varphi\equiv 1 $ on $\mathcal C^Z_r$ as well as $\operatorname{spt}(\varphi)\cap\Sigma \subset\mathcal C^Z_{2r}$. Also, $\varphi$ satisfies $|\nabla\varphi|\leq C/r,$ where $C$ depends on $n, \Lambda, \rho_1$ only.

Consider then the vector field $ Y=\varphi^2z_PW.$ Since $W$ is tangent to $\partial M$ along $\partial M\cap \Ncal^Z_\delta(P_R)$, then $Y$ is an admissible vector field for free-boundary stationarity and $\int_\Sigma \operatorname{div}_\Sigma Y = 0.$ 
Notice that 
\begin{equation}
	\label{eq: adapted tilt identity}
	\langle\nabla_\Sigma z_P,W\rangle = \left\langle |\nabla^\mathrm{g} z_P|Z^T, \frac{Z}{|\nabla^\mathrm{g} z_P|} \right\rangle = |Z^T|^2 = \tau^2.
\end{equation}
Therefore, 
\[
\begin{aligned}
\divergence_\Sigma Y =\operatorname{div}_\Sigma(\varphi^2z_PW) &= \varphi^2\langle\nabla_\Sigma z_P,W\rangle +2\varphi z_P \langle\nabla_\Sigma\varphi,W\rangle +\varphi^2z_P\,\operatorname{div}_\Sigma W
\\ &=\varphi^2\tau^2+2\varphi z_P|\nabla^\mathrm{g}z_P|^{-1}\langle\nabla_\Sigma \varphi, Z^T\rangle+ \varphi^2z_P\divergence_\Sigma W.
\end{aligned}
\] 
Notice that
\[
\begin{aligned}
\operatorname{div}_\Sigma W &= |\nabla^\mathrm{g}z_P|^{-1}\divergence_\Sigma Z-\left\langle Z, |\nabla^\mathrm{g}z_P|^{-2}\nabla_\Sigma |\nabla^\mathrm{g}z_P|\right\rangle
\end{aligned}
\]
so that $|\operatorname{div}_\Sigma W|\leq C(n)\bigl(|\nabla^\mathrm{g}Z|+|\nabla_\mathrm{g}^2z_P|\bigr)\leq C(n)\Lambda$ by \Cref{def: adapted neighbourhood} and \Cref{lemma: quantitative adapted neighbourhood}.
Combining all of the above, we obtain
\begin{align}   \label{eq: preliminary tilt height estimate}
    \int_\Sigma\varphi^2\tau^2
    &\leq 4 \int_\Sigma \varphi|z_P|\,|\nabla_\Sigma\varphi|\,\tau + C(n)\Lambda \int_\Sigma \varphi^2|z_P|.
\end{align}
By Young's inequality, $ 4 \varphi|z_P|\,|\nabla_\Sigma\varphi|\,\tau \leq \frac{1}{2}\varphi^2\tau^2 + 8z_P^2|\nabla_\Sigma\varphi|^2.$ After absorbing the first term into the left-hand side, it follows that
\begin{equation}
\label{eq: caccioppoli adapted tilt}
    \int_\Sigma\varphi^2\tau^2 \leq 16\int_\Sigma z_P^2|\nabla_\Sigma\varphi|^2 + C(n)\Lambda\int_\Sigma \varphi^2|z_P|.
\end{equation}
Using properties of $\varphi$, we conclude that
\begin{equation}
\label{eq: unscaled orthogonal tilt height estimate}
    \int_{\Sigma\cap\mathcal C^Z_r}
    \tau^2 \leq \frac{C}{r^2} \int_{\Sigma\cap\mathcal C^Z_{2r}} z_P^2 + C(n, \Lambda) \int_{\Sigma\cap\mathcal C^Z_{2r}} |z_P|.
\end{equation}
Dividing by $r^n$ proves
\eqref{eq: orthogonal tilt height estimate}.

Finally, $|\nabla_\Sigma z_P|^2 = |(\nabla^\mathrm{g} z_P)^T|^2 = |\nabla^\mathrm{g} z_P|^2|Z^T|^2 = |\nabla^\mathrm{g} z_P|^2\tau^2.$ Thus, $ (1/4)\tau^2 \leq |\nabla_\Sigma z_P|^2 \leq 4\tau^2,$ which proves the corresponding estimate for $|\nabla_\Sigma z_P|^2$.
\end{proof}

\begin{remark}
	The height--tilt estimate also admits a formulation using only $Z$ described in \Cref{rmk: no gradient}, right after \Cref{def: adapted neighbourhood}. In that
	formulation, use the flow time $\mathcal t_Z$ as height. Since
	$Z(\mathcal t_Z)=1$ and $\nabla^{\mathrm g}\mathcal t_Z=Z$ on $P$, the derivative and flow estimates give
	\[
	|\nabla^{\mathrm g}\mathcal t_Z-Z|
	\leq C(n)\Lambda|\mathcal t_Z|.
	\]
	Testing free-boundary stationarity with $\varphi^2\mathcal t_ZZ$ therefore yields \eqref{eq: orthogonal tilt height estimate} with $\mathcal t_Z$ in place of $z_P$: the error in replacing $\langle\nabla_\Sigma\mathcal t_Z,Z\rangle$ by $\tau^2$ is bounded by $C(n)\Lambda|\mathcal t_Z|\tau$, and hence by the linear height term already present in the estimate.
\end{remark}

\subsection{Free-boundary stationary varifolds}\label{sec: FB stationary varifolds}
Notice that \Cref{def: sing min hyp} implies that a singular properly embedded\footnote{Most of the results in this specific subsection do actually hold for singular properly \textit{immersed} free-boundary minimal hypersurfaces $\Sigma$ with minimal changes in the statements and proofs.} 
free-boundary minimal hypersurface $\Sigma$ defines an integral varifold $V = |\Sigma| \coloneqq \mathbf{v}(\Sigma, \1_{\Sigma})$ on $M$. We refer to \cite{Simon1983} for the definition and basic properties of varifolds.

We define the \emph{regular set} $\reg V$ of such a varifold to be:
\begin{equation*}
	\reg(V)\coloneqq\bigg\{x \in \spt \|V\|:
	\begin{array}{l}
		\textrm{there exists $r>0$ such that $\spt\|V\|\cap B_{r}(x)$ is a smooth }
		\\
		\textrm{embedded hypersurface-with-(possibly empty) boundary in $\partial M$}
	\end{array}\bigg\}.
\end{equation*}
If $x\in \reg(V)$ and $\spt\|V\|\cap B_{r}(x)$ is a smooth embedded hypersurface-with-boundary $\Gamma$, then we say $x$ is a \emph{regular interior point} (denoted $x\in \reg_i(V)$) if $x\in \Gamma\setminus\partial \Gamma$, and $x$ is a \emph{regular boundary point} if $x\in \partial \Gamma$ (denoted $x\in \reg_b(V)$).
Moreover, we define its complement to be the \textit{singular set} $\sing(V)$ of $V$ by
\begin{equation*}
	\sing(V)\coloneqq \spt \|V\|\setminus\reg(V).
\end{equation*}
As we mentioned in \Cref{rmk: true singular set}, given a singular properly embedded free-boundary minimal hypersurface $\Sigma$, its closure $\overline{\Sigma}$ in $M$, $\Scal=\overline{\Sigma}\setminus\Sigma$, and the corresponding multiplicity-one varifold $V=|\Sigma|$, we can redefine $\Scal$ and $\Sigma$ so that
\[
	\reg(V)=\Sigma, \quad \reg_i(V)=\Sigma \setminus \partial \Sigma, \quad  \reg_b(V)=\partial \Sigma, \quad \sing(V)=\Scal.
\]

By (1)-(2) in \Cref{def: fb smooth immersed min hyp} and the discussion in \Cref{sec: capacity and extension of stationarity}, the multiplicity-one varifold $V$ associated to a singular properly embedded free-boundary minimal hypersurface $\Sigma$ also satisfies an additional variational property; indeed such a  $V$ is \emph{free-boundary stationary} in $M$.
See \cite{LiPariseSarnataro2023} for more details on free-boundary stationary varifolds. Some of the basic properties of free-boundary stationary varifolds which are needed in this paper are also briefly recalled in \Cref{sec: tangent cones stratification}.

For the multiplicity-one varifold $V$ associated to a singular properly embedded free-boundary minimal hypersurface $\Sigma$, we can easily compute the first variation $\delta V(X)$ for \emph{all} compactly supported vector fields $X\in \mathfrak X(M)$.

\begin{lemma}\label{lemma: first variation trace}
	Let $\Sigma$ be a singular properly embedded free-boundary minimal hypersurface in $M$ and let $V=|\Sigma|$ be its associated multiplicity-one varifold. Then
	\begin{equation}\label{eq: first variation singular FB}
		\delta V(X)=-\int_{\partial \Sigma}\mathrm{g}(X, \eta)\, d\Hcal^{n-1} \quad \text{for all } X\in \mathfrak X(M),
	\end{equation}
	where $\eta$ denotes the inner unit conormal of $\partial \Sigma$ in $\Sigma$. Hence, for all $K\Subset M$, if $K\Subset\operatorname{int} K'\Subset M$, then 
	\begin{equation}\label{eq: locally finite deltaV}
		\sup\left\{|\delta V(X)|: X\in \mathfrak X(M), |X|_{\mathrm{g}}\leq 1, \spt X\subset K\right\}\leq C(K)\Hcal^{n}(\Sigma \cap K'),
	\end{equation}
	i.e. the integral varifold $|\Sigma|$ has locally bounded first variation.
\end{lemma}

\begin{proof}
Notice that for all vector fields $X$ supported away from the singular set $\Scal$, \eqref{eq: first variation singular FB} holds because of \Cref{def: fb smooth immersed min hyp}. On the other hand, for a general $X\in \mathfrak{X}(M)$, if we let $\{\chi_j\}$ be the sequence of test functions in \Cref{cor: cap}, then $X_j\coloneqq \chi_j X$ is supported away from the singular set, and the same excision argument as in the proof of \Cref{lemma: trace inequality} proves the general result.
\end{proof}

\begin{remark}
Notice that our proof crucially uses the property that $V$ is induced by a singular properly embedded free-boundary minimal hypersurface. On the other hand, \eqref{eq: locally finite deltaV} holds in general for free-boundary stationary varifolds, albeit with a more involved proof. The proof under the assumption of integrality is in \cite[Proposition 3.2]{Edelen2020}, and in full generality in \cite[Theorem 1.1]{Masi2020}.
\end{remark}

The results that follow, although stated for varifolds induced by singular properly embedded free-boundary minimal hypersurfaces, hold for general free-boundary stationary integral varifolds and sequences thereof. See \cite{LiPariseSarnataro2026} for a detailed treatment at this level of generality.

In this paper, we chose to restrict our treatment to varifolds induced by singular properly embedded free-boundary minimal hypersurfaces for the sake of simplicity and to offer a (more) self-contained presentation. 
Indeed, this special case will be sufficient for the applications in this paper, provided we supplement it with the following (elementary) observation.

\begin{remark}\label{rmk: extension to limits}
	The conclusions of  \Cref{prop: compactness stationary FB}, \Cref{lemma: density bound}, \Cref{prop: existence VarTan}, and \Cref{prop: stratification} also hold for varifold limits of sequences of singular properly embedded free-boundary minimal hypersurfaces in $M$ with locally bounded mass. (This will be needed in \Cref{subsec: top stratum} and \Cref{subsec: slicing argument}.)
\end{remark}

By Allard's Compactness Theorem \cite[6.4]{Allard1972} and \Cref{lemma: density bound}, we have the following.
\begin{proposition}\label{prop: compactness stationary FB}
	Let $(\Sigma_k)$ be a sequence of singular properly embedded free-boundary minimal hypersurfaces in $M$ and, for each $k$, let $V_k=|\Sigma_k|$ be the associated multiplicity-one varifold. Assume that for each $K\Subset M$
	\begin{equation}
		\sup_k\|V_k\|(K)\leq C(K)<+\infty.
	\end{equation}
	Then, up to a subsequence (not relabelled), 
	there is a free-boundary stationary integral varifold $V_\infty$ on $M$ such that 
	\begin{equation}
		V_k \to V_\infty \quad\text{in the sense of varifolds.}
	\end{equation}
	In particular, $\delta V_k \to \delta V_\infty$ and the limit has locally bounded first variation. Moreover, 
	\begin{equation}\label{eq: Hausdorff convergence}
		\spt \|V_k\| \to \spt \|V_\infty\|
	\end{equation}
	locally in Hausdorff distance.
\end{proposition}

\begin{proof}
The first variation hypothesis in \cite[6.4]{Allard1972} is fulfilled, in view of \Cref{lemma: first variation trace}; this yields the subsequential limit $V_\infty$. The only conclusion that needs justification is \eqref{eq: Hausdorff convergence}, which follows from the standard proof for stationary integral varifolds (as remarked e.g. in \cite[Remark 10]{SchoenSimon81}) once the following lemma is established.
\end{proof}

\begin{lemma}\label{lemma: density bound}
	Let $\Sigma$ be a singular properly embedded free-boundary minimal hypersurface in $M$ and let $V=|\Sigma|$ be the associated multiplicity-one varifold in $M$. Then for each compact $K\subset M$ there exist constants $c_K>0, C_K, r_K>0$ such that 
	\begin{equation}\label{eq: density bound}
		c_Kr^n\leq \|V\|(B_r(y))\leq C_K r^n
	\end{equation}
	for all $y\in\spt \|V\|\cap K$, $0<r<r_K$.
\end{lemma}

\begin{proof}
Let $K\subset M$ be compact and pick $r_K>0$ small enough that 
\begin{itemize}
	\item if $K\subset M\setminus \partial M$, $2r_K \leq \inf_{x\in K}\inj_x(M)$;
	\item if $K\cap \partial M\neq \emptyset$, $F_{\partial M}$ is a diffeomorphism of $(K\cap \partial M)\times [0, 2r_K)$ onto the collar neighbourhood $\Ncal^+_{2r_K}(K\cap \partial M)$ and $2 r_K\leq \inf_{x\in K\setminus \Ncal^+_{2r_K}(K\cap \partial M)}\inj_x(M)$.
\end{itemize}
Then notice that by \cite[\S 17]{Simon1983} and \cite[Theorem 3.4]{Guang2020} there exist constants $\Lambda^\mathrm{int}_K, \Lambda^\mathrm{bd}_K$ depending on $\sup_K|\Riem_\mathrm{g}|$ and $\sup_{\Ncal^+_{2r_K}(K\cap \partial M)}|\nabla^2_\mathrm{g} d_{\partial M}|$ such that for all $y\in K$
\begin{equation}\label{eq: interior monotonicity}
e^{\Lambda^\mathrm{int}_K\sigma}\frac{\|V\|(B_\sigma(y))}{\sigma^n}\leq e^{\Lambda^\mathrm{int}_K\rho}\frac{\|V\|(B_\rho(y))}{\rho^n}\quad\text{whenever }B_\rho(y)\cap \partial M=\emptyset, 0<\sigma<\rho<r_K,
\end{equation}
and
\begin{equation}\label{eq: boundary monotonicity}
e^{\Lambda^\mathrm{bd}_K\sigma}\frac{\|V\|(B_\sigma(y))}{\sigma^n}\leq e^{\Lambda^\mathrm{bd}_K\rho}\frac{\|V\|(B_\rho(y))}{\rho^n}\quad\text{whenever }y\in K\cap \partial M, 0<\sigma<\rho<r_K.
\end{equation}

For simplicity we shall prove the density ratio lower bound in \eqref{eq: density bound}. The proof of the upper bound follows a similar strategy.
Choose $c_K=16^{-n}\omega_n e^{-2\max\{\Lambda^\mathrm{int}_K, \Lambda^\mathrm{bd}_K\}r_K}>0$, and consider first a regular point $y\in \Sigma$.
If $y\in \partial \Sigma$, then $\Theta(\|V\|, y)=1/2$ and, by \eqref{eq: boundary monotonicity}, we can conclude \eqref{eq: density bound} holds for $0<r<r_K$.
Similarly if $y\in (\Sigma\setminus \partial\Sigma) \cap \partial M$.

Now fix $r\in(0, r_K)$. Since $y\in \Sigma\setminus \partial \Sigma$, then $\Theta(\|V\|, y)=1$ and $d\coloneqq \dist_{\mathrm{g}}(y, \partial M)>0$. If $d\geq r/4$, then $B_{r/8}(y)\Subset M\setminus \partial M$ and \eqref{eq: density bound} follows from \eqref{eq: interior monotonicity}. On the other hand, if $0<d< r/4$, then let $\hat y=\pi_{\partial M}(y)$ be the nearest boundary point to $y$. By interior monotonicity at $y$ at scale $d/2$, $\|V\|(B_{d/2}(y))\geq c_K d^n$. Since $B_{d/2}(y)\subset B_{2d}(\hat y)$, $\|V\|(B_{2d}(\hat y))\geq c_K(2d)^n$. By boundary monotonicity \eqref{eq: boundary monotonicity} from scale $2d$ to scale $r/2$,
\[
	\frac{\|V\|(B_{r/2}(\hat y))}{(r/2)^n}\geq c_K\frac{\|V\|(B_{2d}(\hat y))}{(2d)^n}\geq c_K
\]
after absorbing exponential factors in the constant.
Since $B_{r/2}(\hat y)\subset B_r(y)$, then \eqref{eq: density bound} follows.

Finally, for a singular point $y\in \Scal$, choose regular points $\Sigma\ni y_\ell\to y$, and notice that for all sufficiently large $\ell$, $B_{r/2}(y_\ell)\subset B_r(y)$, so \eqref{eq: density bound} at $y$ follows from the statement for regular points we have just proven.
\end{proof}

\begin{remark}
    A quantitative version of the density lower bound in \Cref{lemma: density bound} follows from \cite[Theorem 9.22]{MenneScharrer}.
\end{remark}

\begin{remark}\label{rmk: Allard metric convergence}
	It is easy to see that the conclusion of \Cref{prop: compactness stationary FB} still holds if each $\Sigma_k$ is a singular properly embedded free-boundary minimal hypersurface in $(M, \mathrm{g}_k)$ where the metrics $\mathrm{g}_k$ converge in $C^2_{\mathrm{loc}}$ to a smooth metric $\mathrm{g}$ on $M$.
\end{remark}

The notion of a \emph{tangent cone} will be central to our analysis. We shall provide the definition and recap the basic properties of tangent cones to free-boundary stationary varifolds in \Cref{sec: tangent cones stratification}. For the purposes of the proof of \Cref{thm: compactness}, the following proposition will suffice.
\begin{proposition}\label{prop: existence VarTan}
Let $\Sigma$ be a singular properly embedded free-boundary minimal hypersurface in $M$ and let $V=|\Sigma|$ be the associated multiplicity-one varifold. If $x\in \spt\|V\|=\overline{\Sigma}\subset M$, then $\VarTan(V, x)\neq \emptyset$. Moreover, 
\begin{enumerate}
	\item if $x\in M\setminus \partial M$ and $\bC\in \VarTan(V, x)$, then $\bC$ is a stationary integral cone in $(T_xM, \mathrm{g}_{\mathrm{Euc}})$ of density $\Theta(\|V\|, x)$;
	\item if $x\in \partial M$ and $\bC\in \VarTan(V, x)$, then $\bC$ is a free-boundary stationary integral cone in $(T^+_xM, \mathrm{g}_{\mathrm{Euc}})$ of density $\Theta(\|V\|, x)$, and $\obC = \bC+(\theta_0)_\sharp\bC$ is a stationary integral cone in $(T_xM, \mathrm{g}_{\mathrm{Euc}})$ which is invariant under $\theta_0$, where $\theta_0$ denotes the reflection across $T_x(\partial M)\subset T_xM$ in $(T_xM, \mathrm{g}_{\mathrm{Euc}})$.
\end{enumerate}
\end{proposition}

Finally, we shall need a version of Almgren's stratification for free-boundary stationary integral varifolds $V$. See \Cref{sec: tangent cones stratification} for the relevant notation and the statement of the relevant stratification result and \cite{LiPariseSarnataro2026} for a proof.

\bigskip

We can now proceed with the proof of the Compactness Theorem (\Cref{thm: compactness}). Fix a compact $K\subset M$ and assume without loss of generality (by \Cref{prop: compactness stationary FB}) that the multiplicity-one varifolds $V_k$ associated with $\Sigma_k$ converge in $K$ to a free-boundary stationary varifold $V$. We then begin by analysing the possible tangent cones of $V$.

\subsection{Top-dimensional stratum}\label{subsec: top stratum} We first analyse the top-dimensional part of the singular stratum of $V$, i.e. those (a priori, singular) points such that at least one tangent cone $\mathbf C$ is such that $\hbC$ has a spine of dimension $n$.

\begin{proposition}[Regularity at flat tangent cones]\label{prop: top stratum}
	Let $x\in\spt\|V\|$ and assume $\bC\in \VarTan( V, x)$ is such that $\dim S(\hbC)=n$. Then $x$ is a regular point of $V$ and $\VarTan(V,x)=\{\bC\}$, i.e.  the tangent cone at $x$ is unique.
\end{proposition}

\begin{proof}
For the sake of brevity, we shall focus on the case of $x\in \spt\|V\|\cap\partial M$. The interior case is almost identical to (and, indeed, slightly simpler than) (1) in the dichotomy below. So, for the rest of this proof, let  $x\in \spt\|V\|\cap\partial M$, and, for simplicity, let $H=T_xM$, $H_+=T^+_xM$, and $T=\partial H_+=T_x\partial M$. 

By the constancy theorem \cite{Simon1983}, $\overline{\mathbf{C}} = Q \vert \Pi \vert$ for some $n$-plane $\Pi$. By \Cref{prop: existence VarTan}, the doubled cone $\overline{\mathbf{C}}$ is invariant under the reflection $\theta_0$ across $T$. Hence, we know that the normal to $\Pi$ is an eigenvector of the reflection map.
Consequently, there are two options: 
\begin{enumerate}
   \item \textit{$\Pi = T_x\partial M$, in which case $Q$ is even and $\mathbf{C} = (Q/2) \vert T_x\partial M \vert$.} \\
Let the notation be as in \Cref{sec: tangent cones stratification}, so for $r>0$ the map $\Psi_{x, r}:H_+ \cap  B^{H}_{r^{-1}\rho_x}(0)\to M$ is the rescaling map in \eqref{eq: rescaling boundary} with inverse $\Phi_{x, r}$, and $\mathrm{g}_{x, r}=r^{-2}\Psi_{x, r}^*\mathrm{g}$ as a metric on (a subset of) $H_+$.
By definition (see \Cref{subsec: tangent cones} together with \Cref{rmk: extension to limits}), there exist positive numbers $r_i\to 0$ such that the free-boundary stationary integral varifolds $W_i\coloneqq (\Phi_{x, r_i})_\sharp V$ in $(H_+, \mathrm{g}_i)$ converge locally to $(Q/2) \vert T \vert\subset H_+$, where for simplicity we let $\mathrm{g}_i\coloneqq \mathrm{g}_{x, r_i}$. Let $d_i$ denote the $\mathrm{g}_i$-distance function to $T$ in $(H_+, \mathrm{g}_i)$. Moreover, for each $i$ and $R>0$, let $P^i_R$ be the intrinsic ball in $(T, \mathrm{g}_i|_{T})$ of radius $R$ centred at $0$ and $P_R$ the Euclidean ball in $T$. 
Since $\mathrm{g}_i\to \mathrm{g}_{\mathrm{Euc}}$ locally smoothly, given any $\rho_0>0$, there exists $i_0\geq 1$ such that 
\begin{equation}
\text{\parbox{0.85\textwidth}{\centering
		the normal exponential map of $T_x(\partial M)$ is a diffeomorphism of $P^i_{8\rho_0}\times[0, 8\rho_0)$ and $|\nabla^2_{\mathrm{g}_i}d_i|^2+ |\nabla^3_{\mathrm{g}_i}d_i|+|\Riem_{\mathrm{g}_i}|\leq 1$ for all $i\geq i_0$.
}}
\label{eq: uniform geometric bounds}
\end{equation}	
By \Cref{prop: compactness stationary FB} (and \Cref{rmk: extension to limits}), $\spt \|W_i\| \to T$ locally in Hausdorff distance. As $\mathrm{g}_i\to \mathrm{g}_{\mathrm{Euc}}$, for any compact set $K\subset H_+$, we have
\begin{equation}\label{eq: height to 0 parallel}
	\sup_{\spt\|W_i\|\cap K}d_i \to 0 \quad\text{as }i\rightarrow\infty\footnote{Here we are using that since $\mathrm{g}_i\to \mathrm{g}_{\mathrm{Euc}}$ locally smoothly, for any given compact set $K$, any given $\epsilon\in(0,1)$ and all sufficiently large $i$ the metrics satisfy $(1-\epsilon)\mathrm{g}_i\leq \mathrm{g}_{\mathrm{Euc}}\leq (1+\epsilon)\mathrm{g}_i$ on $K$. We shall use this fact repeatedly throughout the proof.}.
\end{equation}
Fix $\rho_0\in(0, 1]$ and the corresponding $i_0\geq 1$, and notice that by \eqref{eq: uniform geometric bounds}, the threshold radius $R_0$ and the threshold normalised tilt excess $\varepsilon_{0}$ in \Cref{thm: boundary sheeting} can be chosen to hold uniformly for all $i\geq i_0$. By varifold convergence to $(Q/2) \vert T \vert$, we can assume that for all $i \geq i_0$ we have
\begin{equation}\label{eq: mass bound parallel}
\|W_i\|(\Ncal^+_{4\rho_0}(P^i_{4\rho_0}))\leq C(n, Q).
\end{equation}
For fixed $R\in(0, R_0/2)$, by \eqref{eq: mass bound parallel}, we can choose $\delta\in(0, \rho_0/2)$ such that the right hand side of \eqref{eq: tilt-height excess estimate} with $r=2R$ is strictly less than $\varepsilon_0^2/2$ for all $i\geq i_0$.
By \eqref{eq: height to 0 parallel} with $K=\overline{\Ncal^+_{2\rho_0}}(P_{2R})$, by possibly increasing $i_0$, we can guarantee that for all $i \geq i_0$,
\begin{equation*}
	\spt \|W_i\|\cap \Ncal^+_{2\rho_0}(P^i_{2R})\subset \Ncal^+_{\delta/2}(P^i_{2R})
\end{equation*}
Now fix the scale $r_*=r_{i_0}$ and write $\mathrm{g}_*\coloneqq \mathrm{g}_{x, r_{i_0}}$, $P^*_R\coloneqq P^{i_0}_R$, $W^*=W_{i_0}$, and $W^*_k\coloneqq (\Phi_{x, r_{i_0}})_\sharp V_k$. Notice that $W^*_k \to W^*$ locally as varifolds in the fixed Riemannian manifold-with-boundary $(H_+, \mathrm{g}_*)$. By \Cref{prop: compactness stationary FB} (and \Cref{rmk: extension to limits}), for all sufficiently large $k$
\begin{equation*}
	\spt \|W^*_k\|\cap \Ncal^+_{2\rho_0}(P^*_{2R})\subset \Ncal^+_{\delta}(P^*_{2R}),
\end{equation*}
and $W^*_k$ satisfies a mass bound analogous to \eqref{eq: mass bound parallel}.

Therefore, by \Cref{prop: boundary titl height estimate}, we obtain that, for all sufficiently large $k$, we have $$E_{W^*_k, P^*}(2R)^2\leq \varepsilon_0^2/2,$$ and \Cref{thm: boundary sheeting} applies. In particular, (ii) of \Cref{thm: boundary sheeting} applies and for all sufficiently large $k$ (by uniform equivalence of the metrics)
\begin{equation}\label{eq: preliminary sheeting parallel} 
	\spt \|W^*_k\|\cap \Ncal^+_{2\rho_0}(P_{R/2})=\bigcup_{j=1}^{Q_k}F_{\partial M}(\operatorname{graph}(u_{k,j}))
\end{equation}
for some $Q_k\in \mathbb{N}$ and smooth functions $0\leq u_{k,1}<\dots<u_{k, Q_k}$, where $F_{\partial M}$ is the normal exponential map in the fixed rescaled metric. Moreover, $Q_k$ is bounded above uniformly by some $\hat Q$ because of the uniform mass bound. 

Finally, the gradient estimate in \Cref{thm: boundary sheeting}(ii) implies that, after passing to a subsequence, for each $j=1, \dots, \hat Q$, the functions $u_{k,j}$ converge to some function $u_j$ in $C^{2,\alpha}$ and $0\leq u_1\leq \dots\leq u_{\hat{Q}}$ on $P_{R/4}$. Each $u_j$ satisfies the $\mathrm{g}_*$-minimal surface equation on $P_{R/4}$,  so if any two graphs touch they must coincide. After shrinking the neighbourhood further if needed, this shows that the support of $W^*$ in a neighbourhood of $0$ is given by a single smooth embedded minimal hypersurface, possibly with multiplicity, which must then be given by the density $Q/2$ at $0$.
After undoing all the rescalings, this shows that $x$ is a regular point of $V$, and that the tangent cone is unique. It also gives local smooth graphical convergence of the hypersurfaces $\Sigma_k$.

\medskip

\item \textit{$\Pi \subset T_xM$ is an $n$-plane invariant under reflection across $T_x\partial M$ and $\Pi \neq T_x \partial M$. \\}
Notice that then $\bC=Q|\Pi^+|$. Let $L=\Pi\cap T$ and notice that $\Pi=L\oplus\operatorname{span}\{\bn(x)\}$. For $\varrho\in(0, \inj_x(\partial M))$, let $\Gamma\coloneqq\exp^{\partial M}_x(L\cap B_{\varrho}(0))$, so that $T_x\Gamma = L$. Note that $\Gamma$ is a smooth $(n - 1)$-dimensional embedded disc in $\partial M$. By shrinking $\varrho$ if needed, the normal exponential map $F_{\partial M}$ is a diffeomorphism of $\Gamma\times [0, \varrho)$ into a smooth embedded two-sided $n$-dimensional disc $P$ in $M$, with $\partial_0P=\Gamma$. By construction, $P$ meets $\partial M$ orthogonally along $\Gamma$, and $T_y P = T_y \Gamma \oplus \mathrm{span}\{\bn(y)\}$ for all $y\in \Gamma$. At $x$, we have $T_xP=\Pi$.
By \Cref{lemma: adapted neighbourhood} and \Cref{lemma: quantitative adapted neighbourhood}, there is $\varrho_1>0$ such that $P$ has an adapted $\varrho_1$-neighbourhood, and the corresponding $z_P, Z$ as in \Cref{def: adapted neighbourhood} satisfy the estimates in \Cref{lemma: quantitative adapted neighbourhood} for some $\Lambda>0$.

Now, as in (1), let $r_i\to 0$ be such that the free-boundary stationary integral varifolds $W_i\coloneqq (\Phi_{x, r_i})_\sharp V$ in $(H_+, \mathrm{g}_i)$ converge locally to $Q \vert \Pi^+ \vert\subset H_+$. Let $P^i\coloneqq \Phi_{x, r_i}(P)$ and notice that $P^i$ meets $T$ orthogonally in $(H_+, \mathrm{g}_i)$ along $\partial_0P^i$. 
Also, define $z_i\coloneqq r_i^{-1}z_P\circ \Psi_{x, r_i}$, and let $Z_i\coloneqq |\nabla^{\mathrm{g}_i}z_i|^{-1}_{\mathrm{g}_i}\nabla^{\mathrm{g}_i}z_i$. It is easy to check that the flow of $Z_i$ still defines an $r_i^{-1}\varrho_1$-adapted neighbourhood $U_i$ of $P^i$ in $(H_+, \mathrm{g}_i)$ and that $P^i\to \Pi^+, \quad z_i\to z_\Pi, \quad Z_i \to \nu_\Pi$
locally smoothly, where $\nu_\Pi$ is a (Euclidean) unit normal to $\Pi$, and $z_\Pi$ is the linear function vanishing on $\Pi$ with gradient $\nu_\Pi$. The corresponding flows also converge locally smoothly.
Finally, for $R>0$ and all $i$ sufficiently large, let $P^i_R$ be the intrinsic ball in $(P^i, \mathrm{g}_i|_{P^i})$ of radius $R$ centred at 0, which converges to the Euclidean ball $P_R$ in $\Pi^+$.
Again, since $\mathrm{g}_i\to \mathrm{g}_{\mathrm{Euc}}$ locally smoothly, there exists $i_1\geq 1$ such that \eqref{eq: geometric bounds orthogonal} holds with $\Lambda=1$ for all $i\geq i_1$. 

As in (1), by \Cref{prop: compactness stationary FB} (and \Cref{rmk: extension to limits}), $\spt \|W_i\| \to \Pi^+$ locally in Hausdorff distance. As $\mathrm{g}_i\to \mathrm{g}_{\mathrm{Euc}}$, for any compact set $K\subset H_+$, we have
\begin{equation}\label{eq: height to 0 orthogonal}
	\sup_{\spt\|W_i\|\cap K}|z_i| \to 0 \quad\text{as }i\rightarrow\infty\footnote{Again we are using uniform equivalence to $\mathrm{g}_{\mathrm{Euc}}$, in addition to the properties of $z_i$.}.
\end{equation}
Fix $\rho_1>0$ and the corresponding $i_1\geq 1$, and notice that the threshold distance $\delta_1$, the threshold radius $R_1$ and the threshold normalised tilt excess $\varepsilon_{1}$ in \Cref{thm: orthogonal sheeting} can be chosen to hold uniformly for all $i\geq i_1$. Similarly to (1), by varifold convergence we have uniform mass bounds on any fixed compact set, so for a given $R\in(0, R_1/2)$, we can choose $\delta\in(0, \delta_1/2)$ such that the right hand side of \eqref{eq: orthogonal tilt height estimate} with $r=2R$ is strictly less than $\varepsilon_1^2/2$ for all $i\geq i_1$.

Now the same argument as in (1) allows us to fix a scale $r_*=r_{i_1}$ for which we can apply the orthogonal sheeting theorem (\Cref{thm: orthogonal sheeting}(ii)) to rescalings of the varifolds $V_k=|\Sigma_k|$ and deduce that for all large $k$, the support of the rescalings $W^*_k$ in $\Ncal^{Z^*}_{2\rho_1}(P_{R/2})$ is a union of $Q$ smooth $F_{Z^*}$-graphs over $P_{R/2}$. The gradient estimates in \Cref{thm: orthogonal sheeting}(ii) show that the graphs converge in $C^{2, \alpha}$ to graphs whose union induces $W^*$ on a smaller neighbourhood. A priori these graphs are weakly ordered but since they all satisfy a uniformly oblique boundary condition, the boundary and interior strong maximum principles for the minimal surface equation imply that  any two graphs which touch must coincide. Therefore, by shrinking the neighbourhood further if needed, the support of $W^*$ in a neighbourhood of $0$ is given by a single smooth embedded free-boundary minimal hypersurface, possibly with multiplicity. Again as in (1), this shows that $x$ is a regular point of $V$, and that the tangent cone is unique, together with local smooth graphical convergence of the $\Sigma_k$. 
\end{enumerate}
\end{proof}

\subsection{Excluding boundary classical cones: slicing argument}\label{subsec: slicing argument}  We can now move to the $(n - 1)$-dimensional part of the singular stratum. 
We will use the argument of \cite{SchoenSimon81}. Since the interior case is completely analogous to \cite{SchoenSimon81}, we shall focus on the boundary case, i.e. $x\in \spt \|V\| \cap \partial M$. Again, as in the proof of \Cref{prop: top stratum}, we let $H=T_xM$, $H_+=T^+_xM$, and $T=\partial H_+=T_x\partial M$. 

Before we proceed with the argument we need the following elementary lemma, which shows that if $\bC$ is a tangent cone of a stable singular properly embedded free-boundary minimal hypersurface $\Sigma$ at some $x\in \overline{\Sigma}=\spt |\Sigma|$, then the regular part of $\bC$ is a stable free-boundary minimal cone (for the Euclidean metric on $T_xM$). Similarly for tangent cones to varifold limits of stable singular properly embedded free-boundary minimal hypersurfaces.

\begin{lemma}
Let $V$ be the multiplicity-one varifold induced by a stable singular properly embedded free-boundary minimal hypersurface $\Sigma$, or the varifold limit of a sequence of such objects. Let $x\in \spt \|V\|$ and let $\bC\in \VarTan(V,x)$. Then $\reg(\bC)$ is a stable smooth (possibly free-boundary) minimal hypersurface in $(T_xM, \mathrm{g}_{\mathrm{Euc}})$. 
\end{lemma}

\begin{proof}
Let the notation be as in \Cref{sec: tangent cones stratification} and the proof of \Cref{prop: top stratum}. For simplicity, let us assume $V=|\Sigma|$ for a stable singular properly embedded free-boundary minimal hypersurface $|\Sigma|$ and $x\in \overline\Sigma\cap \partial M$. By definition, there exists a sequence $r_i\to 0$ such that the multiplicity-one varifolds $|\Sigma_i|$ induced by stable singular free-boundary minimal hypersurfaces $\Sigma_i=\Phi_{x, r_i}(\Sigma)$ in the rescaled metrics $\mathrm{g}_i$ converge locally to $\bC$. By tracking how the geometric quantities involved transform, we see that, for every $\varphi\in C_c^{0,1}(\Sigma_i)$, we have
\begin{equation}\label{eq: stability rescaled}
	\begin{aligned}
		\int_{\Sigma_i}|\nabla^{\mathrm{g_i}}_{\Sigma_i}\varphi|^2_{\mathrm{g}_i}- &  r_i^2\left((|A_\Sigma|^2_\mathrm{g}+\Ric_\mathrm{g}(\nu_\Sigma, \nu_\Sigma))\circ\Psi_{x, r_i}\right)\varphi^2 \, d\Hcal_{\mathrm{g}_i}^n\\
		&-r_i\int_{\partial \Sigma_i}(\sff^\mathrm{g}_{\partial M}(\nu_\Sigma, \nu_\Sigma)\circ\Psi_{x, r_i})\varphi^2\, d\Hcal_{\mathrm{g}_i}^{n-1}\geq 0.
	\end{aligned}
\end{equation}
By \Cref{cor: cap}, \eqref{eq: stability rescaled} holds for all compactly supported ambient Lipschitz test functions.

Now let $y\in \reg(\bC)$. Then any tangent cone $\mathbf D$ to $\bC$ at $y$ is a constant multiplicity hyperplane or orthogonal half-plane. Therefore, after applying \Cref{prop: top stratum} to a diagonal sequence converging to $\mathbf{D}$, we see that the free-boundary minimal hypersurfaces in the diagonal sequence converge smoothly and graphically to $\mathbf{D}$ on a neighbourhood of $y$. In particular, we can pass \eqref{eq: stability rescaled} to the limit in a neighbourhood of $y$ and obtain that $\reg(\bC)$ satisfies the (Euclidean) stability inequality 
\begin{equation}
	\int_{\reg(\bC)}|A_{\bC}|^2 \varphi^2\leq \int_{\reg(\bC)}|\nabla\varphi|^2
\end{equation}
for all compactly supported Lipschitz test functions $\varphi$.

Finally, the case where $V$ is a varifold limit of stable singular properly embedded free-boundary minimal hypersurfaces follows from the same proof after taking a further diagonal subsequence.
\end{proof}

We are now ready to prove the main result in this subsection, namely that for a varifold limit $V$ of a sequence of stable singular properly embedded free-boundary minimal hypersurfaces the $(n-1)$-dimensional part of the singular stratum $\mathscr{S}_{n-1}\setminus\mathscr{S}_{n-2}$ is empty. Another way to formulate this is that if $\bC$ is a tangent cone to such a varifold $V$ and $\dim_{\mathcal{H}}S(\hbC)\geq n-1$, then actually $\dim_{\mathcal{H}}S(\hbC)= n$ and \Cref{prop: top stratum} applies.

\begin{proposition}
Let $V$ be a varifold limit of a sequence $\{\Sigma_k\}_k$ of stable singular properly embedded free-boundary minimal hypersurfaces in $M$. Then the $(n-1)$-dimensional part of the singular stratum $\mathscr{S}_{n-1}\setminus\mathscr{S}_{n-2}$ of $V$ is empty.
\end{proposition}

\begin{proof}
As we stated above, we shall focus on the case where $x\in \spt\|V\|\cap \partial M$. As in the proof of \Cref{prop: top stratum}, we let $H=T_xM$, $H_+=T^+_xM$, $T=\partial H_+$, and 
$H_+^\circ=H_+\setminus T$.

Let $\bC\in \VarTan(V, x)$ and assume that $\overline{\mathbf{C}}$ has an $(n - 1)$-dimensional spine, denoted by $L$. Since $\overline{\mathbf{C}}$ is invariant under reflection $\theta_0$ across $T = \partial H_+$, then $L$ is invariant under $\theta_0$. As the $(-1)$-eigenspace of the reflection is one-dimensional, exactly one of the following occurs: either $L \subset T$, or $\mathbf{n}_x \in L$, where $\mathbf{n}_x$ is the inward unit normal to $T$. 

If $\mathbf{n}_x \in L$, choose a nonzero point $z \in H_+^\circ\cap L$. In a sufficiently small ball centred at $z$, the reflected part of the cone lies outside $H_+$, and hence $\mathbf{C}$ agrees  with the full cone $\overline{\mathbf{C}}$ there. Since $\overline{\mathbf{C}}$ is invariant along $L$, its tangent cone at $z$ is again a classical cone with $(n - 1)$-dimensional spine. The usual interior slicing argument of \cite{SchoenSimon81} gives the contradiction. Hence, it remains to treat the genuinely free-boundary case $L \subset T$. 

Choose coordinates on $H\cong\R^{n+1}$ so that $H_+ = \{x_1 \geq 0\}$ and $T = \partial H_+ = \{x_1 = 0\}$. After rotation, assume that $L$ is given by $\{x_1 = x_{n + 1} = 0\} = \{0\} \times \mathbb{R}^{n - 1} \times \{0\} \subset T$. Write a point as $X = (x_1, y, x_{n + 1})$, where $y = (x_2, \ldots, x_{n}) \in \mathbb{R}^{n - 1}$. For $\alpha \in [0, \pi]$, define the $n$-dimensional half-hyperplane $$H^+_{\alpha} = \{(r \sin(\alpha), y, - r \cos(\alpha)) \; \colon \; r \geq 0, \, y \in \mathbb{R}^{n - 1} \}.$$ Then $\partial H^+_\alpha = L$ for every $\alpha$. Furthermore, $H^+_0, H^+_\pi \subset T$ (actually $T = H^+_0 \cup H^+_\pi$), while $H^+_{\pi/2} = \{x_{n + 1} = 0, x_1 \geq 0\}$ meets $T$ orthogonally. Set $\omega_\alpha = \sin(\alpha) e_1 - \cos(\alpha) e_{n + 1}$, so that $\omega_\alpha$ points from $L$ into $H_\alpha^+$. A choice of unit normal to $H^+_\alpha$ is given by $\nu_\alpha = \cos(\alpha)e_1 + \sin(\alpha) e_{n + 1}$, while the outward conormal of $H^+_\alpha$ along $L$ is $\eta_\alpha = - \omega_\alpha = - \sin(\alpha)e_1 + \cos(\alpha) e_{n + 1}$. Consequently, since the cone is invariant along $L$, it has the form
\begin{equation*}
    \mathbf{C} = q_0 \vert H^+_0 \vert + \sum_{i = 1}^m p_i \vert H^+_{\alpha_i} \vert + q_\pi \vert H^+_\pi \vert, 
\end{equation*}
with $0 < \alpha_1 < \alpha_2 < \ldots < \pi$, $p_i \in \mathbb{N}$, $q_0, q_\pi \in \mathbb{N} \cup \{0\}$. Because $\mathbf{C}$ is only free-boundary, we do not have the full vector balancing equation. However, testing the first variation near $L$ with a vector field equal to $e_{n + 1}$ which is tangent to $T$ gives the identity $q_0 - q_\pi + \sum_{i = 1}^m p_i \cos(\alpha_i) = 0$. There is no corresponding balancing condition for the $e_1$-component. 

We first prove that $\alpha_i = \pi/2$ for every $i$. Choose a diagonal sequence realising the tangent cone $\bC$. Thus, after relabelling, $\Sigma_j = \Phi_{x, r_j}(\Sigma_{k_j})$ and $V_j = (\Phi_{x, r_j})_\sharp V_{k_j}$; then $V_j\to \bC$ locally as varifolds. We also have the corresponding sequence of Riemannian metrics as in \Cref{prop: top stratum}. Every point of $\spt\|\mathbf{C}\| \setminus L$ has a planar tangent cone. On the one hand, at points of $H^+_{\alpha_i} \setminus L$, with $0 < \alpha_i < \pi$, we are in the interior of $H$, and hence interior sheeting \cite{SchoenSimon81} applies. On the other hand, at points of $(H^+_0 \cup H_\pi^+) \setminus L$, boundary sheeting (cf. \Cref{thm: boundary sheeting}) applies. Consequently, $\Sigma_j \rightarrow \mathbf{C}$ smoothly and graphically on compact subsets of $\overline{H}_+\setminus L$. In other words, on every fixed annulus around the spine $L$, the graphical sheets, their tangent planes, and their unoriented normals converge (locally) uniformly to those of $\mathbf{C}$. 

For $\tau > 0$, define the following $\tau$-tubular neighbourhood $\mathcal{L}_\tau$ of the spine by 
\begin{equation*}
    \mathcal{L}_\tau = \{(x_1, y, x_{n + 1}) \in B_{3/2}^+ \; \colon \; x_1^2 + x_{n + 1}^2 < \tau^2, \vert y \vert < 1\}.
\end{equation*}
Choose $\zeta\in C_c^1(B_2)$ such that $  0\leq\zeta\leq1,$ $ \zeta\equiv1 $ on $B_{3/2},$  and $|D\zeta|\leq 3.$ Since the metrics are converging locally in $C^2$, applying the stability inequality \eqref{eq: stability rescaled} for $\Sigma_j$, reabsorbing the boundary terms using \Cref{lemma: trace inequality}, and then using uniform equivalence of $\mathrm{g}_j$ to $\mathrm{g}_{\mathrm{Euc}}$ gives a uniform bound $\int_{\Sigma_j \cap B_{3/2}^+} \vert A_j \vert^2 \leq C$.\\ Furthermore, cover $L\cap\{|y|<1\}$ by at most $C\tau^{1-n}$ balls of radius $2\tau$ with centres on $L\subset T$. Boundary-centred monotonicity and the uniform mass bound give mass at most $C\tau^n$ in each corresponding ball of radius $4\tau$. Hence $\mathcal H^n_{\mathrm g_j}(\Sigma_j\cap\mathcal L_\tau)\leq C\tau$, uniformly for all sufficiently large $j$ and $0<\tau<\tau_0$. 
Therefore, we obtain
\begin{align}  \label{eq: L1 curvature in spine tube}
    \int_{\Sigma_j\cap\mathcal L_\tau}
    |A_j| \leq \left(\int_{\Sigma_j\cap B_{3/2}^{+}} |A_j|^2 \right)^{1/2} \left( \mathcal H_{\mathrm g_j}^{n} \bigl(\Sigma_j\cap\mathcal L_\tau\bigr)\right)^{1/2} \leq C\tau^{1/2}.
\end{align}
Let now $\pi:H_+\rightarrow L$ be given by $\pi(x_1,y,x_{n+1})=y.$ For $y\in\mathbb R^{n-1}$, define $ P^+_y \coloneqq \pi^{-1}(y) = \left\{(x_1,y,x_{n+1}) \, \colon x_1\geq0 \right\}$ and $D_{\tau, y}^+ = \{(x_1, y, x_{n + 1}) \in P_y^+ \, \colon x_1^2 + x_{n + 1}^2 \leq \tau^2\}$. We denote the curved boundary of this half-disk by $\Lambda_{\tau, y} = \{(x_1, y, x_{n + 1}) \, \colon x_1\geq 0,\ x_1^2 + x_{n + 1}^2 = \tau^2\}$. Since $\mathcal S_j$ has locally finite $\Hcal^{n-2}$ measure, we have $\mathcal H^{n-1}(\mathcal S_j)=0,$ and hence, since $\pi$ is Lipschitz, $\mathcal H^{n-1}\bigl(\pi(\mathcal S_j)\bigr)=0.$ Consequently, for almost every $y \in B_{1/2}^{n-1}$, the point $y$ does not belong to $\pi(\mathcal{S}_j)$ and, by Sard's theorem, is a regular value of both $\pi|_{\Sigma_j\setminus\mathcal S_j}$ and $\pi|_{\partial\Sigma_j}.$ Additionally, the intersection with $\Lambda_{\tau, y}$ is transverse for all sufficiently large $j$. For all such $y$, the slice $\Gamma_{j,y} \coloneqq \Sigma_j\cap D_{\tau, y}^+$ is a one-dimensional, smooth manifold with $\partial \Gamma_{j, y} \subset (\Sigma_j \cap \Lambda_{\tau, y}) \cup (\partial \Sigma_j \cap D_{\tau, y}^+)$. Notice that a point of $(\Sigma_j \setminus \partial \Sigma_j) \cap T$ is an interior point of $\Gamma_{j, y}$ rather than an endpoint. In particular, every connected component of $\Gamma_{j, y}$ is either a smooth embedded circle, or a smooth embedded arc with exactly two endpoints. In the latter case, the unordered pair of endpoint types is one of $\{\Lambda_{\tau, y}, \Lambda_{\tau, y}\}$, $\{\Lambda_{\tau, y}, \partial \Sigma_j\}$, or $\{\partial \Sigma_j, \partial \Sigma_j\}$. No other case arises: the slice does not meet $\mathcal{S}_j$ or the curved part of $\partial B_{3/2}^+$, while properness prevents connected components from terminating or accumulating in the interior of $D_{\tau, y}^+$. 

Let $N_j$ denote a continuous choice of Euclidean unit normal to $\Sigma_j$, and regard $[N_j] \colon \Sigma_j \rightarrow \mathbb{RP}^n$ as the unoriented Gauss map. For a curve $\gamma \subset \Sigma_j$, write 
\[
\mathrm{osc}_\gamma[N_j] = \sup_{a, b \in \gamma} d_{\mathbb{RP}^n}([N_j(a)], [N_j(b)]). 
\]
Suppose for the sake of contradiction that $\alpha_{i_0} \neq \pi/2$ for some $i_0$. For every regular slice, choose one of the outer endpoints $a_{j, y} \in \Sigma_j \cap \Lambda_{\tau, y}$ corresponding to the ray $H^+_{\alpha_{i_0}}$, and let $\gamma_{j, y} \subset \Gamma_{j, y}$ be the connected component containing it. Since $\gamma_{j, y}$ has an endpoint, it is an arc. Denote its other endpoint by $b_{j, y}$. Let $\mathcal{A}_{\mathbf{C}}  \coloneqq \{\alpha_1,\ldots,\alpha_m\} \cup \{0:q_0>0\} \cup \{\pi:q_\pi>0\}$ and define
\begin{align} \label{eq: normal separation constant}
\delta_{\mathbf C}
    \coloneqq
    \min\Bigl\{d_{\mathbb{RP}^n} \bigl([\nu_{\alpha_{i_0}}], \mathbb P(e_1^\perp)\bigr), d_{\mathbb{RP}^n}\bigl([\nu_{\alpha_{i_0}}],[\nu_\alpha]\bigr): \alpha\in\mathcal A_{\mathbf C},\ \alpha\neq\alpha_{i_0}\Bigr\}, 
\end{align}
where the latter family is omitted if empty. Since $d_{\mathbb{RP}^n}([\nu_{\alpha_{i_0}}], \mathbb P(e_1^\perp)) = \vert \alpha_{i_0} - \pi/2 \vert > 0$, and we have finitely many rays, $\delta_{\mathbf{C}} > 0$. We now have three cases for $b_{j, y}$: 
\begin{enumerate}[(i)]
    \item If $b_{j,y}\in\partial\Sigma_j$, then the free-boundary
    condition gives $d_{\mathbb{RP}^n}\bigl([N_j(b_{j,y})],\mathbb P(e_1^\perp)\bigr)=o_j(1),$ while smooth convergence at the outer endpoint gives $[N_j(a_{j,y})] \rightarrow [\nu_{\alpha_{i_0}}].$
    Hence, for all sufficiently large $j$, $\operatorname{osc}_{\gamma_{j,y}}[N_j] \geq \delta_{\mathbf C}/2$. 
    \item If $b_{j,y}\in\Lambda_{\tau,y}$ corresponds to a ray
    $H_\alpha^+$ with $\alpha\neq\alpha_{i_0}$, then $[N_j(a_{j,y})]\rightarrow [\nu_{\alpha_{i_0}}],$ and $[N_j(b_{j,y})]\rightarrow[\nu_\alpha].$ Therefore, for all sufficiently large $j$, $\operatorname{osc}_{\gamma_{j,y}}[N_j]\geq \delta_{\mathbf C}/2.$
    \item Suppose finally that $b_{j,y}\in\Lambda_{\tau,y}$ corresponds to the same ray $H_{\alpha_{i_0}}^+$. Set $E=\operatorname{span}\{e_1,e_{n+1}\}$, and let $N_j$ be as above. Since $y$ is a regular value of $\pi|_{\Sigma_j}$, we have $T_p\gamma_{j,y} = \ker D(\pi|_{\Sigma_j})_p = T_p\Sigma_j\cap E,$ and $\dim T_p\gamma_{j,y}=1.$ In particular, $\operatorname{proj}_E N_j\neq0$ along $\gamma_{j,y}$: otherwise $N_j\in E^\perp$, and hence $E\subset N_j^\perp=T_p\Sigma_j$, contradicting the fact that the kernel above is one-dimensional. Thus, $\widehat N_{j,y} \coloneqq |\operatorname{proj}_E N_j|^{-1} \operatorname{proj}_E N_j$ is a unit normal to the planar curve $\gamma_{j,y}$. Orient this curve from $a_{j,y}$ to $b_{j,y}$, and let $T_{j,y}$ denote its Euclidean unit tangent. Smooth convergence at the two outer endpoints gives, for all sufficiently large $j$, $\langle T_{j,y}(a_{j,y}), \omega_{\alpha_{i_0}}\rangle <- 1/2, $ and $\langle T_{j,y}(b_{j,y}),\omega_{\alpha_{i_0}}\rangle > 1/2,$ since the curve enters $D_{\tau,y}^+$ at the first endpoint and leaves it at the second. If $\operatorname{osc}_{\gamma_{j,y}}[N_j]$ were sufficiently small, then $[\widehat N_{j,y}]$ would remain close to $[\nu_{\alpha_{i_0}}]$, and hence the tangent line would remain close to $[\omega_{\alpha_{i_0}}]$. In particular, for some fixed $c>0$, $\big|\langle T_{j,y},\omega_{\alpha_{i_0}}\rangle\big| \geq c$ along $\gamma_{j,y}.$ This contradicts the intermediate value theorem, since the scalar product is negative at $a_{j,y}$ and positive at $b_{j,y}$. Therefore, $\operatorname{osc}_{\gamma_{j,y}}[N_j] \geq\delta_{\mathrm{turn}}$, for some $\delta_{\mathrm{turn}}>0$ independent of $j,y,\tau$.
\end{enumerate}
Taking $ \beta \coloneqq \min \{\delta_{\mathbf C}/2,\delta_{\mathrm{turn}}\}>0,$ we conclude that $\operatorname{osc}_{\gamma_{j,y}}[N_j] \geq\beta$ for all sufficiently large $j$ and almost every regular $y\in B_{1/2}^{n-1}$.

Next, there exists a sequence $\varepsilon_j \rightarrow 0$ such that, along every $\mathrm g_j$-unit speed curve in $\Sigma_j$, there holds $\vert \partial_s N_j \vert \leq C(\vert A_j \vert + \varepsilon_j)$. Thus, we obtain 
\[
\beta \leq C \int_{\gamma_{j, y}} (\vert A_j \vert + \varepsilon_j) \, d\mathcal{H}_{\mathrm{g}_j}^1, \qquad \text{and} \qquad
\beta \leq C \int_{\Sigma_j \cap P^+_y \cap \mathcal{L}_\tau} (\vert A_j \vert + \varepsilon_j) \, d\mathcal{H}_{\mathrm{g}_j}^1, 
\]
for almost every $y \in B_{1/2}^{n  - 1}$. Integrating and using coarea, we obtain
\begin{align*}
    \beta \mathcal L^{n-1} \bigl(B_{1/2}^{n-1}\bigr)
    &\leq C\int_{B_{1/2}^{n-1}} \int_{\Sigma_j\cap P^+_y\cap\mathcal L_\tau} \bigl(|A_j|+\varepsilon_j\bigr)\, d\mathcal H_{\mathrm{g}_j}^{1}\,dy \leq C\int_{\Sigma_j\cap\mathcal L_\tau} |A_j|\,d\mathcal H_{\mathrm g_j}^{n} + C\varepsilon_j \mathcal H_{\mathrm{g}_j}^{n} \bigl(\Sigma_j\cap\mathcal L_\tau\bigr),
\end{align*}
where we have used the uniform bound $ J_{n-1} (\pi|_{\Sigma_j})\leq C$. We conclude that $\beta \mathcal L^{n-1}(B_{1/2}^{n-1}) \leq C\tau^{1/2} + C\varepsilon_j\tau.$ First let $j\to\infty$, with $\tau>0$ fixed. We obtain $\beta \mathcal L^{n-1} (B_{1/2}^{n-1}) \leq C\tau^{1/2}.$ Letting $\tau\downarrow0$, the right-hand side tends to zero, whereas the left-hand side is strictly positive. This contradiction proves that $\alpha_i= \pi/2$ for every $i$. 

We have proven that $ \mathbf{C} = q_0|H^+_0| + p|H^+_{\pi/2}| + q_\pi|H^+_\pi|,$ where $p=\sum_i p_i.$ Appealing to the free-boundary balancing equation, we infer $ q_0-q_\pi=0.$ Consequently, $q_0=q_\pi=q$
and $q|H^+_0|+q|H^+_\pi|=q|T|.$ Thus, $\mathbf{C} = p \vert H_{\pi/2}^+ \vert + q \vert T \vert$. 

A completely analogous argument allows us to treat the mixed case $p, q \geq 0$. More precisely, one can show that $p q = 0$, in which case either $\mathbf{C} = q \vert T \vert$, or $\mathbf{C} = p \vert H_{\pi/2}^+ \vert$. Doubling and appealing to \Cref{thm: boundary sheeting}, or \Cref{thm: orthogonal sheeting}, we conclude. 
\end{proof}

We can finally conclude the proof of \Cref{thm: compactness} by a simple and (by now) rather standard argument using dimension reduction and Simons' classification of stable regular cones in low dimensions.

\subsection{Conclusion}
The key result is the following spine-increase result, which allows the reduction to the aforementioned Simons' classification of stable regular cones in low dimensions. Let $\mathbf{D}$ be a cone and let $S(\mathbf{D})$ be its spine. If $y \in \mathrm{spt} \Vert \mathbf{D} \Vert$ and $\mathbf{E} \in \VarTan(\mathbf D, y)$, then $S(\mathbf D) + \mathrm{span}\{y\} \subset S(\mathbf{E})$. In particular, if $y \not \in S(\mathbf{D})$, then $\dim S(\mathbf{E}) \geq \dim S(\mathbf{D}) + 1$. Using this and appealing to Simons' classification of stable regular cones of codimension one, we conclude that $\dim_{\mathcal{H}} \mathcal{S} \leq n - 7$. In particular, when $n \leq 6$, $\mathcal{S} = \emptyset$, while in the case $n = 7$, we can additionally prove that $\mathcal{S}$ is discrete. 
For the sake of contradiction, assume there are distinct singular points $z_j \in \mathcal{S}$ converging to $z$, and set $\rho_j = d_\mathrm{g}(z_j, z)$. Blow up $V$ around $z$ at scale $\rho_j$ to obtain, after passing to a subsequence, a tangent cone $\mathbf{C}$. Furthermore, the rescaled points $\rho_j^{-1} \exp_z^{-1}(z_j)$ converge to a point $z^* \in S^n$ satisfying $z^* \in \sing(\mathbf{C})$. Since $\mathbf{C}$ is a cone, $\{tz^* \colon \; t > 0 \} \subset \sing(\mathbf{C})$, which is impossible since we have established that for $n = 7$, $\mathcal{H}^\gamma(\mathcal{S}) = 0$ for all $\gamma > 0$. 

\subsection{Stability of the limit} \label{sec: limit stability}
As in \cite[Section 5]{SchoenSimon81}, the previous steps have already established the dimension bound on the singular set $\mathcal S$ and local smooth graphical convergence to $\Sigma$ on $M\setminus \mathcal S$. 

Notice that by stationarity in the interior and the constancy theorem, the density of the limit varifold is constant on each connected regular component. 

Now let $\Gamma$ be a connected component of $\Sigma$. For $x\in \Gamma$, $y\in \partial \Gamma \subset \partial M$ define
\[
a_\Gamma(x) \coloneqq |A_\Gamma(x)|^2+\Ric_M(\nu(x), \nu(x)), \quad b_\Gamma(y)\coloneqq\sff_{\partial M}(\nu(y), \nu(y)),
\]
where $\nu(x)$ is any choice of unit normal to $\Gamma$ at $x$, and similarly for $\nu(y)$. Notice that $a_\Gamma$, $b_\Gamma$ are well-defined on $\Gamma$ and $\partial \Gamma$ respectively, since each term is invariant under $\nu \mapsto -\nu$. 
Now define
\begin{equation}
    Q_\Gamma(\varphi, \varphi)=\int_\Gamma \bigl(|\nabla_\Gamma \varphi|^2-a_\Gamma\varphi^2\bigr)-\int_{\partial\Gamma}b_\Gamma\varphi^2 
\end{equation}
for $\varphi\in C^{0,1}_c(\Gamma)$.

Let $\varphi\in C^{0,1}_c(\Gamma)$ be arbitrary. Since $\Gamma$ is embedded, there is a relatively open neighbourhood $U\Subset M\setminus \mathcal S$ of the support of $\varphi$ which does not intersect any other component of $\Sigma$, and an extension $\hat\varphi\in C^{0,1}_c(U)$ of $\varphi$ to $U$. By (iv) in \Cref{thm: compactness}, if $\varphi_k\coloneqq \hat\varphi|_{\Sigma_k}$, then 
\begin{equation}\label{eq: stability k}
\int_{\Sigma_k} |\nabla_{\Sigma_k} \varphi_k|^2 - \int_{\Sigma_k}\big(|A_{\Sigma_k}|^2 + \Ric_M(\nu_{\Sigma_k},\nu_{\Sigma_k})\big) \varphi_k^2  - \int_{\partial \Sigma_k} \sff_{\partial M}(\nu_{\Sigma_k}, \nu_{\Sigma_k}) \varphi_k^2\geq 0
\end{equation}
for all $k$. As $|\Sigma_k|\to \theta|\Sigma|$ in the sense of varifolds, and $\Sigma_k$ converges locally smoothly to $\Sigma$ on $M\setminus \mathcal{S}$, then, after taking the limit as $k\to \infty$ in \eqref{eq: stability k}, we obtain that $mQ_\Gamma(\varphi, \varphi)\geq 0$, where $m$ is the positive constant $\theta|_\Gamma$. Therefore $Q_\Gamma(\varphi, \varphi)\geq 0$. 

For the one-sided case, we adapt the even--odd argument in the proof
of \cite[Lemma~A.1]{Bellettini22} to the free-boundary setting. If $\Gamma$ is two-sided, the condition $Q_\Gamma(\varphi, \varphi)\geq 0$ for all $\varphi\in C^{0,1}_c(\Gamma)$ says exactly that the stability inequality \eqref{eq: stability intro} is satisfied with $\Gamma$ in place of $\Sigma_k$ and therefore that $\Gamma$ is stable.

Now suppose that $\Gamma$ is one-sided, and consider the orientable double cover $\widetilde{\Gamma}$ of $\Gamma$ in the normal bundle $N\Gamma$ defined by 
\[
\widetilde{\Gamma}=\{(x, v): x\in \Gamma, v\in N_x\Gamma, |v|_{\mathrm g}=1\}.
\]
Equip $\widetilde{\Gamma}$ with the induced pullback metric. The immersion $\tilde{\iota}: \widetilde{\Gamma}\to M$ given by $\tilde{\iota}(x, v)=x$ is free-boundary minimal and has a globally defined unit normal $\tilde\nu_{\widetilde{\Gamma}}(x, v)=v$.
Let
\begin{equation*}
	Q_{\widetilde{\Gamma}}(\phi, \phi)\coloneqq \int_{\widetilde{\Gamma}} |\nabla_{\widetilde{\Gamma}} \phi|^2 - \int_{\widetilde{\Gamma}}\big(|A_{\widetilde{\Gamma}}|^2 + \Ric_{M}(\tilde{\nu}_{\widetilde{\Gamma}},\tilde{\nu}_{\widetilde{\Gamma}})\big) \phi^2  - \int_{\partial \widetilde{\Gamma}} \sff_{\partial M}(\tilde{\nu}_{\widetilde{\Gamma}}, \tilde{\nu}_{\widetilde{\Gamma}}) \phi^2
\end{equation*}
for all $\phi\in C^{0,1}_c(\widetilde{\Gamma})$.

If $\phi\in C^{0,1}_c(\widetilde{\Gamma})$ is \emph{even}, i.e. $\phi(x, -v)=\phi(x, v)$, then it descends to a function $\varphi\in C^{0,1}_c(\Gamma)$ defined by $\varphi(x)=\phi(x,v)$ for either unit normal $v\in N_x\Gamma$, and, by definition,
\[
Q_{\widetilde{\Gamma}}(\phi, \phi) = 2 Q_\Gamma(\varphi, \varphi)\geq 0.
\]
If $\phi\in C^{0,1}_c(\widetilde{\Gamma})$ is \emph{odd}, i.e. $\phi(x, -v)=-\phi(x, v)$, then $|\phi|$ is an even compactly supported Lipschitz function on $\widetilde{\Gamma}$ such that $|\nabla_{\widetilde{\Gamma}}|\phi||^2=|\nabla_{\widetilde{\Gamma}}\phi|^2$ almost everywhere. Then 
\[
Q_{\widetilde{\Gamma}}(\phi, \phi) =Q_{\widetilde{\Gamma}}(|\phi|, |\phi|)\geq 0.
\]
Finally, an arbitrary $\phi\in C^{0,1}_c(\widetilde{\Gamma})$ decomposes into the sum $\phi=\phi_{\mathrm e}+\phi_{\mathrm{o}}$ of an even and an odd function and it is easy to check that by definition of $Q_{\widetilde{\Gamma}}$, we have $Q_{\widetilde{\Gamma}}(\phi_{\mathrm e}, \phi_{\mathrm{o}})=0$.
Therefore, for every $\phi\in C^{0,1}_c(\widetilde{\Gamma})$
\[
Q_{\widetilde{\Gamma}}(\phi, \phi)=Q_{\widetilde{\Gamma}}(\phi_{\mathrm e}, \phi_{\mathrm{e}})+Q_{\widetilde{\Gamma}}(\phi_{\mathrm o}, \phi_{\mathrm{o}})\geq 0.
\]
Therefore,
\begin{equation*}
	\int_{\widetilde{\Gamma}} |\nabla_{\widetilde{\Gamma}} \phi|^2 - \int_{\widetilde{\Gamma}}\big(|A_{\widetilde{\Gamma}}|^2 + \Ric_{M}(\tilde{\nu}_{\widetilde{\Gamma}},\tilde{\nu}_{\widetilde{\Gamma}})\big) \phi^2  - \int_{\partial \widetilde{\Gamma}} \sff_{\partial M}(\tilde{\nu}_{\widetilde{\Gamma}}, \tilde{\nu}_{\widetilde{\Gamma}}) \phi^2\geq 0,
\end{equation*}
for all $\phi\in C^{0,1}_c(\widetilde{\Gamma})$.

\appendix \label{sec: appendix}
\section{Removability of singularities}

We now prove the simple PDE result that is needed in the proofs of the sheeting theorems, namely \Cref{thm: boundary sheeting} and \Cref{thm: orthogonal sheeting}. The result is somewhat standard, and we state it here in the exact form we need in the paper.

Let $D\subset \mathbb{R}^n$ be a bounded open domain and let $E$ be a (possibly empty) relatively open smooth subset of $\partial D$. 
\begin{lemma}\label{lemma: cutoff Hn-1}
	Let $\Gamma$ be a relatively closed subset of $D\cup E$ and assume $\Hcal^{n-1}(\Gamma)=0$. For every compact set $K\subset D\cup E$ there exist $\chi_j\in C^\infty(D\cup E)$ with $0\leq \chi_j\leq 1$, each $\chi_j$ vanishing on a relative neighbourhood of $\Gamma\cap K$ in $D\cup E$ such that $\chi_j\to 1$ locally on $(D\cup E)\cap K\setminus \Gamma$ and $\int_{D\cap K} |\nabla\chi_j|\to 0$.
\end{lemma} 
\begin{proof}
This follows from \cite[Theorem 4.15(v)]{Evans2015}.
\end{proof}

Let $C_0, C_1>0$ be fixed and let $\mathcal L:(D\cup E)\times (-C_0, C_0)\times B^n_{C_1}(0)\to \R$ be a smooth function. Assume there are constants $0<\lambda_1\leq \lambda_2<+\infty$ such that for every $(y, z, q)\in (D\cup E)\times (-C_0, C_0)\times B^n_{C_1}(0)$ and $\xi\in \R^n$
\begin{equation}\label{eq: ellipticity}
	\lambda_1 |\xi|^2\leq \partial_{q_iq_j}\mathcal{L}(y, z, q)\xi^i\xi^j\leq \lambda_2 |\xi|^2.
\end{equation}
\begin{lemma}\label{lemma: removabillity}
Let $\Gamma$ be a relatively closed subset of $D\cup E$ and assume $\Hcal^{n-1}(\Gamma)=0$. Let $u$ be a Lipschitz function on $D\cup E$ such that 
\begin{equation}\label{eq: strict bounds}
	\|u\|_{L^\infty(D\cup E)}\leq c_0<C_0, \quad \operatorname{Lip}(u; D\cup E)\leq c_1<C_1,
\end{equation}
and suppose that the identity
\begin{equation}\label{eq: EL equation}
	\int_D\bigl(D_q\mathcal{L}(y, u, \nabla u)\cdot \nabla\varphi+\partial_z\mathcal{L}(y, u, \nabla u)\varphi\bigr)\, dy=0
\end{equation}	
holds for every test function $\varphi\in C^1(D\cup E)$ such that $\spt \varphi \Subset (D\cup E)\setminus \Gamma$.\\
Then \eqref{eq: EL equation} holds for every $\varphi\in C^1(D\cup E)$ such that $\spt \varphi \Subset (D\cup E)$ and $u\in C^\infty_{\mathrm{loc}}(D\cup E)$. In particular, $u$ is a classical solution of
\begin{equation}\label{eq: EL nonvar}
-\partial_{y_i}(\partial_{q_i}\mathcal L(y, u, \nabla u))+\partial_z\mathcal L(y, u, \nabla u)=0
\end{equation}
and 
\begin{equation}\label{eq: EL BC}
D_q\mathcal{L}(y, u, \nabla u)\cdot \eta_D=0\quad \text{on $E$}
\end{equation}
where $\eta_D$ is the outward unit normal to $D$ along $E$.
\end{lemma}
\begin{proof}
Fix $\varphi\in C^1(D\cup E)$ such that $\spt \varphi \Subset (D\cup E)$ and let $K=\spt \varphi$. Then let $\chi_j$ be the cutoff functions provided by \Cref{lemma: cutoff Hn-1}. Then $\chi_j\varphi$ is an admissible test function for \eqref{eq: EL equation}, so that
\begin{equation}\label{eq: EL variational}
0 = \int_D\chi_j \bigl(D_q\mathcal{L}(y, u, \nabla u)\cdot \nabla\varphi+\partial_z\mathcal{L}(y, u, \nabla u)\varphi\bigr)\, dy+\int_D\varphi D_q\mathcal L(y, u, \nabla u)\cdot \nabla \chi_j\, dy.
\end{equation}
By \eqref{eq: strict bounds} and smoothness of $\mathcal L$, there is some large constant $C<+\infty$ such that \[\|D_q\mathcal{L}(y, u, \nabla u)\|_{L^\infty(D\cap K)}+\|\partial_z\mathcal{L}(y, u, \nabla u)\|_{L^\infty(D\cap K)}\leq C<+\infty.\]
Hence, 
\[
\left|\int_D\varphi D_q\mathcal L(y, u, \nabla u)\cdot \nabla \chi_j\, dy\right|\leq C\|\varphi\|_{L^\infty}\int_{D\cap K}|\nabla \chi_j| \to 0
\]
by \Cref{lemma: cutoff Hn-1}. The first integral in \eqref{eq: EL variational} converges by the dominated convergence theorem. Therefore \eqref{eq: EL equation} holds for $\varphi$, and, hence, by arbitrariness of $\varphi$ for all compactly supported $C^1$ functions on $D\cup E$.

Smoothness of $u$ in the interior then follows from elliptic regularity theory since the divergence form operator associated to \eqref{eq: EL equation} is uniformly elliptic by \eqref{eq: ellipticity}, the coefficients of $\mathcal{L}$ are assumed to be smooth and $u$ is Lipschitz. Along $E$, the natural boundary operator is $\mathcal{B}(y, z, q)\coloneqq D_q\mathcal L(y,z, q)\cdot \eta_D(y)$. Notice that it is uniformly oblique as
\[
D_q\mathcal B(y, z, q)\cdot \eta_D=\partial_{q_iq_j}\mathcal L(y, z, q)\eta_D^i\eta_D^j\geq \lambda_1|\eta_D|^2=\lambda_1>0.
\]
By \cite[Theorem 1]{Lieberman88}, there exists $\alpha\in(0,1)$ such that $u\in C^{1, \alpha}_{\mathrm{loc}}(D\cup E)$. By standard elliptic theory, smoothness follows from $C^{2, \alpha}$ regularity, which is a consequence of the boundary Schauder estimates of \cite{LiebTru86}.
\end{proof}

\begin{remark}
	Notice that \Cref{lemma: removabillity} is considerably weaker and more elementary than Simon’s removable singularity theorem in \cite{SimonRemovable77}. Indeed,  in \Cref{lemma: removabillity} the function $u$ is already defined and Lipschitz across the deleted set $\Gamma$, so only the weak Euler–Lagrange identity must be extended. Simon on the other hand assumes no a priori boundedness or continuity of the solution near $\Gamma$ and proves the existence of a classical extension. However, Simon's theorem as formulated in \cite{SimonRemovable77} does not cover integrands having $z$-dependence or oblique boundary conditions.
\end{remark}

\section{Capacity and extension of stationarity} \label{sec: capacity and extension of stationarity}
We elaborate on \Cref{rmk: area bounds}. In particular, we prove that under the assumptions of \Cref{def: sing min hyp} we can extend stationarity across $\mathcal{S}$. The proof uses some ideas introduced by Harvey--Lawson in \cite{HarveyLawson2014} and employs some tools from classical potential theory. 
The case $n = 2$ of \Cref{prop: locally finite then global stationarity} is due to Simon \cite{SimonIsolated} (without assuming local finiteness a priori). 

The following lemma provides the basic analytic construction behind \Cref{prop: locally finite then global stationarity}, and is based on \cite[Propositions 10.2--10.4]{HarveyLawson2014}. For the sake of completeness, we provide a detailed proof.

\begin{lemma}\label{L: Riesz potential}
	Let $n\geq 2$, and let $E\subset\mathbb R^{n+1}$ be compact with $  \mathcal H^{n-2}(E)<\infty.$ Then there exists an upper semicontinuous function $\psi:\mathbb R^{n+1} \rightarrow[-\infty,\infty)$ such that
	\[
	\{\psi=-\infty\}=E,\qquad
	\psi\in C^\infty(\mathbb R^{n+1}\setminus E),
	\]
	and
	\[
	\mathrm{tr}_{P}D^2\psi(x)\geq0
	\qquad
	\text{for every }x\in\mathbb R^{n+1}\setminus E
	\text{ and }P\in G(n,n+1).
	\]
\end{lemma}
\begin{proof}
	Let $K$ be the fundamental solution of the Laplacian in $\R^n$, so that $K(r)\sim r^{2-n}$ for $n\geq 3$ and $K(r)\sim -\log r$ for $n=2$.
	We first consider $n\geq3$. Define $U_{n - 2}^\sigma(x)\coloneqq\int K(|x-y|)\,d\sigma(y)$ for a finite positive measure $\sigma$. We shall use the following elementary fact (see \cite[Theorem 4.13(a)]{Falconer}).
	
	\begin{fact}\label{fact: falconer Riesz energy}
		Let $F\subset\mathbb R^{n+1}$ be Borel, and let $\sigma$ be a
		nonzero finite positive measure supported on $F$. If 
		\[
		I_{n - 2}(\sigma) \coloneqq \iint K(|x-y|)\,d\sigma(x)d\sigma(y)<\infty,
		\]
		then $\mathcal H^{n - 2}(F)=\infty.$ 
	\end{fact}
	
	Let now $\lambda\coloneqq\mathcal H^{n - 2} \mres E,$ which is a finite positive measure by hypothesis. Set $ Z\coloneqq\{x\in E:U_{n - 2}^\lambda(x)<\infty\}, $ and $Z_m\coloneqq\{x\in E:U_{n - 2}^\lambda(x)\leq m\},$ so that $Z_m \subset Z$ and $Z = \bigcup_m Z_m$. We claim that $ \mathcal H^{n - 2}(Z)=0.$\\ Indeed, if $\lambda(Z_m)>0$ for some $m$, then the nonzero measure $ \sigma\coloneqq\lambda\mres Z_m$ satisfies
	\[
	I_{n - 2}(\sigma) =\int_{Z_m}U_{n - 2}^\sigma\,d\lambda \leq\int_{Z_m}U_{n - 2}^\lambda\,d\lambda \leq m\lambda(Z_m)<\infty,
	\]
	and \Cref{fact: falconer Riesz energy} would then imply $\mathcal H^{n - 2}(Z_m)=\infty$, contradicting $Z_m\subset E$ and $\mathcal H^{n - 2}(E)<\infty$. Hence $\lambda(Z_m)=0$ for every $m$, and therefore $\mathcal H^{n - 2}(Z)=\lambda(Z)=0.$
	
	Thus $U_{n - 2}^\lambda(\cdot)=+\infty$ at every point of $E\setminus Z$. To conclude the argument, we shall now construct another finite positive measure $\nu$ such that $U_{n - 2}^\nu(x)=+\infty$ for all $x\in Z$ as well. In the end, we shall define $\mu\coloneqq \lambda+\nu$ and take $\psi = -U^\mu_{n-2}$, which, as we will show, has the desired properties.
	
	Since $\mathcal H^{n - 2}(Z)=0$, for each $k\geq1$ we may choose a countable cover
	\[
	Z\subset\bigcup_i B_{r_{k,i}}(a_{k,i}),
	\qquad a_{k,i}\in Z,
	\]
	such that $ \sum_i r_{k,i}^{n - 2}<2^{-k}.$ Define $\nu_k\coloneqq\sum_i r_{k,i}^{n - 2}\,\delta_{a_{k,i}}$, and $\nu\coloneqq\sum_{k=1}^\infty\nu_k$.
	Then $\nu$ is a finite positive measure supported on $E$, since $\nu(\mathbb R^{n+1}) \leq\sum_{k=1}^\infty2^{-k}<\infty$.
	For every $x\in Z$ and every $k$, there is an $i$ such that $|x-a_{k,i}|<r_{k,i}$, from which we deduce $ U_{n - 2}^{\nu_k}(x) \geq |x-a_{k,i}|^{2 -n} r_{k,i}^{n - 2} \geq 1.$ It follows that
	\[
	U_{n - 2}^\nu(x) = \sum_{k=1}^\infty U_{n - 2}^{\nu_k}(x) =+\infty \qquad\text{for every }x\in Z.
	\]
	
	Consequently, the finite positive measure $\mu\coloneqq\lambda+\nu$ is supported on $E$ and satisfies $U_{n - 2}^\mu(x)=+\infty$ for every $x \in E$. Define now $\psi(x) \coloneqq -U_{n - 2}^\mu(x)$ so that $\psi=-\infty$ on $E$. If $x \notin E$, then $\operatorname{dist}(x,E)>0,$ and hence $U_{n - 2}^\mu(x)<\infty$. Thus $\{\psi=-\infty\}=E.$ Moreover, on every compact subset of $\mathbb R^{n+1}\setminus E$, the kernel and all of its derivatives are uniformly bounded for $y\in E$; differentiation under the integral therefore gives $\psi\in C^\infty(\mathbb R^{n+1}\setminus E).$
	
	It remains to verify the Hessian inequality. For $z\neq0$, writing
	$e=z/|z|$, one computes
	\[
	D^2\bigl(-|z|^{2-n}\bigr) = (n-2)|z|^{-n}\bigl(I-n\,e\otimes e\bigr).
	\]
	Hence, for every $P\in G(n,n+1)$, we have $\mathrm{tr}_P D^2\bigl(-|z|^{2-n}\bigr) = (n-2)|z|^{-n} \bigl(n-n|\pi_Pe|^2\bigr) = n(n-2)|z|^{-n}\bigl(1-|\pi_Pe|^2\bigr)\geq0$.
	It follows that, for $x\notin E$,
	\[
	\begin{split}
		\mathrm{tr}_P D^2\psi(x)
		&=
		\frac{n}{\omega_n}
		\int_E
		\frac{1-|\pi_Pe_{x-y}|^2}{|x-y|^n}\,d\mu(y)
		\geq0,
	\end{split}
	\]
	where $e_{x-y}\coloneqq(x-y)/|x-y|$.
	
	
	If $n=2$, then $\mathcal H^0(E)<\infty$, so that $E=\{a_1,\ldots,a_N\}$ is finite. In this case take
	\[
	\mu\coloneqq\sum_{j=1}^N\delta_{a_j}, \quad \text{and} \quad \psi(x)\coloneqq\int K(|x-y|)\,d\mu(y) =\frac{1}{2\pi}\sum_{j=1}^N\log|x-a_j|.
	\]
	Clearly $\{\psi=-\infty\}=E$ and $\psi\in C^\infty(\mathbb R^3\setminus E)$. Since $D^2\log|z| = |z|^{-2}\bigl(I-2e\otimes e\bigr),$ for every $P\in G(2,3)$, $ \mathrm{tr}_P D^2\log|z| = 2|z|^{-2}\bigl(1-|\pi_Pe|^2\bigr)\geq0.$ 
\end{proof}

\begin{remark}
	The function $\psi$, extended by $-\infty$ on $E$, is \emph{$n$-plurisubharmonic} in the sense of \cite[\S 9]{HarveyLawson2014}. Consequently the set $E$ is \emph{$C^\infty$ $n$-pluripolar} (again, see \cite[\S 9]{HarveyLawson2014}).
\end{remark}

Let us now adapt the basic construction of \Cref{L: Riesz potential} to the setting of Riemannian manifolds (with and without boundary). For $r>0$, let 
\[
B_r\coloneqq B^{n+1}_r(0), \quad C_r\coloneqq B^n_r(0)\times[0,r), \quad \partial_0 C_r\coloneqq B^n_r(0)\times\{0\}\subset C_r.
\]

\begin{lemma}\label{lemma: Riesz Riemannian}
	Let $n\geq 2$.
	\begin{enumerate}[(a)]
		\item Let $\mathrm{g}$ be a smooth Riemannian metric on a neighbourhood of $\overline{B}_2$. There exists $r_0\in(0,1)$ such that for every compact set $E\subset \overline{B}_{r_0}$ with $\mathcal{H}^{n-2}(E)<\infty$ there is a function $v\in C^\infty(B_{2r_0}\setminus E)$, $v\geq 1$ such that $v(x)\to +\infty$ as $x\to E$ and $\mathrm{tr}_P\nabla^2_\mathrm{g}v(x)\leq 0$ for all $x\in B_{2r_0}\setminus E$ and every $P\in G(n, n+1)$.
		
		\item Let $\mathrm{g}$ be a smooth Riemannian metric on a neighbourhood of $\overline{C}_2$ of the form $\mathrm{g}=dt^2+h_{ij}(y, t)dy^idy^j$, where $(h_{ij})$ is smooth and positive-definite for all $t\in[0, 2]$. There exists $r_0\in(0,1)$ such that for every compact set $E\subset \overline{B}^n_{r_0}\times\{0\}$ with $\mathcal{H}^{n-2}(E)<\infty$ there is a function $v\in C^\infty(C_{2r_0}\setminus E)$, $v\geq 1$ such that $v(x)\to +\infty$ as $x\to E$ and $\mathrm{tr}_P\nabla^2_\mathrm{g}v(x)\leq 0$ for all $x\in C_{2r_0}\setminus E$ and every $P\in G(n, n+1)$, and $\partial_tv=0$ on $\partial_0C_{2r_0}\setminus E$.
	\end{enumerate}
    In both cases, the extension of $v$ by $+\infty$ on $E$ is lower semicontinuous on $B_{2r_0}$ and $C_{2r_0}$, respectively.
\end{lemma}

\begin{proof}
	In order to deal with the errors arising from the Riemannian metric $\mathrm{g}$, it is convenient to introduce a strictly superharmonic (with respect to the Euclidean metric) perturbation of the Riesz kernel in \Cref{L: Riesz potential}.
	For $\lambda>0$, choose $r_*>0$ such that $\lambda r_*<n$, and define
	\begin{equation}\label{eq: perturbation Riesz kernel}
		K_{\lambda}(r)\coloneqq \int_r^{r_*}s^{1-n}e^{\lambda s}\, ds\quad \text{for } 0<r<r_*, 
	\end{equation}
	and notice that, as $r\downarrow 0$, $K_{\lambda}$ has the same leading order behaviour as the Riesz kernel in \Cref{L: Riesz potential}.
	
	Now, for (a), by shrinking $r_0$, let us assume that $B_{3r_0}$ is strongly $\mathrm{g}$-convex and for $a\in B_{r_0}$ the distance $r_a(\cdot)=\dist_\mathrm{g}(\cdot, a)$ is smooth on $B_{2r_0}\setminus \{a\}$ and satisfies $|\nabla_\mathrm{g}^2r_a-r_a^{-1}(\mathrm{g}-dr_a\otimes dr_a)|\leq Cr_a$ (which is a consequence of the Hessian comparison theorem in \cite[Theorem 11.7]{LeeRiem}).
	Then, if $P\in G(n, n+1)$ and $\alpha=|\pi^\mathrm{g}_P\nabla^\mathrm{g}r_a|_\mathrm{g}^2\in[0,1]$, we have
	\begin{equation}\label{eq: interior trace}
		\begin{aligned}
			\mathrm{tr}_P\nabla^2_\mathrm{g}K_{\lambda}(r_a) \leq K''_{\lambda}(r_a)+\frac{n-1}{r_a}K'_{\lambda}(r_a)+Cr_a|K'_{\lambda}(r_a)| \leq (-\lambda+Cr_a)|K'_{\lambda}(r_a)|.
		\end{aligned}
	\end{equation}
	Take $\lambda=1$ and reduce $r_0$ so that $Cr_a<1$; then $\mathrm{tr}_P\nabla^2_\mathrm{g}K_{1}(r_a)\leq 0$ for all $P\in G(n, n+1)$.
	
	If we let $\mu$ be the finite positive measure given in the proof of \Cref{L: Riesz potential}, and 
	\[
	v(x)\coloneqq 1+\int_E K_{1}(\dist_\mathrm{g}(x,y))\, d\mu(y),
	\]
	then it is easy to check that, after shrinking $r_0$ if needed, $v$ has all the required properties, since $\dist_{\mathrm{g}}$ is uniformly equivalent to the Euclidean distance on $B_{3r_0}$ with constants going to 1 as $r_0\to 0$.
    In particular, the integral diverges on $E$. Its extension by $+\infty$ on $E$ is lower semicontinuous by Fatou's lemma, which gives $v(x)\to+\infty$ as $x\to E$.
    
	For (b), write $h_t\coloneqq h_{ij}(y, t)dy^idy^j$ on $C_2$ and let $\hat{\mathrm{g}}\coloneqq dt^2+h_0(y)$ on the doubled cylinder $\hat{C}_2\coloneqq B^n_2(0)\times (-2,2)$. For $a=(a', 0)\in E$, let $\hat{r}_a(y, t)\coloneqq \dist_{\hat{\mathrm{g}}}((y, t), (a',0))=(\dist_{h_0}(y, a')^2+t^2)^{1/2}$. Then notice that $\partial_t\hat r_a(y, 0)=0$ for $y\neq a'$.
	
	The interior computation in \eqref{eq: interior trace} implies that for every $P\in G(n, n+1)$, we have $\mathrm{tr}_P\nabla^2_{\hat{\mathrm{g}}}K_\lambda(\hat r_a)\leq (-\lambda-C\hat r_a)|K'_{\lambda}(\hat r_a)|$. By smoothness of $h_t$ in $t$ we have $|h_t-h_0|=O(|t|)$, $|\mathrm{g}-\hat{\mathrm{g}}|_{\hat{\mathrm{g}}}\leq C|t|$, $|\nabla^\mathrm{g}-\nabla^{\hat{\mathrm{g}}}|_{\hat{\mathrm{g}}}\leq C$.
	Then,
	\[
	\mathrm{tr}_P\nabla^2_{\mathrm{g}}K_\lambda(\hat r_a)\leq \mathrm{tr}_P\nabla^2_{\hat{\mathrm{g}}}K_\lambda(\hat r_a)+C\left(|t||\nabla^2_{\hat{\mathrm{g}}}K_\lambda(\hat r_a)|_{\hat{\mathrm g}}+ |K'_\lambda(\hat r_a)|\right).
	\]
	If $\lambda \hat r_a\leq n/2$, then $|\nabla^2_{\hat{\mathrm{g}}}K_\lambda(\hat r_a)|_{\hat{\mathrm g}}\leq C\hat r_a^{-1}|K'_\lambda(\hat r_a)|$, so that using $|t|\leq \hat r_a$
	\begin{equation}\label{eq: boundary trace}
		\begin{aligned}
			\mathrm{tr}_P\nabla^2_\mathrm{g}K_{\lambda}(\hat r_a) \leq  (-\lambda+C_0+C_1\hat r_a)|K'_{\lambda}(\hat r_a)|.
		\end{aligned}
	\end{equation}
	Choose first $ \lambda > 2C_0+1$ and then $r_0<r_*$ so small that $\lambda r_*<n/2$, and $C_1\hat r_a\leq 1$. Then $\mathrm{tr}_P\nabla^2_\mathrm{g}K_{\lambda}(\hat r_a)\leq 0$ for all $P\in G(n,n+1)$.
	
	Finally, if we define
	\[
	v(x)\coloneqq 1+\int_E K_{1}(\dist_{\hat{\mathrm{g}}}(x,y))\, d\mu(y),
	\]
	for the finite positive measure $\mu$ in the proof of \Cref{L: Riesz potential}, then it is easy to check that it satisfies all the required properties. Divergence of the integral at $E$ follows as in (a), and $\partial_tv=0$ on $\partial_0C_{2r_0}\setminus E$ follows from $\partial_t\hat r_a(y,0)=0$.
\end{proof}

We can now state and prove the main result of this appendix.

\begin{proposition} \label{prop: locally finite then global stationarity}
	Let $n \geq 2$, and let $(M^{n+1}, \mathrm{g})$ be a smooth Riemannian manifold-with-boundary with (possibly empty) boundary $\partial M$. Let $\mathcal{S} \subset M$ be a relatively closed subset with locally finite $\mathcal{H}^{n-2}$ measure, i.e. $\mathcal{H}^{n - 2}(\mathcal S \cap K) < + \infty$ for every $K \Subset M$. 
	Let $V$ be an $n$-dimensional rectifiable varifold in $M$ such that 
	\[
	\delta V(X) = 0 \qquad \text{for every $X \in \mathfrak{X}^{\mathrm{tan}}(M \setminus \mathcal{S})$}. 
	\]
	Then $V$ is free-boundary stationary in $M$, that is 
	\[
	\delta V(X) = 0 \qquad \text{for every $X \in \mathfrak{X}^{\mathrm{tan}}(M)$}. 
	\]
	Consequently, if $\Sigma$ is a smooth proper minimal free-boundary immersion in $M \setminus \mathcal{S}$ and $|\Sigma|$ has locally finite mass, then $|\Sigma|$ is free-boundary stationary in $M$. 
\end{proposition}
\begin{proof}
	Let $U\Subset M$ be a subset of the domain of a coordinate chart such that $U$ is either a metric ball or a metric half-cylinder in Fermi coordinates with respect to $\partial M$ as in \Cref{lemma: Riesz Riemannian}, and let $X\in \mathfrak X^\mathrm{tan}(U)$. The global statement then follows from this local one after covering $\spt X$ by finitely many sufficiently small metric balls or half-cylinders. 
	
	Choose $\spt X \Subset U_0 \Subset U_1\Subset U$ and set $E = \mathcal{S} \cap \overline{U}_1$, so that $E$ is compact and satisfies $\mathcal{H}^{n - 2}(E) < \infty$ by hypothesis. By \Cref{lemma: Riesz Riemannian}, there is a function $v\in C^\infty(U\setminus E)$ such that $v\geq 1$ on $\overline{U}_1\setminus E$, $v(x)\to +\infty$ as $x\to E$ and $\mathrm{tr}_P\nabla^2_\mathrm{g}v\leq 0$ for all $P\in G(n, n+1)$. In particular, $v$ is superharmonic.
	
	We now prove a logarithmic Caccioppoli estimate. Let $\eta \in C_c^\infty(U_1)$ satisfy $\eta \equiv 1$ on $U_0$, and let 
	\[
	q_\ell(v) \coloneqq \left( \frac{1}{v} - \frac{1}{\ell} \right)_+, \qquad \text{for $\ell > 1$}. 
	\]
	Extend $v$ by $+\infty$ on $E$.  Lower semicontinuity and the fact that $\spt\eta$ is compact imply that $\spt\eta\cap\{v\leq\ell\}$ is a compact subset of $U_1\setminus\mathcal S$. Thus $Y_\ell\coloneqq\eta^2q_\ell(v)\nabla^{\mathrm g}v$ is supported away from $\mathcal S$ and is tangent to $\partial M$ in the boundary case.
	Free-boundary stationarity away from $\mathcal S$ gives 
	\[
	0 = \delta V(Y_\ell) = \int_{G_n(U_1)}[\eta^2 q_\ell(v) \mathrm{tr}_P \nabla^2_\mathrm{g} v + \langle \nabla_P(\eta^2 q_\ell(v)), \nabla_P v \rangle] \, dV(x, P). 
	\]
	The first term is non-positive, thus implying 
	\[
	\int_{G_n(U_1)} \langle \nabla_P (\eta^2 q_\ell(v)), \nabla_P v\rangle  \, dV(x, P) \geq 0. 
	\]
	Now, on $\{v < \ell\}$, there holds $\nabla_P q_\ell(v) = -v^{-2} \nabla_P v$. Therefore, 
	\begin{align*}
		A_\ell \coloneqq \int_{\{v < \ell\}} \eta^2 \vert \nabla_P \log(v) \vert^2 \, dV & \leq 2 \int_{\{v < \ell\}} \eta \left( 1 - \frac{v}{\ell} \right) \langle \nabla_P \eta, \nabla_P \log(v) \rangle \, dV \\ 
		& \leq \frac{1}{2} A_\ell + 2 \int \vert \nabla_P \eta \vert^2 \, dV. 
	\end{align*}
	Rearranging, 
	\[
	A_\ell \leq 4 \int \vert \nabla_P \eta \vert^2 \, dV \leq 4 \Vert D\eta \Vert_\infty^2 \|V\|(\mathrm{spt} \eta) < + \infty.
	\]
	The last quantity is finite because $V$ has locally finite mass by assumption. Letting $\ell \rightarrow \infty$ gives 
	\[
	\int_{U_0} \vert \nabla_P \log(v) \vert^2 \, dV < + \infty. 
	\]
	We now construct the cut-off. Choose $\vartheta \in C^\infty(\mathbb{R})$ such that $0 \leq \vartheta \leq 1$ , $\vartheta \equiv 1$ on $(-\infty, 1]$ and $\vartheta \equiv 0$ on $[2, \infty)$. Define then $\chi_j \coloneqq \vartheta\left( j^{-1} \log(v)\right).$ For every fixed $j$, the function $\chi_j$ vanishes in an ambient neighbourhood of $E$, so that $\chi_j X$ is an admissible test vector field supported away from $\mathcal{S}$. Moreover, $\chi_j \rightarrow 1$ pointwise on $U_0 \setminus \mathcal{S}$. Lastly, 
	\[
	\int_{U_0} \vert \nabla_P \chi_j \vert^2 \, dV \leq \frac{\Vert \vartheta^\prime \Vert_\infty^2}{j^2} \int_{U_0} \vert \nabla_P \log(v) \vert^2 \, dV \rightarrow 0,  
	\]
	as $j \rightarrow \infty$. Consequently, from the stationarity condition
	\[
	0 = \delta V(\chi_j X) = \int \chi_j \mathrm{div}_P X \, dV + \int \langle \nabla_P \chi_j, X \rangle \, dV,
	\]
	we can use dominated convergence to pass the first term on the right-hand side to the limit, while for the second term we have 
	\[
	\left\vert \int \langle \nabla_P \chi_j, X \rangle \, dV \right\vert \leq \Vert X \Vert_\infty \|V\|(\mathrm{spt} X)^{1/2} \left(\int_{\mathrm{spt} X} \vert \nabla_P \chi_j \vert^2 \, dV \right)^{1/2} \rightarrow 0, 
	\]
	as $j \rightarrow \infty$. Thus, $V$ satisfies $\delta V(X) = 0$ in $U$ meaning that it is free-boundary stationary across $\mathcal{S}$. 
\end{proof}
We record the following consequence of \Cref{prop: locally finite then global stationarity} which will be the form we shall use for excision arguments (see \ref{not: immersion} for the notational convention in the integrals). 
\begin{corollary} \label{cor: cap}
	Let $\mathcal{S}$ be as in \Cref{def: sing min hyp}. Then for every compact $K\subset M$ there exists a sequence of smooth cutoff functions $\{\chi_j\}_j$ on $M$, each supported away from $\mathcal{S}$, such that $0\leq \chi_j\leq 1$, $\chi_j\to 1$ on $\Sigma\cap K$, and $\int_{\Sigma\cap K}|\nabla_\Sigma\chi_j|^2\to 0.$ This will sometimes be denoted by $\operatorname{cap}_2(\mathcal{S}) = 0$ and referred to as having vanishing relative $2$-capacity in $\Sigma$ (see below). 
\end{corollary}

\begin{remark}
    A weaker hypothesis suffices for extending stationarity when a local area-growth bound is already available. Namely, assume $\mathcal H^{n-1}(\mathcal S)=0$ and $\mathcal{H}^n(\Sigma\cap B_r(c))\leq Cr^n$ for $x\in\mathcal S$ and all sufficiently small radii for some constant $C$ uniform on compact sets. Then stationarity away from $\mathcal S$ extends across $\mathcal S$.
	
	Indeed, fix $X\in \mathfrak{X}^{\mathrm{tan}}(U)$ and choose a compact $K\Subset U$ whose relative interior contains $\operatorname{spt}X$. Since $\mathcal H^{n-1}(\mathcal{S}\cap K)=0$, for every $\varepsilon>0$ we may cover $\mathcal{S} \cap K$ by finitely many balls $B_{r_i}(x_i)$, with $r_i<\varepsilon$ and $\sum_i r_i^{n-1}<\varepsilon$. Choose $0\leq\eta\leq1$, equal to $0$ near $\bigcup_iB_{r_i}(x_i)$ and to $1$ outside $\bigcup_iB_{2r_i}(x_i)$, with $|D\eta|\leq C\sum_i r_i^{-1}\mathbf 1_{B_{2r_i}(x_i)}$. The Euclidean growth bound $\|V\|(B_r(x))\leq Cr^n$ at $x\in \mathcal S$ gives
	\[
	\int |\nabla_P\eta|\,dV\leq C\sum_i r_i^{n-1}\leq C\varepsilon,
	\qquad
	\|V\|(\{\eta\neq1\})\leq C\sum_i r_i^n\leq C\varepsilon^2.
	\]
	Since $\eta X$ is supported away from $\mathcal{S}$, $0=\delta V(\eta X)$; hence
	\[
	|\delta V(X)| \leq \|DX\|_\infty\|V\|(\{\eta\neq1\}) +\|X\|_\infty\int|\nabla_P\eta|\,dV \leq C\bigl(\|DX\|_\infty\varepsilon^2+\|X\|_\infty\varepsilon\bigr).
	\]
	Letting $\varepsilon\downarrow0$ proves that $\delta V(X)=0$.
\end{remark}

\section{Tangent cones and stratification for free-boundary stationary varifolds}\label{sec: tangent cones stratification}
\subsection{Free-boundary stationary varifolds}\label{subsec: fb varifolds}
An integral\footnote{It is actually not necessary to assume integrality, but, for the purposes of this paper, we shall not need to consider any larger classes of varifolds.} varifold $V$ on $M$ is said to be \emph{free-boundary stationary} if its first variation vanishes for all compactly supported vector fields on $M$ which are tangential to $\partial M$ along $\partial M\subset M$, i.e.
\begin{equation}\label{eq: fb stationary}
	\delta V(X)=0\quad \text{for all } X\in \mathfrak{X}^{\mathrm{tan}}(M).
\end{equation}
In particular, nothing is assumed on the first variation of $V$ in the direction of vector fields that are not tangential to $\partial M$. Notice that by \Cref{def: sing min hyp}, it is easy to see that if $\Sigma$ is a singular properly embedded (or properly immersed) free-boundary minimal hypersurface, then the multiplicity-one varifold $|\Sigma|$ induced by it is a free-boundary stationary varifold. 

\begin{remark}
Note that if $\Gamma$ is a boundary component of $M$, then the varifold $|\Gamma|$ is free-boundary stationary, independently of whether the mean curvature of $\Gamma$ in $M$ vanishes, even though $\Gamma$ is a smooth hypersurface. On the other hand, notice that $\Gamma$ would not be a free-boundary minimal hypersurface according to our \Cref{def: fb smooth immersed min hyp}. Therefore, free-boundary stationary varifolds are strictly more general than free-boundary minimal hypersurfaces, even if we restrict to smooth objects.
\end{remark}

Although it is not apparent from the definition, free-boundary stationary integral varifolds have locally bounded first variation by \cite[Proposition 3.2]{Edelen2020}\footnote{This is true even without assuming integrality by \cite{Masi2020}.}. From this, Edelen manages to deduce a compactness theorem for sequences of free-boundary stationary integral varifolds with locally uniformly bounded mass \cite[Theorem 3.5]{Edelen2020} from Allard's Compactness theorem \cite[6.4]{Allard1972}.
In our simplified setting of singular free-boundary minimal hypersurfaces, the full first variation can be computed rather directly as in \Cref{lemma: first variation trace}, so compactness can be reduced immediately to Allard's result.

If we consider free-boundary stationary varifolds in a Euclidean half-space, then we have the following rather useful reflection principle.

\begin{lemma}[Reflection principle, {\cite[Lemma 2.2]{Li2021}}]\label{L:reflection stationary}
Let $v \in \mathbb{R}^{n+1}$ be a unit vector and let  $P^+$ be the closed half-space $\{u: u \cdot v \geq 0\}$. Let $\theta: \mathbb{R}^{n+1} \to \mathbb{R}^{n+1}$ be the reflection map across $T=v^\perp$ given by $\theta(u) = u - 2 (u \cdot v) v$. If $V$ is a free-boundary stationary integral varifold in $P^+$, then the doubled varifold 
\begin{equation*}
	\overline{V} \coloneqq  V + \theta_\sharp V 
\end{equation*}
is a stationary integral varifold in $\R^{n+1}$ and it is invariant under $\theta$, i.e. $\theta_\sharp \overline{V} = \overline{V}$. 
\end{lemma}

Now let $V$ be a free-boundary stationary integral varifold and let $x\in M$. 
\begin{itemize}
\item If $x\in M\setminus \partial M$, then clearly for any open set $U\Subset M\setminus \partial M$ containing $x$, $V$ is stationary in $U$ and therefore the usual \emph{interior monotonicity formula} (see \cite[\S17]{Simon1983}) holds for balls $B_r(x)\subset U$. 
\item If $x\in \partial M$, then \cite[Theorem 3.4]{Guang2020} provides a \emph{boundary monotonicity formula} for balls centred at $x$.
\end{itemize}
The combination of these two monotonicity formulas implies existence of density and of tangent varifolds at every point. It also allows us to show that any tangent varifold to $V$ at a point $x\in\spt\|V\|$ is a cone. (See \Cref{subsec: tangent cones} for the exact definition in the manifold setting and some properties.)

On the other hand, some standard properties of stationary integral varifolds do not hold directly for free-boundary stationary integral varifolds. As an example, upper semicontinuity of density does not hold, even in the smooth setting (think of a sequence of \emph{interior} points converging to a \emph{boundary} point). In order to have a quantity that does not distinguish interior balls from balls centred at the boundary and satisfies a uniform monotonicity formula, it is possible to introduce the so-called \emph{reflected density} (see \cite[Section 3]{Grueter1986} for the Euclidean case) which turns out to be upper-semicontinuous. As the adaptation to the manifold setting is somewhat involved, we shall avoid introducing the reflected density, and work with the two distinct monotonicity formulas above, which will suffice for our purposes.

\subsection{Tangent cones}\label{subsec: tangent cones}
Let $V$ be a free-boundary stationary integral varifold in $M$, let $x\in \spt \|V\|$, and $r_i \to 0$ be a sequence of positive numbers.
We shall consider two cases.
\begin{enumerate}
	\item If $x\in M\setminus\partial M$, let $0<\rho_x<\min\{\inj_x(M), \dist_\mathrm{g}(x, \partial M)\}$ so that the exponential map of $M$ at $x$ is a diffeomorphism of $B^{T_xM}_{\rho_x}(0)\subset T_xM$ onto its image $B^M_{\rho_x}(x)$, and let $\Psi_{x, r_i}:B^{T_xM}_{r_i^{-1}\rho_x}(0) \to M$ be given by
	\begin{equation*}
		\Psi_{x, r_i}(v)\coloneqq \exp_x(r_i v).
	\end{equation*}
	Let $\mathrm{g}_{x, r_i}\coloneqq r_i^{-2}\Psi_{x, r_i}^*\mathrm{g}$. Then, notice that
	\begin{equation*}
		(B^{T_xM}_{r_i^{-1}\rho_x}(0), \mathrm{g}_{x, r_i})\to (T_xM, \mathrm{g}_{\mathrm{Euc}}) \quad \text{ locally smoothly}.
	\end{equation*}
	\item If $x\in \partial M$, let $0<\rho_x<\min\{\inj_x(\partial M), \inj_x(M)\}$ be such that the exponential map of $\partial M$ at $x$ is a diffeomorphism of $B^{T_x(\partial M)}_{\rho_x}(0)\subset T_x(\partial M)$ onto its image $B^{\partial M}_{\rho_x}(x)\subset \partial M$ and the normal exponential map $F_{\partial M}$ of $\partial M$ is a diffeomorphism of $B^{\partial M}_{\rho_x}(x)\times[0, \rho_x)$ onto its image. Then let $\Psi_{x, r_i}:B^{T^+_xM}_{r_i^{-1}\rho_x}(0) = T_x^+M \cap  B^{T_xM}_{r_i^{-1}\rho_x}(0)\to M$ be given by
	\begin{equation}\label{eq: rescaling boundary}
		\Psi_{x, r_i}(v+t\bn(x))\coloneqq F_{\partial M}(\exp^{\partial M}_x(r_iv), r_i t),
	\end{equation}
	for $v\in B^{T_x(\partial M)}_{r_i^{-1}\rho_x}(0)$, $t\geq 0$, $v+t\bn(x)\in B^{T^+_xM}_{r_i^{-1}\rho_x}(0)$. Again, let $\mathrm{g}_{x, r_i}\coloneqq r_i^{-2}\Psi_{x, r_i}^*\mathrm{g}$, and notice that
	\begin{equation*}
		(B^{T^+_xM}_{r_i^{-1}\rho_x}(0), \partial_0 B^{T^+_xM}_{r_i^{-1}\rho_x}(0), \mathrm{g}_{x, r_i})\to (T^+_xM, T_x(\partial M), \mathrm{g}_{\mathrm{Euc}}) \quad \text{ locally smoothly}.
	\end{equation*}
\end{enumerate}
In either case, let $\Phi_{x, r_i}$ denote the inverse map $\Psi_{x, r_i}^{-1}$.

If we apply Edelen's Compactness theorem \cite[Theorem 3.5]{Edelen2020} to the sequence $V_i\coloneqq (\Phi_{x, r_i})_\sharp V$\footnote{It is easy to see that \cite[Theorem 3.5]{Edelen2020} applies to a sequence of metrics converging in $C^2_\mathrm{loc}$.}, then, up to a subsequence (not relabelled), there exists an integral varifold $V_\infty$ such that $V_i\to V_\infty$ locally in the sense of varifolds.
By convergence of the first variation, we can pass the variational property of $V$ to the limit. More precisely, we have the following.
\begin{enumerate}
	\item If $x\in M\setminus \partial M$, then $V_\infty$ is a stationary integral varifold in $(T_xM, \mathrm{g}_{\mathrm{Euc}})\cong (\R^{n+1}, \mathrm{g}_{\mathrm{Euc}})$. The usual monotonicity formula for stationary varifolds applies and one deduces that $V_\infty$ is a stationary integral cone of density $\Theta(\|V\|, x)$. 
	\item If $x\in \partial M$, then $V_\infty$ is a free-boundary stationary integral varifold in $(T^+_xM, \mathrm{g}_{\mathrm{Euc}})\cong (\R_+^{n+1}, \mathrm{g}_{\mathrm{Euc}})$. The monotonicity formula for free-boundary stationary varifolds of \cite{Guang2020} applies, and one deduces that $V_\infty$ is a free-boundary stationary integral cone of density $\Theta(\|V\|, x)$. Moreover, if we denote by $\theta_{0}$ the map on $(T_xM, \mathrm{g}_{\mathrm{Euc}})$ given by reflection in $T_x(\partial M)\subset T_xM$, then, by the reflection principle in \cite[Lemma 2.2]{Li2021}, the \emph{doubled} cone
	\begin{equation}
		\overline{V}_\infty\coloneqq V_\infty + (\theta_0)_\sharp V_\infty
	\end{equation}
	is a stationary integral cone in $(T_xM, \mathrm{g}_{\mathrm{Euc}})$, which is also invariant under $\theta_0$.
\end{enumerate}
Any varifold $V_\infty$ that arises in this way is said to be a \emph{tangent cone} to $V$ at $x$. The set of all tangent cones to $V$ at $x$ is denoted by $\VarTan(V, x)$.

\subsection{Singular set and stratification for free-boundary stationary integral varifolds}\label{subsec: stratification}
In order to analyse the regularity of a free-boundary stationary integral varifold $V$, we need to introduce some further definitions.

\begin{definition}[Regular and singular set]
Let $V$ be a free-boundary stationary integral varifold in $M$. We define the \emph{regular set} $\reg(V)$ of $V$ to be
\begin{equation*}
	\reg(V)\coloneqq\bigg\{x \in \spt \|V\|:
	\begin{array}{l}
		\textrm{there exists $r>0$ such that $\spt\|V\|\cap B_{r}(x)$ is a smooth }
		\\
		\textrm{embedded hypersurface-with-(possibly empty) boundary in $\partial M$}
	\end{array}\bigg\}.
\end{equation*}

Similarly, we define the \emph{singular set} $\sing(V)$ of $V$ by
\begin{equation}
    \sing(V)=\spt \|V\|\setminus\reg(V).
\end{equation}
\end{definition}
We shall now subdivide the singular set into pieces according to the possible tangent cones that can occur.
In order to state the results of this stratification procedure, we need some more definitions and notation about cones. Let $\eta_{v, \lambda}$ be the map $x\mapsto \lambda^{-1}(x-v)$.

\begin{definition}[Spine of a cone]\label{D:spine}
Let $\bC$ be an $n$-dimensional cone in $\R^{n+1}$. We define the \emph{spine} $S(\bC)$ of $\bC$ to be the translation invariance set of $\bC$, namely
\begin{equation}
    S(\bC)\coloneqq  \{v\in \R^{n+1}:(\eta_{v,1})_\sharp\bC=\bC\},
\end{equation}
which is a vector subspace of $\R^{n+1}$.
\end{definition}
For $x\in \spt \|V\|$ and $\bC\in \VarTan(V,x)$ let
\begin{equation}\label{eq: doubled cone stratification}
	\hbC=\begin{cases}
		\bC &\text{if }x\in M\setminus \partial M, \\ 
		\bC+(\theta_{0})_\sharp\bC &\text{if }x\in\partial M.
	\end{cases}
\end{equation}

\begin{definition}[{$j$-stratum}]
For $j\in\{0, 1,\dots, n\}$, we define the \textit{$j$-stratum} $\mathscr{S}_j$ as follows:
\begin{equation}
    \mathscr{S}_j\coloneqq \{x\in\sing(V):\mathrm{dim}(S(\hbC))\leq j\text{ for all } \bC\in\VarTan(V,x)\}.
\end{equation}
\end{definition}

\begin{remark}
Clearly, 
\begin{equation}
    \mathscr{S}_0\subset\mathscr{S}_1\subset\dots\subset\mathscr{S}_{n}=\sing(V).
\end{equation}
\end{remark}

The usefulness of this stratification of the singular set is due to the following result, which was first proved by Almgren in the setting of area-minimising currents.

\begin{proposition}[{\cite[Proposition 3.30]{LiPariseSarnataro2026}}]\label{prop: stratification}
    Let $V$ be a free-boundary stationary integral varifold in $M$. Then, for all $j\in\{0,1,\dots, n\}$
    \begin{equation}
        \dim_\Hcal\mathscr{S}_j\leq j,
    \end{equation}
    where $\dim_\Hcal$ denotes the Hausdorff dimension.
\end{proposition}

\section{Non-emptiness of the boundary} \label{sec: non-empty boundary}
We expand on \Cref{rmk: non-empty boundary}. In particular, we first provide an example of a compact Riemannian manifold-with-boundary $(M, \mathrm g)$ where the one-parameter min-max construction of \Cref{thm: existence} yields a stationary integral varifold $V$ which is supported on an optimally regular minimal hypersurface entirely contained in the interior of $M$, then show that this cannot happen under the additional assumption that $\Ric_g\geq0$ on $M$ and $\partial M$ is strictly mean-convex.

\subsection{$\spt\Vert V\Vert\cap\partial M= \emptyset$}\label{subsec: example empty boundary} For sufficiently large $L>0$, let the manifold $M_L$ be the following truncated ellipsoid
    \[
        E_L=\left\{(x,z)\in\mathbb R^{n+1}\times\mathbb R:
        |x|^2+\frac{z^2}{L^2}=1\right\}\subset \mathbb{R}^{n+2},
        \quad
        M_L=E_L\cap\{z\leq L/2\},
    \]
endowed with the induced metric coming from the Euclidean metric on $\mathbb{R}^{n+2}$. The latitude sweepout $t \mapsto M_L\cap \{z=L/2-t\}$ for $t\in[0, 3L/2]$ implies that the first free boundary min-max width $W_1(M_L, \partial M_L)$ of $M_L$ is at most $\mathcal H^n(\mathbb S^n)$, independently of $L$. 

For $t\in[0, L/2)$, the level sets $M_L\cap\{z=L/2-t\}$ have mean curvature vector pointing towards $\partial M_L$. In particular, by White's varifold maximum principle \cite{WhiteMaxPrinciple}, if $\Sigma$ is a singular properly immersed free-boundary minimal hypersurface with $\overline{\Sigma}\cap \partial M_L\neq \emptyset$, then $\overline\Sigma$ must intersect every level set $M_L\cap\{z=t\}$ for $t\in[0, L/2]$. On the other hand, by applying interior monotonicity to $\Sigma \cap \{L/8\leq z\leq 3L/8\}$, we obtain $\mathcal{H}^n(\overline\Sigma)\geq c_nL$ whenever $\overline{\Sigma}\cap \partial M_L\neq \emptyset$, where $c_n>0$ is independent of $L$.  Consequently, for $L$ sufficiently large, every singular 
free-boundary minimal hypersurface realising $W_1(M_L, \partial M_L)$ must be disjoint from $\partial M_L$. 

\subsection{$\Ric_g\geq0$ on $M$ and strictly mean-convex boundary}\label{subsec: Ricci mean convex}
Let us then give a quick proof of the following simple fact.
\begin{lemma}
If $(M, \mathrm g)$ is a compact connected Riemannian manifold-with-boundary with $\partial M \neq \emptyset$ such that $\Ric_g\geq0$ and $\partial M$ is strictly mean-convex (namely $H_{\partial M}>0$), and the free boundary stationary integral varifold $V$ in $M$ is induced by a non-empty singular properly immersed free-boundary minimal hypersurface $\Sigma$, then $\partial \Sigma \neq \emptyset$.    
\end{lemma}

\begin{proof}
We shall argue by contradiction. Notice that $\overline{\Sigma}=\spt \|V\|$.
By compactness of $M$, there is $c>0$ such that $H_{\partial M}\geq c$, so that, by \cite[Lemmas 2.1--2.2]{ChoeFraser}, the distance function $d_{\partial M}$ from $\partial M$ satisfies $\Delta_\mathrm{g}d_{\partial M}\leq -c$ in the barrier sense\footnote{This is stated clearly in Lemma 2.1 of \cite{ChoeFraser} for the region where $d_{\partial M}$ is smooth, and it follows from the proof of Lemma 2.2 (in the barrier sense) outside this region. This statement also appears (without proof) as Proposition 5 in \cite{XWang21}. A similar barrier construction appears in \cite[Proof of Theorem 1.1]{Perales}.}. Now suppose $x\in \overline{\Sigma}$ is such that $d_{\partial M}(x)=\dist_\mathrm{g}(\overline{\Sigma}, \partial M)>0$. For $\varepsilon\in(0, c/2)$, the construction in the proof of  \cite[Lemma 2.2]{ChoeFraser}, applied to the strictly mean-convex hypersurface $\partial M$, gives a neighbourhood $U$ of $x$ in $M$ and a smooth upper barrier function $u:U\to \R$ for $d_{\partial M}$ such that $|\nabla^\mathrm{g} u| = 1$ on $U$ and $\Delta_\mathrm gu(x)\leq -c+\varepsilon\leq -c/2<0$.\\
For $y\in U\cap \overline\Sigma$, we have
    \[
        u(x)= d_{\partial M}(x)\leq d_{\partial M}(y)\leq u(y)
    \]
since $u$ is an upper barrier for $d_{\partial M}$ and $x$ minimises distance to $\partial M$. Therefore, $\overline{\Sigma}\cap U$ lies on one side of $S=\{u=u(x)\}$ and it touches $S$ at $x$. Since $|\nabla_\mathrm{g}u|=1$, it is easy to check that $H_S(x)=-\Delta_\mathrm{g}u(x)\geq c/2>0$, and the domain $\{u\geq u(x)\}\supset \overline{\Sigma}\cap U$ is strongly $n$-convex at $x$. The contradiction then follows from White's varifold maximum principle \cite{WhiteMaxPrinciple}.
\end{proof}

\bibliographystyle{amsalpha}
\bibliography{bibliography}

@article {Li-Zhou-max,
    AUTHOR = {Li, Martin Man-chun and Zhou, Xin},
     TITLE = {A maximum principle for free boundary minimal varieties of
              arbitrary codimension},
   JOURNAL = {Comm. Anal. Geom.},
  FJOURNAL = {Communications in Analysis and Geometry},
    VOLUME = {29},
      YEAR = {2021},
    NUMBER = {6},
     PAGES = {1509--1521},
      ISSN = {1019-8385,1944-9992},
   MRCLASS = {53C42 (49Q15)},
  MRNUMBER = {4367433},
MRREVIEWER = {Fei-Tsen\ Liang},
       DOI = {10.4310/CAG.2021.v29.n6.a7},
       URL = {https://doi-org.easyaccess1.lib.cuhk.edu.hk/10.4310/CAG.2021.v29.n6.a7},
}

@article {Bellettini,
    AUTHOR = {Bellettini, Costante},
     TITLE = {Extensions of {S}choen-{S}imon-{Y}au and {S}choen-{S}imon
              theorems via iteration \`a{} la {D}e {G}iorgi},
   JOURNAL = {Invent. Math.},
  FJOURNAL = {Inventiones Mathematicae},
    VOLUME = {240},
      YEAR = {2025},
    NUMBER = {1},
     PAGES = {1--34},
      ISSN = {0020-9910,1432-1297},
   MRCLASS = {49Q20 (58E12)},
  MRNUMBER = {4871955},
MRREVIEWER = {Nicolau\ S.\ Aiex},
       DOI = {10.1007/s00222-025-01317-0},
       URL = {https://doi.org/10.1007/s00222-025-01317-0},
}

@article {WicM2,
    AUTHOR = {Wickramasekera, Neshan},
     TITLE = {A regularity and compactness theory for immersed stable
              minimal hypersurfaces of multiplicity at most 2},
   JOURNAL = {J. Differential Geom.},
  FJOURNAL = {Journal of Differential Geometry},
    VOLUME = {80},
      YEAR = {2008},
    NUMBER = {1},
     PAGES = {79--173},
      ISSN = {0022-040X,1945-743X},
   MRCLASS = {53C42 (49Q15)},
  MRNUMBER = {2434260},
MRREVIEWER = {Fei-Tsen\ Liang},
       URL = {http://projecteuclid.org/euclid.jdg/1217361067},
}

@book {Gray,
    AUTHOR = {Gray, Alfred},
     TITLE = {Tubes},
    SERIES = {Progress in Mathematics},
    VOLUME = {221},
   EDITION = {Second},
      NOTE = {With a preface by Vicente Miquel},
 PUBLISHER = {Birkh\"auser Verlag, Basel},
      YEAR = {2004},
     PAGES = {xiv+280},
      ISBN = {3-7643-6907-8},
   MRCLASS = {53-02 (53-01 53B25 53C40)},
  MRNUMBER = {2024928},
       DOI = {10.1007/978-3-0348-7966-8},
       URL = {https://doi.org/10.1007/978-3-0348-7966-8},
}

@book {LeeSmooth,
    AUTHOR = {Lee, John M.},
     TITLE = {Introduction to smooth manifolds},
    SERIES = {Graduate Texts in Mathematics},
    VOLUME = {218},
   EDITION = {Second},
 PUBLISHER = {Springer, New York},
      YEAR = {2013},
     PAGES = {xvi+708},
      ISBN = {978-1-4419-9981-8},
   MRCLASS = {58-01 (53-01 57-01)},
  MRNUMBER = {2954043},
}

@article {LiPariseSarnataro2023,
    AUTHOR = {Li, Martin Man-chun and Parise, Davide and Sarnataro, Lorenzo},
     TITLE = {Boundary behavior of limit-interfaces for the {A}llen-{C}ahn
              equation on {R}iemannian manifolds with {N}eumann boundary
              condition},
   JOURNAL = {Arch. Ration. Mech. Anal.},
  FJOURNAL = {Archive for Rational Mechanics and Analysis},
    VOLUME = {248},
      YEAR = {2024},
    NUMBER = {6},
     PAGES = {Paper No. 124, 43},
      ISSN = {0003-9527,1432-0673},
   MRCLASS = {35R01 (35J25 35J67 49Q15)},
  MRNUMBER = {4831349},
MRREVIEWER = {Haitao\ Wan},
       DOI = {10.1007/s00205-024-02070-z},
       URL = {https://doi.org/10.1007/s00205-024-02070-z},
}

@article{Masi2020,
	author = {De Masi, Luigi},
	doi = {10.1512/iumj.2021.70.9401},
	fjournal = {Indiana University Mathematics Journal},
	issn = {0022-2518},
	journal = {Indiana Univ. Math. J.},
	mrclass = {49Q15 (53C42)},
	mrnumber = {4359920},
	mrreviewer = {Mario Santilli},
	number = {6},
	pages = {2603--2651},
	title = {Rectifiability of the free boundary for varifolds},
	url = {https://doi.org/10.1512/iumj.2021.70.9401},
	volume = {70},
	year = {2021}}

@article {SimonWicFrequency,
    AUTHOR = {Simon, Leon and Wickramasekera, Neshan},
     TITLE = {A frequency function and singular set bounds for branched
              minimal immersions},
   JOURNAL = {Comm. Pure Appl. Math.},
  FJOURNAL = {Communications on Pure and Applied Mathematics},
    VOLUME = {69},
      YEAR = {2016},
    NUMBER = {7},
     PAGES = {1213--1258},
      ISSN = {0010-3640,1097-0312},
   MRCLASS = {58E12 (35B65 53C42)},
  MRNUMBER = {3503021},
MRREVIEWER = {Futoshi\ Takahashi},
       DOI = {10.1002/cpa.21642},
       URL = {https://doi.org/10.1002/cpa.21642},
}

@article {GuangWangZhouFiniteness,
    AUTHOR = {Guang, Qiang and Wang, Zhichao and Zhou, Xin},
     TITLE = {Compactness and generic finiteness for free boundary minimal
              hypersurfaces, {I}},
   JOURNAL = {Pacific J. Math.},
  FJOURNAL = {Pacific Journal of Mathematics},
    VOLUME = {310},
      YEAR = {2021},
    NUMBER = {1},
     PAGES = {85--114},
      ISSN = {0030-8730,1945-5844},
   MRCLASS = {53A10 (53C42)},
  MRNUMBER = {4229234},
MRREVIEWER = {Changwei\ Xiong},
       DOI = {10.2140/pjm.2021.310.85},
       URL = {https://doi.org/10.2140/pjm.2021.310.85},
}

@article{Guang2020,
	author = {Guang, Qiang and Li, Martin Man-chun and Zhou, Xin},
	doi = {10.1515/crelle-2018-0008},
	issn = {0075-4102},
	journal = {Journal f\"{u}r die Reine und Angewandte Mathematik. [Crelle's Journal]},
	mrnumber = {4058180},
	pages = {245--264},
	title = {Curvature estimates for stable free boundary minimal hypersurfaces},
	volume = {759},
	year = {2020}}

@article{Li2021,
	author = {Li, Martin Man-chun and Zhou, Xin},
	doi = {10.4310/jdg/1625860624},
	issn = {0022-040X},
	journal = {Journal of Differential Geometry},
	mrnumber = {4285846},
	number = {3},
	pages = {487--553},
	title = {Min-max theory for free boundary minimal hypersurfaces {I}---{R}egularity theory},
	volume = {118},
	year = {2021}}

@article{Allard1972,
	author = {Allard, William K.},
	doi = {10.2307/1970868},
	issn = {0003-486X},
	journal = {Annals of Mathematics. Second Series},
	mrnumber = {307015},
	pages = {417--491},
	title = {On the first variation of a varifold},
	volume = {95},
	year = {1972}}

@book{Evans2015,
	author = {Evans, Lawrence C. and Gariepy, Ronald F.},
	edition = {Revised},
	isbn = {978-1-4822-4238-6},
	mrnumber = {3409135},
	pages = {xiv+299},
	publisher = {CRC Press, Boca Raton, FL},
	series = {Textbooks in Mathematics},
	title = {Measure theory and fine properties of functions},
	year = {2015}}

@article{Grueter1986,
	author = {Gr\"{u}ter, Michael and Jost, J\"{u}rgen},
	issn = {0391-173X},
	journal = {Annali della Scuola Normale Superiore di Pisa. Classe di Scienze. Serie IV},
	mrnumber = {863638},
	number = {1},
	pages = {129--169},
	title = {Allard type regularity results for varifolds with free boundaries},
	volume = {13},
	year = {1986}}

@article{Guaraco2018,
	author = {Guaraco, Marco A. M.},
	doi = {10.4310/jdg/1513998031},
	issn = {0022-040X},
	journal = {Journal of Differential Geometry},
	mrnumber = {3743704},
	number = {1},
	pages = {91--133},
	title = {Min-max for phase transitions and the existence of embedded minimal hypersurfaces},
	volume = {108},
	year = {2018}}

@article{Hutchinson2000,
	author = {Hutchinson, John E. and Tonegawa, Yoshihiro},
	doi = {10.1007/PL00013453},
	issn = {0944-2669},
	journal = {Calculus of Variations and Partial Differential Equations},
	mrnumber = {1803974},
	number = {1},
	pages = {49--84},
	title = {Convergence of phase interfaces in the van der {W}aals-{C}ahn-{H}illiard theory},
	volume = {10},
	year = {2000}}

@article {SimonWicExa,
    AUTHOR = {Simon, Leon and Wickramasekera, Neshan},
     TITLE = {Stable branched minimal immersions with prescribed boundary},
   JOURNAL = {J. Differential Geom.},
  FJOURNAL = {Journal of Differential Geometry},
    VOLUME = {75},
      YEAR = {2007},
    NUMBER = {1},
     PAGES = {143--173},
      ISSN = {0022-040X,1945-743X},
   MRCLASS = {53C42 (49Q05 53A07)},
  MRNUMBER = {2282727},
MRREVIEWER = {Giandomenico\ Orlandi},
       URL = {http://projecteuclid.org/euclid.jdg/1175266256},
}

@article{Tonegawa2012,
	author = {Tonegawa, Yoshihiro and Wickramasekera, Neshan},
	doi = {10.1515/crelle.2011.134},
	issn = {0075-4102},
	journal = {Journal f\"{u}r die Reine und Angewandte Mathematik. [Crelle's Journal]},
	mrnumber = {2948876},
	pages = {191--210},
	title = {Stable phase interfaces in the van der {W}aals--{C}ahn--{H}illiard theory},
	volume = {668},
	year = {2012}}

@book{Simon1983,
	author = {Simon, Leon},
	isbn = {0-86784-429-9},
	mrnumber = {756417},
	pages = {vii+272},
	publisher = {Australian National University, Centre for Mathematical Analysis, Canberra},
	series = {Proceedings of the Centre for Mathematical Analysis, Australian National University},
	title = {Lectures on geometric measure theory},
	volume = {3},
	year = {1983}}

@article{Edelen2020,
	author = {Edelen, Nick},
	doi = {10.1515/crelle-2017-0053},
	issn = {0075-4102},
	journal = {Journal f\"{u}r die Reine und Angewandte Mathematik. [Crelle's Journal]},
	mrnumber = {4048443},
	pages = {95--137},
	title = {The free-boundary {B}rakke flow},
	volume = {758},
	year = {2020}}

@book{Colding2011,
	author = {Colding, Tobias Holck and Minicozzi, II, William P.},
	doi = {10.1090/gsm/121},
	isbn = {978-0-8218-5323-8},
	mrnumber = {2780140},
	pages = {xii+313},
	publisher = {American Mathematical Society, Providence, RI},
	series = {Graduate Studies in Mathematics},
	title = {A course in minimal surfaces},
	volume = {121},
	year = {2011}}

@article {GruterJostIHP86,
    author = {Gr\"{u}ter, Michael and Jost, J\"{u}rgen},
     TITLE = {On embedded minimal disks in convex bodies},
   JOURNAL = {Ann. Inst. H. Poincar\'e{} Anal. Non Lin\'eaire},
  FJOURNAL = {Annales de l'Institut Henri Poincar\'e. Analyse Non
              Lin\'eaire},
    VOLUME = {3},
      YEAR = {1986},
    NUMBER = {5},
     PAGES = {345--390},
      ISSN = {0294-1449},
   MRCLASS = {49F10 (49F20 53A10 58E12)},
  MRNUMBER = {868522},
MRREVIEWER = {Harold\ Parks},
       URL = {http://www.numdam.org/item?id=AIHPC_1986__3_5_345_0},
}

@article {SchoenSimon81,
    AUTHOR = {Schoen, Richard and Simon, Leon},
     TITLE = {Regularity of stable minimal hypersurfaces},
   JOURNAL = {Comm. Pure Appl. Math.},
  FJOURNAL = {Communications on Pure and Applied Mathematics},
    VOLUME = {34},
      YEAR = {1981},
    NUMBER = {6},
     PAGES = {741--797},
      ISSN = {0010-3640,1097-0312},
   MRCLASS = {49F22 (53C42 58E15)},
  MRNUMBER = {634285},
MRREVIEWER = {F.\ J.\ Almgren, Jr.},
       DOI = {10.1002/cpa.3160340603},
       URL = {https://doi.org/10.1002/cpa.3160340603},
}

@article {Edelen2016,
    AUTHOR = {Edelen, Nick},
     TITLE = {Convexity estimates for mean curvature flow with free
              boundary},
   JOURNAL = {Adv. Math.},
  FJOURNAL = {Advances in Mathematics},
    VOLUME = {294},
      YEAR = {2016},
     PAGES = {1--36},
      ISSN = {0001-8708,1090-2082},
   MRCLASS = {53C44 (53C42)},
  MRNUMBER = {3479560},
MRREVIEWER = {En-Tao\ Zhao},
       DOI = {10.1016/j.aim.2016.02.026},
       URL = {https://doi.org/10.1016/j.aim.2016.02.026},
}

@article {Wick14,
    AUTHOR = {Wickramasekera, Neshan},
     TITLE = {A general regularity theory for stable codimension 1 integral
              varifolds},
   JOURNAL = {Ann. of Math. (2)},
  FJOURNAL = {Annals of Mathematics. Second Series},
    VOLUME = {179},
      YEAR = {2014},
    NUMBER = {3},
     PAGES = {843--1007},
      ISSN = {0003-486X,1939-8980},
   MRCLASS = {49Q20 (49N60 58E05 58K99)},
  MRNUMBER = {3171756},
MRREVIEWER = {S\l awomir\ Kolasi\'nski},
       DOI = {10.4007/annals.2014.179.3.2},
       URL = {https://doi.org/10.4007/annals.2014.179.3.2},
}

@article {SimonRemovable77,
    AUTHOR = {Simon, Leon},
     TITLE = {On a theorem of de {G}iorgi and {S}tampacchia},
   JOURNAL = {Math. Z.},
  FJOURNAL = {Mathematische Zeitschrift},
    VOLUME = {155},
      YEAR = {1977},
    NUMBER = {2},
     PAGES = {199--204},
      ISSN = {0025-5874,1432-1823},
   MRCLASS = {49F22 (35J20)},
  MRNUMBER = {454857},
MRREVIEWER = {U.\ D'Ambrosio},
       DOI = {10.1007/BF01214218},
       URL = {https://doi.org/10.1007/BF01214218},
}

@book {Petersen,
    AUTHOR = {Petersen, Peter},
     TITLE = {Riemannian geometry},
    SERIES = {Graduate Texts in Mathematics},
    VOLUME = {171},
   EDITION = {Third},
 PUBLISHER = {Springer, Cham},
      YEAR = {2016},
     PAGES = {xviii+499},
      ISBN = {978-3-319-26652-7; 978-3-319-26654-1},
   MRCLASS = {53-01 (53C20 53C21 53C23)},
  MRNUMBER = {3469435},
       DOI = {10.1007/978-3-319-26654-1},
       URL = {https://doi.org/10.1007/978-3-319-26654-1},
}

@book {LeeRiem,
    AUTHOR = {Lee, John M.},
     TITLE = {Introduction to {R}iemannian manifolds},
    SERIES = {Graduate Texts in Mathematics},
    VOLUME = {176},
   EDITION = {Second},
 PUBLISHER = {Springer, Cham},
      YEAR = {2018},
     PAGES = {xiii+437},
      ISBN = {978-3-319-91754-2; 978-3-319-91755-9},
   MRCLASS = {53-01 (53B20 53B30 53C20 53C21)},
  MRNUMBER = {3887684},
MRREVIEWER = {Robert\ J.\ Low},
}

@article {LiebTru86,
    AUTHOR = {Lieberman, Gary M. and Trudinger, Neil S.},
     TITLE = {Nonlinear oblique boundary value problems for nonlinear
              elliptic equations},
   JOURNAL = {Trans. Amer. Math. Soc.},
  FJOURNAL = {Transactions of the American Mathematical Society},
    VOLUME = {295},
      YEAR = {1986},
    NUMBER = {2},
     PAGES = {509--546},
      ISSN = {0002-9947,1088-6850},
   MRCLASS = {35J65},
  MRNUMBER = {833695},
MRREVIEWER = {Pierre-Louis\ Lions},
       DOI = {10.2307/2000050},
       URL = {https://doi.org/10.2307/2000050},
}

@article {Lieberman88,
	AUTHOR = {Lieberman, Gary M.},
	TITLE = {Boundary regularity for solutions of degenerate elliptic
	equations},
	JOURNAL = {Nonlinear Anal.},
	FJOURNAL = {Nonlinear Analysis. Theory, Methods \& Applications. An
	International Multidisciplinary Journal},
	VOLUME = {12},
	YEAR = {1988},
	NUMBER = {11},
	PAGES = {1203--1219},
	ISSN = {0362-546X,1873-5215},
	MRCLASS = {35J70 (35B65)},
	MRNUMBER = {969499},
	MRREVIEWER = {Zuchi\ Chen},
	DOI = {10.1016/0362-546X(88)90053-3},
	URL = {https://doi.org/10.1016/0362-546X(88)90053-3},
}

@book {Pitts81,
    AUTHOR = {Pitts, Jon T.},
     TITLE = {Existence and regularity of minimal surfaces on {R}iemannian
              manifolds},
    SERIES = {Mathematical Notes},
    VOLUME = {27},
 PUBLISHER = {Princeton University Press, Princeton, NJ; University of Tokyo
              Press, Tokyo},
      YEAR = {1981},
     PAGES = {iv+330},
      ISBN = {0-691-08290-1},
   MRCLASS = {49F22 (53C42)},
  MRNUMBER = {626027},
MRREVIEWER = {J.\ E.\ Brothers},
}

@article{HarveyLawson2014,
  author        = {Reese F. Harvey and H. Blaine Lawson{, Jr}},
  title         = {Removable Singularities for Nonlinear Subequations},
  journal       = {Indiana University Mathematics Journal},
  volume        = {63},
  number        = {5},
  pages         = {1525--1552},
  year          = {2014},
  doi           = {10.1512/iumj.2014.63.5398},
  eprint        = {1303.0437},
  archivePrefix = {arXiv},
  primaryClass  = {math.AP},
  mrnumber      = {3283560}
}

@incollection {SimonIsolated,
    AUTHOR = {Simon, Leon},
     TITLE = {On isolated singularities of minimal surfaces},
 BOOKTITLE = {Miniconference on partial differential equations ({C}anberra,
              1981)},
    SERIES = {Proc. Centre Math. Anal. Austral. Nat. Univ.},
    VOLUME = {1},
     PAGES = {70--100},
 PUBLISHER = {Austral. Nat. Univ., Canberra},
      YEAR = {1982},
      ISBN = {0-86784-123-0},
   MRCLASS = {49F22 (49F10 53A10 58E12)},
  MRNUMBER = {758452},
MRREVIEWER = {Anna\ Salvadori},
}

@book {Falconer,
    AUTHOR = {Falconer, Kenneth},
     TITLE = {Fractal geometry},
   EDITION = {Second},
      NOTE = {Mathematical foundations and applications},
 PUBLISHER = {John Wiley \& Sons, Inc., Hoboken, NJ},
      YEAR = {2003},
     PAGES = {xxviii+337},
      ISBN = {0-470-84861-8},
   MRCLASS = {28-01 (00A69 11K55 28A75 28A78 28A80 37C45 37F10)},
  MRNUMBER = {2118797},
MRREVIEWER = {Esa\ J\"arvenp\"a\"a},
       DOI = {10.1002/0470013850},
       URL = {https://doi.org/10.1002/0470013850},
}

@book {AdamsPotential,
    AUTHOR = {Adams, David R. and Hedberg, Lars Inge},
     TITLE = {Function spaces and potential theory},
    SERIES = {Grundlehren der mathematischen Wissenschaften [Fundamental
              Principles of Mathematical Sciences]},
    VOLUME = {314},
 PUBLISHER = {Springer-Verlag, Berlin},
      YEAR = {1996},
     PAGES = {xii+366},
      ISBN = {3-540-57060-8},
   MRCLASS = {46E35 (31B99 46N20)},
  MRNUMBER = {1411441},
MRREVIEWER = {Takahide\ Kurokawa},
       DOI = {10.1007/978-3-662-03282-4},
       URL = {https://doi.org/10.1007/978-3-662-03282-4},
}

@article {SSY75,
    AUTHOR = {Schoen, Richard and Simon, Leon and Yau, Shing-Tung},
     TITLE = {Curvature estimates for minimal hypersurfaces},
   JOURNAL = {Acta Math.},
  FJOURNAL = {Acta Mathematica},
    VOLUME = {134},
      YEAR = {1975},
    NUMBER = {3-4},
     PAGES = {275--288},
      ISSN = {0001-5962,1871-2509},
   MRCLASS = {53C40 (49F10)},
  MRNUMBER = {423263},
MRREVIEWER = {Robert\ Gulliver},
       DOI = {10.1007/BF02392104},
       URL = {https://doi.org/10.1007/BF02392104},
}

@article{LiPariseSarnataro2026,
    AUTHOR = {Li, Martin Man-chun and Parise, Davide and Sarnataro, Lorenzo},
    TITLE = {Boundary behaviour of stable {A}llen--{C}ahn limit-interfaces in {R}iemannian manifolds},
    year={2026},
    journal={arXiv},
    pages={2610.04833},
    eprint={2610.04833},
    archivePrefix={arXiv},
    primaryClass={math.DG}

}

@article {DeLellisRamic,
    AUTHOR = {De Lellis, Camillo and Ramic, Jusuf},
     TITLE = {Min-max theory for minimal hypersurfaces with boundary},
   JOURNAL = {Ann. Inst. Fourier (Grenoble)},
  FJOURNAL = {Universit\'e{} de Grenoble. Annales de l'Institut Fourier},
    VOLUME = {68},
      YEAR = {2018},
    NUMBER = {5},
     PAGES = {1909--1986},
      ISSN = {0373-0956,1777-5310},
   MRCLASS = {53C42 (49Q05 53A10)},
  MRNUMBER = {3893761},
MRREVIEWER = {Annalisa\ Cesaroni},
       DOI = {10.5802/aif.3200},
       URL = {https://doi.org/10.5802/aif.3200},
}

@article{GWangcapillary,
      title={Regularity of stable capillary minimal hypersurfaces}, 
      author={Wang, Gaoming and Zhang, Xuwen},
      year={2026},
      journal={arXiv},
      pages={2605.20964},
      eprint={2605.20964},
      archivePrefix={arXiv},
      primaryClass={math.DG}
}

@article{GWangimmersed,
      title={Regularity of branched stable minimal immersed hypersurfaces}, 
      author={Wang, Gaoming and Zhang, Xuwen},
      year={2026},
      journal={arXiv},
      pages={2607.29632},
      eprint={2607.29632},
      archivePrefix={arXiv},
      primaryClass={math.DG}
}

@article {ChoeFraser,
    AUTHOR = {Choe, Jaigyoung and Fraser, Ailana},
     TITLE = {Mean curvature in manifolds with {R}icci curvature bounded
              from below},
   JOURNAL = {Comment. Math. Helv.},
  FJOURNAL = {Commentarii Mathematici Helvetici. A Journal of the Swiss
              Mathematical Society},
    VOLUME = {93},
      YEAR = {2018},
    NUMBER = {1},
     PAGES = {55--69},
      ISSN = {0010-2571,1420-8946},
   MRCLASS = {53C42 (53C20)},
  MRNUMBER = {3777125},
MRREVIEWER = {Theodoros\ Vlachos},
       DOI = {10.4171/CMH/429},
       URL = {https://doi.org/10.4171/CMH/429},
}

@article {Perales,
    AUTHOR = {Perales, Raquel},
     TITLE = {Volumes and limits of manifolds with {R}icci curvature and
              mean curvature bounds},
   JOURNAL = {Differential Geom. Appl.},
  FJOURNAL = {Differential Geometry and its Applications},
    VOLUME = {48},
      YEAR = {2016},
     PAGES = {23--37},
      ISSN = {0926-2245,1872-6984},
   MRCLASS = {53C20 (53C21)},
  MRNUMBER = {3534433},
MRREVIEWER = {Wen-Haw\ Chen},
       DOI = {10.1016/j.difgeo.2016.05.004},
       URL = {https://doi.org/10.1016/j.difgeo.2016.05.004},
}

@article {XWang21,
    AUTHOR = {Wang, Xiaodong},
     TITLE = {On compact {R}iemannian manifolds with convex boundary and
              {R}icci curvature bounded from below},
   JOURNAL = {J. Geom. Anal.},
  FJOURNAL = {Journal of Geometric Analysis},
    VOLUME = {31},
      YEAR = {2021},
    NUMBER = {4},
     PAGES = {3988--4003},
      ISSN = {1050-6926,1559-002X},
   MRCLASS = {53C21 (58J90)},
  MRNUMBER = {4236549},
MRREVIEWER = {Stefano\ Pigola},
       DOI = {10.1007/s12220-020-00422-1},
       URL = {https://doi.org/10.1007/s12220-020-00422-1},
}

@article {WhiteMaxPrinciple,
    AUTHOR = {White, Brian},
     TITLE = {The maximum principle for minimal varieties of arbitrary
              codimension},
   JOURNAL = {Comm. Anal. Geom.},
  FJOURNAL = {Communications in Analysis and Geometry},
    VOLUME = {18},
      YEAR = {2010},
    NUMBER = {3},
     PAGES = {421--432},
      ISSN = {1019-8385,1944-9992},
   MRCLASS = {53C42 (49Q20)},
  MRNUMBER = {2747434},
MRREVIEWER = {Fei-Tsen\ Liang},
       DOI = {10.4310/CAG.2010.v18.n3.a1},
       URL = {https://doi.org/10.4310/CAG.2010.v18.n3.a1},
}

@article{MenneScharrer,
      title={A priori bounds for geodesic diameter. {P}art II. 
{F}ine connectedness properties of varifolds}, 
      author={Menne, Ulrich and Scharrer, Christian},
      year={2022},
      journal={arXiv},
      pages={2209.05955},
      eprint={2209.05955},
      archivePrefix={arXiv},
      primaryClass={math.DG}
}

@article {stableBernstein4,
    AUTHOR = {Chodosh, Otis and Li, Chao},
     TITLE = {Stable minimal hypersurfaces in {${\bf R}^4$}},
   JOURNAL = {Acta Math.},
  FJOURNAL = {Acta Mathematica},
    VOLUME = {233},
      YEAR = {2024},
    NUMBER = {1},
     PAGES = {1--31},
      ISSN = {0001-5962,1871-2509},
   MRCLASS = {53A10 (53C42)},
  MRNUMBER = {4816633},
MRREVIEWER = {Jinyu\ Guo},
       DOI = {10.4310/acta.2024.v233.n1.a1},
       URL = {https://doi.org/10.4310/acta.2024.v233.n1.a1},
}

@article {stableBernstein5,
    AUTHOR = {Chodosh, Otis and Li, Chao and Minter, Paul and Stryker,
              Douglas},
     TITLE = {Stable minimal hypersurfaces in {$\bold R^5$}},
   JOURNAL = {Ann. of Math. (2)},
  FJOURNAL = {Annals of Mathematics. Second Series},
    VOLUME = {204},
      YEAR = {2026},
    NUMBER = {1},
     PAGES = {423--453},
      ISSN = {0003-486X,1939-8980},
   MRCLASS = {53C42 (35J50 53A07 53A10)},
  MRNUMBER = {5094441},
       DOI = {10.4007/annals.2026.204.1.5},
       URL = {https://doi.org/10.4007/annals.2026.204.1.5},
}

@article{stableBernstein6,
       title={Stable minimal hypersurfaces in $\mathbb R^6$}, 
      author={Laurent Mazet},
      year={2024},
      journal={arXiv},
      eprint={2405.14676},
      archivePrefix={arXiv},
      primaryClass={math.DG},
      url={https://arxiv.org/abs/2405.14676}, 
}

@article{stableBernstein7,
       title={The stable Bernstein theorem in $\mathbb{R}^{7}$}, 
      author={Han Hong and Haizhong Li and Gaoming Wang},
      year={2026},
      journal={arXiv},
      eprint={2609.15720},
      archivePrefix={arXiv},
      primaryClass={math.DG},
      url={https://arxiv.org/abs/2609.15720}, 
}

@article {Bellettini22,
    AUTHOR = {Bellettini, Costante},
     TITLE = {Generic existence of multiplicity-1 minmax minimal
              hypersurfaces via {A}llen-{C}ahn},
   JOURNAL = {Calc. Var. Partial Differential Equations},
  FJOURNAL = {Calculus of Variations and Partial Differential Equations},
    VOLUME = {61},
      YEAR = {2022},
    NUMBER = {4},
     PAGES = {Paper No. 149, 25},
      ISSN = {0944-2669,1432-0835},
   MRCLASS = {53A10 (35J20 35J61 35K58 49Q05 49Q20 53C21 58J60)},
  MRNUMBER = {4438340},
MRREVIEWER = {Serena\ Dipierro},
       DOI = {10.1007/s00526-022-02261-0},
       URL = {https://doi.org/10.1007/s00526-022-02261-0},
}

@article {BellettiniJiang26,
    AUTHOR = {Bellettini, Costante and Jiang, Yuze},
     TITLE = {Stable hypersurfaces with mean curvature prescribed by a given ambient {S}obolev function},
   JOURNAL = {in preparation},
     PAGES = {},
      YEAR = {2026},

}

\end{document}